%% file: CornerExtensions4.tex
\documentclass[a4paper]{article}

\usepackage{
amsmath,
amsthm,
amscd,
amssymb,
wasysym
}
\usepackage{comment}
\usepackage{tikz}
\usetikzlibrary{positioning,arrows,calc}
\usepackage{xspace}
\usepackage{subcaption}
\usepackage{verbatim}
\usepackage{verbatimbox}

\usepackage{pgfplotstable}

\usepackage{graphicx}

\usepackage{mathtools}
\usepackage{url}

\theoremstyle{plain}
\newtheorem{theorem}{Theorem}
\newtheorem{lemma}[theorem]{Lemma}
\newtheorem{definition}[theorem]{Definition}
\newtheorem{corollary}[theorem]{Corollary}
\newtheorem{proposition}[theorem]{Proposition}

\newtheorem{question}[theorem]{Question}

\newtheoremstyle{derp}
{3pt}
{3pt}
{}
{}
{\upshape}
{:}
{.5em}
{}
\theoremstyle{derp}
\newtheorem{example}{Example}

\newcommand{\R}{\mathbb{R}}
\newcommand{\Q}{\mathbb{Q}}
\newcommand{\Z}{\mathbb{Z}}
\newcommand{\C}{\mathcal{C}}
\newcommand{\D}{\mathcal{D}}
\newcommand{\T}{\mathcal{T}}
\newcommand{\N}{\mathbb{N}}

\newcommand\xqed[1]{%
  \leavevmode\unskip\penalty9999 \hbox{}\nobreak\hfill
  \quad\hbox{#1}}
\newcommand\qee{\xqed{$\fullmoon$}}

\makeatletter
\def\namedlabel#1#2{\begingroup
   \def\@currentlabel{#2}%
   \label{#1}\endgroup
}
\makeatother

\newcommand{\geod}{\mathrm{geod}}

\newcommand{\FinSet}{\mathrm{FinSet}}
\newcommand{\supp}{\mathrm{supp}}

\newcommand{\uu}{\vec{u}}
\newcommand{\vv}{\vec{v}}

\makeatletter
\tikzset{
    zero color/.initial=white,
    zero color/.get=\zerocol,
    zero color/.store in=\zerocol,
    one color/.initial=red,
    one color/.get=\onecol,
    one color/.store in=\onecol,
    two color/.initial=blue,
    two color/.get=\twocol,
    two color/.store in=\twocol,
    cell wd/.initial=1,
    cell wd/.get=\cellwd,
    cell wd/.store in=\cellwd,
    cell ht/.initial=1,
    cell ht/.get=\cellht,
    cell ht/.store in=\cellht,
}

\makeatletter
\DeclareFontFamily{OMX}{MnSymbolE}{}
\DeclareSymbolFont{MnLargeSymbols}{OMX}{MnSymbolE}{m}{n}
\SetSymbolFont{MnLargeSymbols}{bold}{OMX}{MnSymbolE}{b}{n}
\DeclareFontShape{OMX}{MnSymbolE}{m}{n}{
    <-6>  MnSymbolE5
   <6-7>  MnSymbolE6
   <7-8>  MnSymbolE7
   <8-9>  MnSymbolE8
   <9-10> MnSymbolE9
  <10-12> MnSymbolE10
  <12->   MnSymbolE12
}{}
\DeclareFontShape{OMX}{MnSymbolE}{b}{n}{
    <-6>  MnSymbolE-Bold5
   <6-7>  MnSymbolE-Bold6
   <7-8>  MnSymbolE-Bold7
   <8-9>  MnSymbolE-Bold8
   <9-10> MnSymbolE-Bold9
  <10-12> MnSymbolE-Bold10
  <12->   MnSymbolE-Bold12
}{}

\let\llangle\@undefined
\let\rrangle\@undefined
\DeclareMathDelimiter{\llangle}{\mathopen}%
                     {MnLargeSymbols}{'164}{MnLargeSymbols}{'164}
\DeclareMathDelimiter{\rrangle}{\mathclose}%
                     {MnLargeSymbols}{'171}{MnLargeSymbols}{'171}
\makeatother

\DeclareMathOperator{\bs}{\setminus}
\DeclareMathOperator{\di}{/}

\newcommand{\drawitman}{
\begin{tikzpicture}[scale = 0.2]

  \pgfplotstablegetrowsof{\matrixfile} 
  \pgfmathtruncatemacro{\totrow}{\pgfplotsretval}
  \pgfplotstablegetcolsof{\matrixfile} 
  \pgfmathtruncatemacro{\totcol}{\pgfplotsretval}
  
  \pgfplotstableforeachcolumn\matrixfile\as\col{
    \pgfplotstableforeachcolumnelement{\col}\of\matrixfile\as\colcnt{%
      \ifnum\colcnt=0
        \fill[white]($ -\pgfplotstablerow*(0,\cellht) + \col*(\cellwd,0) $) rectangle +(\cellwd,\cellht);
      \fi
      \ifnum\colcnt=1
        \fill[black!50!white]($ -\pgfplotstablerow*(0,\cellht) + \col*(\cellwd,0) $) rectangle+(\cellwd,\cellht);
      \fi
      \ifnum\colcnt=2
        \fill[red]($ -\pgfplotstablerow*(0,\cellht) + \col*(\cellwd,0) $) rectangle+(\cellwd,\cellht);
      \fi
      \ifnum\colcnt=3
        \fill[green]($ -\pgfplotstablerow*(0,\cellht) + \col*(\cellwd,0) $) rectangle+(\cellwd,\cellht);
      \fi
      \ifnum\colcnt=4
        \fill[blue]($ -\pgfplotstablerow*(0,\cellht) + \col*(\cellwd,0) $) rectangle+(\cellwd,\cellht);
      \fi
      \ifnum\colcnt=5
        \fill[yellow]($ -\pgfplotstablerow*(0,\cellht) + \col*(\cellwd,0) $) rectangle+(\cellwd,\cellht);
      \fi
    }
  }
 \draw[black!20!white] (0, -19) grid (20,1);
\end{tikzpicture}
}

\newcommand{\braitman}{
\begin{tikzpicture}[scale = 0.2]

  \pgfplotstablegetrowsof{\matrixfile} 
  \pgfmathtruncatemacro{\totrow}{\pgfplotsretval}
  \pgfplotstablegetcolsof{\matrixfile} 
  \pgfmathtruncatemacro{\totcol}{\pgfplotsretval}
  
  \pgfplotstableforeachcolumn\matrixfile\as\col{
    \pgfplotstableforeachcolumnelement{\col}\of\matrixfile\as\colcnt{%
      \ifnum\colcnt=0
        \fill[white]($ -\pgfplotstablerow*(0,\cellht) + \col*(\cellwd,0) $) rectangle +(\cellwd,\cellht);
      \fi
      \ifnum\colcnt=a
        \fill[red]($ -\pgfplotstablerow*(0,\cellht) + \col*(\cellwd,0) $) rectangle+(\cellwd,\cellht);
      \fi
      \ifnum\colcnt=b
        \fill[green]($ -\pgfplotstablerow*(0,\cellht) + \col*(\cellwd,0) $) rectangle+(\cellwd,\cellht);
      \fi
      \ifnum\colcnt=*
        \fill[black!50!white]($ -\pgfplotstablerow*(0,\cellht) + \col*(\cellwd,0) $) rectangle+(\cellwd,\cellht);
      \fi
      \ifnum\colcnt=A
        \fill[blue]($ -\pgfplotstablerow*(0,\cellht) + \col*(\cellwd,0) $) rectangle+(\cellwd,\cellht);
      \fi
      \ifnum\colcnt=B
        \fill[yellow]($ -\pgfplotstablerow*(0,\cellht) + \col*(\cellwd,0) $) rectangle+(\cellwd,\cellht);
      \fi
    }
  }
 \draw[black!20!white] (0, -19) grid (20,1);
\end{tikzpicture}
}

\makeatletter
\def\squarecorner#1{
    \pgf@x=\the\wd\pgfnodeparttextbox%
    \pgfmathsetlength\pgf@xc{\pgfkeysvalueof{/pgf/inner xsep}}%
    \advance\pgf@x by 2\pgf@xc%
    \pgfmathsetlength\pgf@xb{\pgfkeysvalueof{/pgf/minimum width}}%
    \ifdim\pgf@x<\pgf@xb%
        \pgf@x=\pgf@xb%
    \fi%
    \pgf@y=\ht\pgfnodeparttextbox%
    \advance\pgf@y by\dp\pgfnodeparttextbox%
    \pgfmathsetlength\pgf@yc{\pgfkeysvalueof{/pgf/inner ysep}}%
    \advance\pgf@y by 2\pgf@yc%
    \pgfmathsetlength\pgf@yb{\pgfkeysvalueof{/pgf/minimum height}}%
    \ifdim\pgf@y<\pgf@yb%
        \pgf@y=\pgf@yb%
    \fi%
    \ifdim\pgf@x<\pgf@y%
        \pgf@x=\pgf@y%
    \else
        \pgf@y=\pgf@x%
    \fi
    \pgf@x=#1.5\pgf@x%
    \advance\pgf@x by.5\wd\pgfnodeparttextbox%
    \pgfmathsetlength\pgf@xa{\pgfkeysvalueof{/pgf/outer xsep}}%
    \advance\pgf@x by#1\pgf@xa%
    \pgf@y=#1.5\pgf@y%
    \advance\pgf@y by-.5\dp\pgfnodeparttextbox%
    \advance\pgf@y by.5\ht\pgfnodeparttextbox%
    \pgfmathsetlength\pgf@ya{\pgfkeysvalueof{/pgf/outer ysep}}%
    \advance\pgf@y by#1\pgf@ya%
}
\makeatother

\pgfdeclareshape{square}{
    \savedanchor\northeast{\squarecorner{}}
    \savedanchor\southwest{\squarecorner{-}}

    \foreach \x in {east,west} \foreach \y in {north,mid,base,south} {
        \inheritanchor[from=rectangle]{\y\space\x}
    }
    \foreach \x in {east,west,north,mid,base,south,center,text} {
        \inheritanchor[from=rectangle]{\x}
    }
    \inheritanchorborder[from=rectangle]
    \inheritbackgroundpath[from=rectangle]
}

\newcommand{\cor}{{\angle\!\!\cdot\;\!}}

\title{Cutting Corners}

\author{
Ville Salo \\
vosalo@utu.fi
}

\begin{document}
\maketitle

\begin{abstract}
We define and study a class of subshifts of finite type (SFTs) defined by a family of allowed patterns of the same shape where, for any contents of the shape minus a corner, the number of ways to fill in the corner is the same. The main results are that for such an SFT, a locally legal pattern of convex shape is globally legal, and there is a measure that samples uniformly on all convex sets. Under suitable computability assumptions, this measure can be sampled, and legal configurations counted and enumerated, effectively and efficiently. We show by example that these subshifts need not admit a group (more generally unital magma or quasigroup) structure by shift-commuting continuous operations. Our approach to convexity is axiomatic, and only requires an abstract convex geometry that is ``midpointed with respect to the shape''. We construct such convex geometries on several groups, in particular all strongly polycyclic groups and free groups. We also show some other methods for sampling finite patterns, one based on orderings and one based on contructing new ``independent sets'' from old. We also show a link to conjectures of Gottshalk and Kaplansky.
\end{abstract}

\section{Introduction}

In multidimensional symbolic dynamics (and a fortiori symbolic dynamics on groups), for many of the natural finitely presented objects, in particular SFTs and sofics, most natural questions are undecidable. In fact, given a sofic shift $X$ (in any standard way), one can prove analogously to Rice's theorem \cite{Ri53} that it is in general impossible to say anything about it algorithmically. The situation is not much ``better'' for SFTs; non-emptiness \cite{Be66} and the extension problem (does a given pattern appear in a fixed SFT) \cite{Wa61} are undecidable, we refer to \cite{JeRa15} for a discussion of the history and state-of-the-art on the basic tiling problem.

The fact we cannot expect to ever figure out the basic finitely-presented objects can be off-putting. One common solution to this dilemma is to restate the undecidability results in the language of mathematical logic: for example the fact we cannot compute the entropies of subshifts of finite type can be refined to the statement that the set of entropies is precisely the set of $\Pi^0_1$ (or right computabily enumerable) reals \cite{HoMe10}. Some other invariants that turn out to have a recursion-theoretic characterization are possible sets of directions of periodicity \cite{JeMoVa20}, Cantor-Bendixson ranks of countable SFTs \cite{BaDuJe08,SaTo13,To18} and directions of nonexpansivity \cite{Zi16}. The theory of SFTs on groups has a similar flavor, but with the added component of geometric group theory, for example on any product of finitely-generated infinite groups with decidable word problem, entropies have the same characterization as on $\Z^2$ \cite{Ba19}.

The complementary solution to the dilemma, and the one pursued in the present paper, is to try to find subclasses of SFTs where some typical behaviors of SFTs appear, but some things can also be decided. In this paper, we study a class of SFTs defined in a purely combinatorial way, by defining the SFT by a family of allowed patterns of the same shape where, if all but a corner is filled (arbitrarily), the number of ways to legally fill the corner is the same. We call these the subshifts \emph{totally extremally permutive (TEP)}. On the groups $\Z^d$, the language of every TEP subshift is decidable (uniformly in the description of the forbidden patterns), and there is a very natural invariant measure on the subshift -- the TEP measure --, which samples uniformly (thus with maximal entropy) on all convex sets. The following theorem is a simplified version of Theorem~\ref{thm:UCE} and Theorem~\ref{thm:Measure}.

\begin{theorem}
\label{thm:Intro}
Let $A$ be a finite alphabet, $G$ a countable group, $S \subset G$ finite, and $\T \subset A^S$ a finite set. Suppose $|\T| = |A^{|S|-1}|$ and $\forall s \in S: \T|_{S \setminus \{s\}} = A^{S \setminus \{s\}}$. Let $X = \{x \in A^G \;|\; \forall g \in G: gx|_S \in \T\}$. If $G$ admits a translation-invariant $S^{-1}S$-midpointed convex geometry, then
\begin{itemize}
\item locally legal patterns are globally legal: a pattern $P$ of convex shape appears in a configuration of $X$ if and only if all translates of $S$ that fit in its domain contain a pattern in $\T$, and
\item $X$ admits a translation-invariant measure, the \emph{TEP measure}, such that for all convex sets $C$, $\mu|_C$ is the uniform distribution on $X|_C$.
\end{itemize}
\end{theorem}

The assumption on $\T$ can be stated as ``any content is allowed on $S \setminus \{s\}$, and this determines a unique symbol for $s$'' (in fact, it is enough for this to be true when $s$ is a ``corner'' of $S$).
As suggested in the paragraph above, this theorem can be seen as a solution to the undecidability dilemma: if further $G$ has decidable word problem and the convex geometry is decidable, it follows that the language of $X$ is decidable and the measure $\mu$ is computable; this is the case for $\Z^d$ with its standard convex geometry (real convex sets intersected with the lattice).

See Section~\ref{sec:Convexity} for the definition of translation-invariant $S^{-1}S$-midpointed convex geometries, see Theorem~\ref{thm:UCE} and Theorem~\ref{thm:Measure} for more precise statements. Uniformly random samples from TEP measures in the classical multidimensional setting and on the free group can be seen in Figure~\ref{fig:Samples}.

The prime example of a TEP subshift is the Ledrappier three-dot subshift, so it makes sense to compare the properties of TEP subshifts with those of algebraic subshifts, see \cite{Sc95} for a general setting. In the case of the groups $\Z^d$, it is known that the language of a group shift is always decidable uniformly in the defining forbidden patterns \cite{KiSc88}, and of course there is a very nice shift-invariant measure, and thus a ``most natural way'' to sample the subshift, namely the Haar measure.

Like in the case of group shifts, examples of TEP subshifts can be built from finite group (see Example~\ref{ex:Quasigroup}), and when a TEP subshift happens to be a group shift, the TEP measure is the Haar measure. However, even TEP subshifts built directly from a (nonabelian) group operation need not be group shifts: we show in Example~\ref{ex:NotConjugate} that the TEP subshift sampled in Figure~\ref{fig:S3triangle} is not conjugate to any group shift, and more generally that it does not admit continuous shift-commuting quasigroup or unital magma operations. In this example, the subshift clearly has an algebra connection (and it has subshifts that are group shifts, which we use in the proof). In Example~\ref{ex:Arbitrary}, we show that we can inject arbitrary functions into TEP subshifts in the coordinates that are not corners, so it seems unlikely that a generic TEP subshift has a stronger link to algebra than a generic SFT, although it is hard to formalize such a statement (let alone prove it).

Our definition was introduced as a generalization of the TEP cellular automata of \cite{SaTo13b} to subshifts; indeed the spacetime subshift \cite{Sa14} of a TEP cellular automaton is a TEP subshift, and TEP subshifts give rise to certain permutive cellular automata in the abelian case (though this $\Z^d$-specific theory is not studied in the present paper).

The definition of TEP is also motivated by, and close in definition with, the corner deterministic shapes which arise in the study of Nivat's conjecture (see \cite{CyKr15}) and which have been studied in \cite{GuKaZi15,FrKr19a}, in the latter under the name \emph{polygonal subshifts}, mainly on the group $\Z^2$. This more general class has (to some extent) a dynamical characterization beyond what is obvious from the definition, and exhibits interesting directional entropy properties. However, it is not much better than general SFTs in terms of decidability properties, in that the language of an SFT admitting a corner deterministic shape is not uniformly decidable in the allowed patterns \cite{Lu10}, although it has been recently proved that the polygonal SFTs arising from Nivat's conjecture \emph{do} have decidable languages \cite{KaMo20} (as indeed Nivat's conjecture predicts).

As can be seen in Theorem~\ref{thm:Intro}, we take an axiomatic approach to convexity, in that we formulate all statements in terms of abstract convex geometries (more generally so-called ``convexoids''). Though the formalism takes a few pages to set up, it has the benefit of separating the geometric discussion from the symbolic dynamics arguments. It also allows us to generalize the results to other groups. We construct convex geometries with the necessary properties for several groups, in particular all strongly polycyclic groups and free groups.

The main geometric properties we need from our convex geometries are variants of \emph{midpointentedness}, that a convex set containing $g$ and $g^{-1}$ also contains the identity. The convex geometries we give on free groups are \emph{midpointed}, and we report also a construction of such natural convex geometries due to Yves de Cornulier on f.g.\ torsion-free nilpotent groups of small nilpotency class. On general polycyclic groups and groups like the Baumslag-Solitar group $\Z[1/2] \wr \Z$ and the lamplighter group $\Z_2 \wr \Z$ groups, we obtain only translation-invariant convex geometries that are midpointed with respect to a given finite set of elements $g \in S$ (we show that on $\Z[1/2] \wr \Z$ one cannot do better than this). The bare minimum needed to obtain meaningful TEP corollaries is a (not necessarily translation-invariant) $S$-midpointed convex geometry for a given finite set $S$, and the exisence of such a convex geometry is characterized in terms of a well-ordering condition on the group.

On a high level, what is provided by the convex geometry is a meaningful notion of ``independent/free'' coordinates and ``determined'' coordinates in a configuration, namely as we go up an ``anti-shelling'', we encounter free coordinates (where we may pick any symbol) and determined coordinates where we have a unique choice (or one of $k$ choices). We also show some other methods for finding sets of free coordinates, one based on orderings (\emph{contours}, Section~\ref{sec:Contours}) and one based on contructing new ``independent sets'' from old (the \emph{solitaire}, Section~\ref{sec:Solitaire}).

The existence of an midpointed convex geometry requires that the acting group is torsion-free, and conversely our results show that on torsion-free groups admitting such a convex geometry, constructing legal configurations for a TEP subshift is easy. We prove that the problem of whether a ``linear'' TEP subshift contains at least two points on every torsion-free group sits between Kaplansky's unit conjecture, and the conjunction of Gottshalk's surjunctivity conjecture and Kaplansky's unit conjecture holds. Thus, without a convex geometry, even for particularly simple TEP subshifts, the issue of building legal configurations is non-trivial.


\setlength{\abovecaptionskip}{5pt}

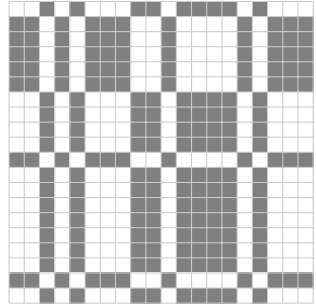
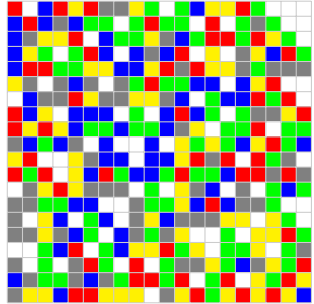
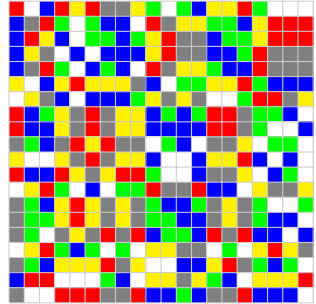
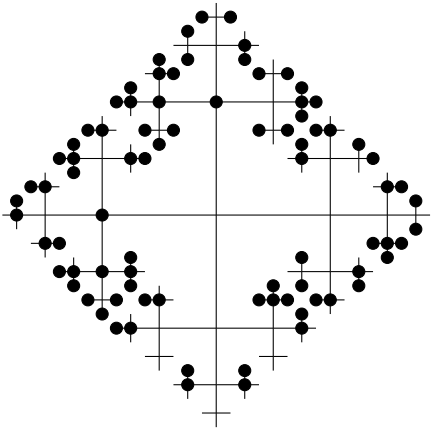
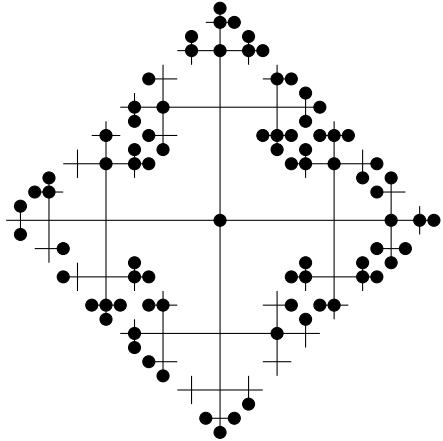
\begin{figure}
\begin{minipage}[t]{0.5\textwidth}
\pgfplotstableread{bintriangle.cvs}{\matrixfile}
\begin{center} \drawitman{} \end{center}
\vspace*{-4mm}
\subcaption{$x_{(1,0)} \equiv x_{(0,1)} + x_{(1,1)} \bmod 2$}
\vspace*{2mm}
\end{minipage}
\begin{minipage}[t]{0.5\textwidth}
\pgfplotstableread{binsquare.cvs}{\matrixfile}
\begin{center} \drawitman{} \end{center}
\vspace*{-4mm}
\subcaption{$x_{(1,0)} \equiv x_{(0,0)} + x_{(0,1)} + x_{(1,1)} \bmod 2$}
\vspace*{2mm}
\end{minipage}
\begin{minipage}[t]{0.5\textwidth}
\pgfplotstableread{S3triangle.cvs}{\matrixfile}
\begin{center} \drawitman{} \end{center}
\vspace*{-4mm}
\subcaption{$x_{(1,0)} = x_{(0,1)} \circ x_{(1,1)}$ in $S_3$}
\vspace*{2mm}
\label{fig:S3triangle}
\vspace*{2mm}
\end{minipage}
\begin{minipage}[t]{0.5\textwidth}
\pgfplotstableread{S3square.cvs}{\matrixfile}
\begin{center} \drawitman{} \end{center}
\vspace*{-4mm}
\subcaption{$x_{(1,0)} = x_{(0,0)} \circ x_{(0,1)} \circ x_{(1,1)}$ in $S_3$}
\vspace*{2mm}
\end{minipage}
\begin{minipage}[t]{0.5\textwidth}
\begin{center}
\begin{tikzpicture}[scale = 0.75]
\input{FreeLedra}
\end{tikzpicture}
\end{center}
\vspace*{-4mm}
\subcaption{$x_{1_{F_2}} = x_a + x_b + x_{a^{-1}} + x_{b^{-1}}$ in $\Z_2^{F_2}$ }
\vspace*{2mm}
\label{fig:FreeLedra}
\vspace*{2mm}
\end{minipage}
\begin{minipage}[t]{0.5\textwidth}
\begin{center}
\begin{tikzpicture}[scale = 0.75]
\input{FreeU}
\end{tikzpicture}
\end{center}
\vspace*{-4mm}
\subcaption{$x_{1_{F_2}} = x_b + x_a + x_{ab}$ in $\Z_2^{F_2}$ }
\vspace*{2mm}
\label{fig:FreeU}
\vspace*{2mm}
\end{minipage}
\caption{Uniform samples from some TEP subshifts. Example (a) is the Ledrappier example. The TEP subshifts in (a), (b), (e), (f) are group shifts, and the TEP measure is the Haar measure, while (c) is not even conjugate to a group shift, see Example~\ref{ex:NotConjugate}. Meanings of colors can be deduced. Examples (e) and (f) are on the free group.}
\label{fig:Samples}
\end{figure}


\section{Definitions}
\label{sec:Definitions}

For $G$ a set, write $S \Subset G$ for $S \subset G \wedge |S| < \infty$. Write $\FinSet(G) = \{S \;|\; S \Subset G\}$. We have $0 \in \N$. The quantifiers $\exists^{\leq k}$ and $\exists^k$ mean ``exists at most $k$'' and ``exist exactly $k$'', respectively. We use the notation $0^S$ for the unique element of $\{0\}^S$. Acting groups are discrete. Conjugation in a group is $h^g = ghg^{-1}$. ``Measure'' refers to a Borel probability measure.

For basic theory and examples of groups, see standard references \cite{LySc15,Ha00,Sc12}. The \emph{strongly polycyclic groups} are the smallest family of groups containing the trivial group and such that $G$ is strongly polycyclic whenever $1 \rightarrow K \rightarrow G \rightarrow \Z \rightarrow 1$ is an exact sequence and $K$ is strongly polycyclic. When $1 \rightarrow K \rightarrow G \rightarrow H \rightarrow 1$ is exact $G$ is a \emph{group extension of $K$ by an action of $H$}. If $F$ is a field and $G$ a group, the \emph{group ring $F[G]$} is the ring of formal sums $\sum c_g g$ where $c_g \in F$ for all $g \in G$ and $c_g \neq 0$ for only finitely many $g$ (the ring structure is the obvious one). We write $\Z[\frac1n]$ for smallest subgroup of the additive group of rational numbers containing $1/n^k$ for all $k \in \N$, so $\Z[\frac12]$ is the additive group of dyadic rationals. For a group property $P$ (i.e.\ family of groups), a group is \emph{locally-$P$} if all its finitely-generated subgroups have property $P$. \emph{Torsion-free} means all nontrivial elements have infinite order.

The \emph{(right) Cayley graph} of a group $G$ with generators $N$ is the simple graph with vertices $G$ and edges $\{\{g, gn\} \;|\; n \in N \setminus \{1\}\}$. When $G$ is finitely generated, we usually pick $N$ as a set of generators (usually symmetric), and leave its choice implicit. The Cayley graph is then connected.

If $G$ is a group and $A$ a finite set, a \emph{pattern} is an element of $A^S$ for $S \subset G$, and a \emph{finite} pattern is one where $S$ is finite. We call $S$ its \emph{domain}, or sometimes \emph{shape} in the finite case, and in general the term ``shape'' is used for finite subsets of a group $G$. Finite (discrete) sets $A$ used to label elements of groups are called \emph{alphabets}, and their elements are called \emph{symbols}. Elements of groups being labeled are sometimes called \emph{cells}. We sometimes write the pattern $a^S$ with $S = \{s\}$ a singleton as $s \mapsto a$. Sometimes we write the domain of a pattern as a tuple rather than a set, so we can specify its contents as a tuple as well.

The \emph{full shift} on a group $G$ with alphabet $A$ is the set $A^G$ with the product topology, and with a $G$-action given by $gx_h = x_{g^{-1}h}$. Its elements $x \in A^G$ are called \emph{configurations}. A \emph{subshift} is a closed $G$-invariant subset. An \emph{SFT} is a subshift of the form $X = \{x \in A^G \;|\; \forall g \in G: gx \in U \}$ where $U \subset A^G$ is clopen. We can write a clopen $C$ as a set of patterns $\T \subset A^S$ for some $S \Subset G$, and we call $\T$ a set of \emph{allowed patterns} for the SFT $X$, and say $X$ is \emph{defined by} $\T$.

Let $G$ be a group, $S \Subset G$ and let $\T \subset A^S$. If $P \in A^C$, write $gP \in A^{gC}$ for the pattern defined by $gP_{h} = P_{g^{-1}h}$. A pattern $P \in A^C$ for $C \subset G$ is \emph{$\T$-legal} if $\forall g \in G: gS \subset C \implies g^{-1}P|_S \in \T$. Thus SFTs are just the sets of configurations that are $\T$-legal as patterns, for some fixed $\T$. If $P \in A^C, Q \in A^D$ are patterns and $C \cap D = \emptyset$, define $P \sqcup Q \in A^{C \cup D}$ by $(P \sqcup Q)|_C = P$, $(P \sqcup Q)|_D = Q$. If $X \subset A^G$ is a subshift and $P \in A^C$ for $C \subset G$, write $P \sqsubset X \iff \exists x \in X: x|_C = P$ and say $P$ \emph{occurs in $X$}. The \emph{language} of a subshift is the set of patterns that occur in it. If an SFT is defined by allowed patterns $\T$, we also use the terms \emph{locally legal} for $\T$-legal patterns and \emph{globally legal} for patterns in the language.

Throughout, we mention some recursion-theoretic and complexity-theoretic corollaries. For decidability results, an intuitive understanding of computability should suffice. For complexity-theoretic claims, some familiarity with the basic theory may be needed, and we refer to \cite{ArBa09}.

For any computability-related discussion, we need computational presentations of groups. We use abstract presentations, to avoid a technical discussion. An \emph{encoded group} is a group together with a bijection with some sublanguage of $A^*$ for some finite alphabet $A$, which is part of the structure and usually left implicit (when useful, we see the group directly as a subset of $A^*$). We say an encoded group is \emph{computable} if the product and inversion of elements are computable operations and $G \subset A^*$ is decidable, and it is a \emph{polytime group} if they are computable in polynomial time and the language $G \subset A^*$ is in the complexity class $P$. One can always recode the alphabet to be $A = \{0,1\}$ in polynomial time, if needed.

For finitely-generated groups, ``computable encoded group'' is essentially a synonym for ``recursively presented group with decidable word problem'' in the standard sense of combinatorial group theory, in fact for such a group there is only one way to see it as an encoded group, up to computable bijection, and indeed there is always a computable bijection with the presentation of elements by lexicographically minimal products of generators. Thus, for the benefit of readers who skipped this discussion, we say a countable group has \emph{decidable word problem} if we are considering it as an encoded group (with respect to some encoding) with respect to which it is computable, and this indeed corresponds to the standard meaning of the term.

For polytime groups, representation issues are less trivial. For the group $\Z^d$, the most common encoding is presumably the presentation of vectors as a tuple of binary numbers, but it is crucial in our results to instead use the \emph{unary} computational presentation, i.e.\ $\vec v \in \Z^d$ is represented as (for example) the word $1^{f(\vec v_1)} 2^{f(\vec v_2)} 3^{f(\vec v_3)} \cdots d^{f(\vec v_d)} \in \{1, 2,..., d\}^*$, $f(n) = 2n$ for nonnegative $n$ and $f(n) = -2n-1$ for negative $n$. This is equivalent to the coding one obtains from the standard group presentation of $\Z^d$ as a finitely-presented group, up to polynomial time computable bijection.

If $G$ is encoded, then $\FinSet(G)$ and elements $A^C$ for $C \in \FinSet(G)$ can be also encoded as words, and we pick the standard encoding. Thus the language of a subshift on $G$ can be seen as a set of words, and we can speak of its decidability and computational complexity.

Write $\mathcal{M}(X)$ for the set of measures on a compact metrizable space $X$. A measure $\mu \in \mathcal{M}(X)$ for $X \subset A^G$ a subshift is \emph{computable} if $X$ has decidable language and given a finite pattern $P \in A^C$ with $P \sqsubset X$ and a rational number $\epsilon > 0$, we can compute a rational number in $[\mu([P])-\epsilon, \mu([P])+\epsilon]$ in finite time, where $[P] = \{x \in X \;|\; x|_C = P\}$. Equivalently, $\mu([P])$ is a lower semicomputable real number for each finite pattern $P$. We can \emph{perfectly sample} $\mu$ if we can, given access to a source of random bits, algorithmically enumerate a configuration, so that the resulting configuration is distributed according to the measure $\mu$. It is not hard to show that perfect samplability is equivalent to computability, see the appendix of \cite{BlMoRoAnAd14} for an analogous result on finite words.

If $X \subset A^G$ is a subshift, the \emph{marginal distribution} of a measure $\mu \in \mathcal{M}(X)$ on $C \Subset G$, denoted $\mu|_C$, is the measure $\mu|_C \in \mathcal{M}(A^C)$ defined by $\mu|_C(P) = \mu([P])$, where for $P \in A^C$ we denote by $[P]$ the \emph{cylinder} $\{x \in X \;|\; x|_C = P\}$.

In figures, we orient $\Z^2$ in figures so that the first axis increases to the right, and the second axis upward.

Basic knowledge of cellular automata comes up in some examples and discussion (but is not needed in any of the results). We refer to \cite{Ka05} for the basic theory.

\section{Convexity}
\label{sec:Convexity}

We define our abstract notion of convexity and give the technical results needed in our applications, we cite \cite{KoLoSc12} for a reference on set systems. 

\begin{definition}
\label{def:CG}
Let $G$ be a \emph{(ground) set}. A set $\C \subset \FinSet(G)$ is a \emph{convex pregeometry} if $\C = \{ \tau(S) \;|\; S \Subset G \}$ where $\tau$ is a closure operator on finite subsets of $G$, i.e. for all $S, T \Subset G$ we have
\begin{itemize}
\item $\tau(\emptyset) = \emptyset$,
\item $S \subset \tau(S)$,
\item $S \subset T \implies \tau(S) \subset \tau(T)$,
\item $\tau(\tau(S)) = \tau(S)$,
\end{itemize}
It is a \emph{convex geometry} if additionally the \emph{anti-exchange axiom}
\[ C \in \C \wedge C \cap \{y, z\} = \emptyset \wedge y \in \tau(C \cup \{z\}) \implies z \notin \tau(C \cup \{y\}) \]
holds.
\end{definition}

When the convex pregeometry is clear from context, we use the notation $\tau(C) = \overline{C}$ for the closure.
If $\C$ is a convex pregeometry (and sometimes even for more general $\C \subset \FinSet(G)$), sets in $\C$ are called \emph{convex} and abusing terminology we also say an infinite set $A \subset M$ is \emph{convex} if $\overline{B} \subset A$ for all $B \Subset A$.  More generally, for any family $\C \subset \FinSet(G)$ and $S \Subset G$, we write $\overline{S} = \bigcap_{S \subset C \in \C} C$ 
(if $\C$ is not a convex pregeometry, $\overline{S}$ may not be in $\C$).

The following lemma is essentially classical, and we leave the proof to the reader.

\begin{lemma}
\label{lem:Moore}
A family $\C \subset \FinSet(G)$ is a convex pregeometry if and only if it is closed under intersections, every finite set in $G$ is contained in some set in $\C$, and $\C$ contains the empty set.
\end{lemma}

The sets $\C = \{A \Subset \Z^d \;|\; A = \mathrm{conv}(A) \cap \Z^d\}$, where $\mathrm{conv}(A) \subset \R^d$ denotes the real convex hull, are well known to be a convex geometry. We call this \emph{the standard convex geometry of $\Z^d$}, and by default convex sets on $\Z^d$ will refer to these sets. These convex sets have many important additional properties:
\begin{itemize}
\item the convex hull of a finite set $B$ is of polynomial size and can be computed (as a finite set) in polynomial time, as a function of the maximal length vector in $B$ (for fixed $d$; recall also that we use unary notation for elements of $\Z^d$),
\item if $C \Subset \Z^d$ is convex, then $\vec v + C$ is convex for all $\vec v \in \Z^d$, and
\item if $C \Subset \Z^d$ is convex and $\{\vec u -\vec v, \vec u+\vec v\} \subset C$, then $\vec u \in C$.
\end{itemize}
For the first item, we give the easy argument in Proposition~\ref{prop:PolyTime}, and the latter two hold by the definition of a convex set. In sections~\ref{sec:AE} and~\ref{sec:TranslationInvariant}, we study the consequences and non-abelian analogs of the latter two properties. Especially the last item -- midpointedness -- play a key role in our results.

On specific groups, one can occasionally find notions of convex sets that seem natural, and we will see relatively natural convex geometries on at least free groups and some torsion-free nilpotent groups. A general, somewhat trivial way to satisfy the axioms of a convex geometry is to order the set $G$ with an order $<$ of type $\omega$ and declare the lower sets as convex. This idea is explored in Section~\ref{sec:NonTranslationInvariant}.

Another general way is building convex geometries from convex geometries on supersets or subsets. We list a few ``obvious'' constructions below. These are straightforward to prove, in each case by guessing the closure operation, verifying it gives the right sets, and then verifying anti-exchange.

\begin{lemma}
\label{lem:ToSubset}
Let $H \subset G$ be two sets and let $\C \subset \FinSet(G)$ be a convex geometry. Then $\{C \cap H \;|\; C \in \C\} \subset \FinSet(H)$ is a convex geometry on $H$.
\end{lemma}

\begin{lemma}
\label{lem:FromPartition}
Let $(G_i)_i$ be a family of sets and $G = \bigsqcup_i G_i$ their disjoint union. Let $\C_i \subset \FinSet(G_i)$ be a convex geometry for each $i$. Then
\[ \{C \Subset G \;|\; \forall i: C \cap G_i \in \C_i\} \]
is a convex geometry on $G$.
\end{lemma}

\begin{lemma}
\label{lem:FromUnion}
Let $(G_i)_{i \in \N}$ be a family of sets and $G = \bigcup_i G_i$ their increasing union. Let $\C_i \subset \FinSet(G_i)$ be a convex geometry for each $i$, such that whenever $i < j$, we have $\C_i \subset \C_j$ and $\C_i = \{C \cap G_i \;|\; C \in \C_j\}$.
Then
\[ \{C \Subset G \;|\; \forall i: C \subset G_i \implies C \in \C_i\} \]
is a convex geometry on $G$.
\end{lemma}

If we define an infinite set to be convex if the closures of its finite subsets are contained in it, then the assumptions of Lemma~\ref{lem:FromUnion} can be equivalently phrased as ``$G_i$ is convex in $G_j$ and $\C_i$ is the restriction of $\C_j$ to $G_i$ in the sense of Lemma~\ref{lem:ToSubset}'' (we omit the proof of this equivalence).

In particular, on the direct union $\Z^\infty = \bigcup \Z^d$ (with the embeddings $\Z^d \cong \Z^d \times \{0\} \leq \Z^{d+1}$) we have a natural convex geometry obtained from Lemma~\ref{lem:FromUnion} applied to the standard convex geometries of the $\Z^d$.

\subsection{Anti-exchange, corner addition and convexoids}
\label{sec:AE}

\begin{definition}
If $\C \subset \FinSet(G)$ is a family of sets, the \emph{corners} of $C \in \C$ are the set
\[ \cor_\C C = \{a \in C \;|\; C \setminus \{a\} \in \C\}. \]
The \emph{lax corners} of $C \Subset G$ are the set
\[ \angle_\C C = \{a \in C \;|\; a \notin \overline{C \setminus \{a\}}\}. \]
\end{definition}

Usually we write $\angle = \angle_\C$ for the lax corners and $\cor = \cor_\C$ for corners, when it is clear which family of sets $\C$ is being discussed. When $\C$ is a convex geometry, for a convex set $C \in \C$ its corners are precisely its lax corners, and in general all corners are lax corners but we may have $\angle C \supsetneq \cor C$ even for $C \in \C$ and $\C$ a convexoid (defined later). Lax corners of $C$ are the elements that can be separated from other elements of $C$ by some set in $\C$.

The main way the anti-exchange axiom features in our applications is in terms of the following property.

\begin{definition}
Let $G$ be a set and $\C \subset \FinSet(G)$. Say $\C$ has the \emph{corner addition property} if
\[ \forall C, D \in \C: (C \subsetneq D \implies \exists a \in D \setminus C: C \cup \{a\} \in \C). \]
\end{definition}

In words, the corner addition property states that if we have two convex sets, one inside the other, then some element of the larger can be added to the smaller so that the resulting set is convex, and the added element is of course a corner of the new set.

\begin{lemma}
\label{lem:AECE}
Let $\C \subset \FinSet(G)$ be a convex pregeometry. Then the following are equivalent:
\begin{itemize}
\item $\C$ is a convex geometry,
\item $\C$ has the corner addition property.
\end{itemize}
\end{lemma}

\begin{proof}
We need to show that the anti-exchange axiom is equivalent to the corner addition property, under the convex pregeometry axioms. Suppose that corner addition fails, and $C, D \in \C, C \subsetneq D$ such that there does not exist $a \in D \setminus C$ such that $C \cup \{a\} \in \C$.

For each $a \in D \setminus C$, $\overline{C \cup \{a\}} \subset D$. Pick $a \in D \setminus C$ such that $\overline{C \cup \{a\}}$ has minimal cardinality. If $\overline{C \cup \{a\}} = C \cup \{a\}$, we are done. Otherwise, let $b \in \overline{C \cup \{a\}} \setminus (C \cup \{a\})$. We have
\[ \overline{C \cup \{b\}} \subset \overline{C \cup \{a, b\}} \subset \overline{\overline{C \cup \{a\}} \cup \{b\}} = \overline{\overline{C \cup \{a\}}} = \overline{C \cup \{a\}}. \]
Since $\overline{C \cup \{a\}}$ was picked to have minimal cardinality, we must have $\overline{C \cup \{b\}} = \overline{C \cup \{a\}}$. But then $a, b \notin C$, $C \in \C$, $a \in \overline{C \cup \{b\}}$ and $b \in \overline{C \cup \{a\}}$, contradicting anti-exchange.

Suppose then that anti-exchange fails, i.e.\ for some $C \in \C$ and $a, b \notin C$, we have $C \in \C$, $a \in \overline{C \cup \{b\}}$ and $b \in \overline{C \cup \{a\}}$. It is easy to show that $\overline{C \cup \{a\}} = \overline{C \cup \{b\}}$, denote this set by $D$. By possibly increasing $C$, we may further assume that if $C \subset E \subset D$ and $E \in \C$, then $E = C$ or $E = D$. This is because if $E$ is a maximal convex subset of $D$ containing $C$ and such that $a, b \notin E$, then $\overline{E \cup \{a\}} = \overline{E \cup \{b\}} = D$, and we may replace $C$ by $E$ without changing $D$.

From the maximality assumption on $C$, we have that if $c \in D \setminus C$, then $\overline{C \cup \{c\}} \cap \{a, b\} \neq \emptyset$, thus $\overline{C \cup \{c\}} = D$ for all $c \in D \setminus C$. Since $|D| \geq |C \cup \{a, b\}| \geq |C| + 2$, this contradicts the corner addition property for the pair $C \subsetneq D$.
\end{proof}

Corner addition has the benefit that stating it does not require the existence of a closure operation, rather it can be stated for any family of sets. We introduce a relaxed notion of convexity which turns out to be sufficient for our purposes.

\begin{definition}
\label{def:Convexoid}
Let $G$ be a set and $\C \subset \FinSet(G)$. We say $\C$ is a \emph{convexoid} if $\emptyset \in \C$, every $B \Subset G$ is contained in some $C \in \C$, and the corner addition property holds.
\end{definition}

We also call elements of a convexoid convex. The term ``convexoid'' is loosely based on the term ``greedoid'': the main difference between the definitions is (arguably) that comparison of cardinalities is replaced by set inclusion, whose relevance to convexity is clear from the above proof. Observe that, by Lemma~\ref{lem:Moore}, a convexoid is a convex geometry if and only if it is closed under intersections.

In general, there certainly exist convexoids that are not convex geometries -- $\{\emptyset, \{0\}, \{1\}, \{0,2\}, \{1, 2\}, \{0,1,2\}\}$ for instance. We do not know interesting examples of such convexoids on groups, and all of our constructions produce true convex geometries. Nevertheless, we state the main results for TEP for convexoids rather than convex geometries, since it makes the results a priori stronger, is precisely what is needed in the proof of the main theorem, and does not lengthen any of the proofs (in fact our experience is that not allowing the use of the algebraic properties of the closure operator often directs one to a \emph{simpler} proof).

\begin{definition}
Let $\C \subset \FinSet(G)$ be a convexoid. An \emph{anti-shelling (from $C_0$ to $C_n$)} is a list $(C_0, C_1, C_2, ..., C_n)$, where for all applicable $i$, $C_i \in \C$ and $C_{i+1} \setminus C_i = \{s_{i+1}\}$ for some elements $s_{i+1} \in G$. We also consider \emph{unbounded anti-shellings} $C_0, C_1, ...$. We then also require $\bigcup_i C_i = G$, and call the set $C_0$ the \emph{base} of the anti-shelling.
\end{definition}

The motivation of the term is that the reverse of a (bounded) antishelling is usually called a shelling in the setting of set systems. Since we need unbounded anti-shellings, this ordering seems more appropriate.

\begin{lemma}
\label{lem:GrowthSequences}
Let $G$ be a set and $\C \subset \FinSet(G)$ a convexoid. If $C,D \in \C$ and $C \subset D$ then there is an anti-shelling from $C$ to $D$, and every set $C \in \C$ is the base of an unbounded anti-shelling.
\end{lemma}

\begin{proof}
Suppose $C, D \in \C$. To find an an anti-shelling $C_0 = C, C_1, C_2, ..., C_n = D$, $C_i \in \C$, simply apply the corner addition property to the pairs $(C_i, D)$ with $i$ taking values $0, ..., n = |D \setminus C| - 1$ in order. To find an unbounded anti-shelling, enumerate $G = \{a_1,a_2,...\}$, and iterate the following procedure starting from $C$, adding one group element at a time: Suppose we have constructed $C_0, ..., C_i$ so far. Then take the minimal $j$ such that $a_j \notin C_i$, let $D$ be any convex set containing $C_i \cup \{a_j\}$ and pick an anti-shelling $C_{i+1}, C_{i+2},  ..., C_{i+|D \setminus C_i|} = D$. Concatenate this to the end of $C_0, ..., C_i$. In the limit, this gives an anti-shelling whose union must contain every element $a_j$, thus $G = \bigcup_k C_k$.
\end{proof}

\begin{example}
\label{ex:ConvexOrdering}
Figure~\ref{fig:ConvexOrdering} shows the convex set $C \subset \Z^2$ (with respect to the standard convex geometry of $\Z^2$) obtained from the closed ball in $\R^2$ of radius $\sqrt{19}$ by discretizing, and an anti-shelling of it (the number $i$ represents the element $s_i$ in the definition of the anti-shelling). The anti-shelling was obtained by starting from an empty set $C_0 = \emptyset$, and iteratively adding a uniformly randomly picked element $a \in C \setminus C_i$ such that $C_{i+1} = C_i \cup \{a\}$ is convex. One can check that this is an anti-shelling by connecting the dots and drawing the successive convex hulls. (The resulting figure is included in Figure~\ref{fig:ConvexOrdering} for completeness, although only the process of building it is useful.)
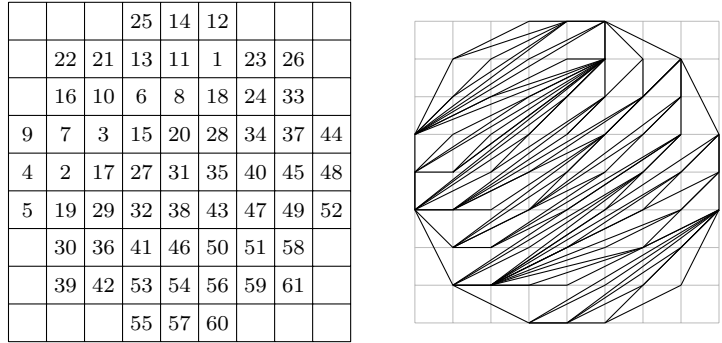
\begin{figure}
\begin{center}
\begin{tikzpicture}
\input{convex_ordering}
\draw[step = 0.5, shift={(0.25,0.25)}] (-2.5,-2.5) grid (2,2);
\end{tikzpicture}
\;\;\;\;\;\;
\begin{tikzpicture}[scale = 0.5,baseline=-114]
\begin{scope}[shift={(0,0.5)}]
\draw[black!25!white] (0,-8) grid (8,0);
\input{convex_as}
\end{scope}
\end{tikzpicture}
\end{center}
\caption{An anti-shelling from $\emptyset$ to a ball.}
\label{fig:ConvexOrdering}
\end{figure} \qee
\end{example}

\subsection{Invariant convex geometries and midpoints}
\label{sec:TranslationInvariant}

If $G$ is a group, a family $\C \subset \FinSet(G)$ is \emph{invariant} if $C \in \C \implies \forall g \in G: gC \in \C$. The following definitions are specific to the setting of subsets of groups.

\begin{definition}
A family $\C \subset \FinSet(G)$ on a group $G$ is \emph{$S$-midpointed} if
\[ \forall g \in G, h \in S, C \in \C: g \in \overline{\{gh, gh^{-1}\}}, \]
and \emph{midpointed} if it is $G$-midpointed. Say $\C$ has \emph{unique corner positioning for $S \Subset G$ (or $S$-UCP)} if
\[ \forall C \in \C: a \in \angle C \implies \exists^{\leq 1} g: gS \subset C \wedge gS \ni a, \]
and say $\C$ has \emph{weak unique corner positioning for $S \Subset G$ (or $S$-UCP)} if
\[ \forall C \in \C: a \in \cor C \implies \exists^{\leq 1} g: gS \subset C \wedge gS \ni a. \]
We say $\C$ has UCP if it has UCP for all $S \Subset G$. We say $\C$ has (weak) UCP if it has (weak) $S$-UCP for all $S \Subset G$.
\end{definition}

Clearly $S$-UCP implies weak $S$-UCP, and clearly an invariant family $\C \subset \FinSet(G)$ is midpointed if and only if $1_G \in \overline{\{g, g^{-1}\}}$ for all $g \in G$. On occasion, we also need to talk about individual midpointed sets $C \Subset G$, meaning that the family $\{C\}$ is ($S$-)midpointed.

There is a simple connection between midpointedness and unique corner positioning:

\begin{lemma}
\label{lem:UCP}
Let $G$ be a group, $\C \subset \FinSet(G)$ a convexoid and $S \Subset G$. If $\C$ is $S^{-1}S$-midpointed, then it has $S$-UCP. Conversely, if $\C$ has weak $T$-UCP for all $T \subset S$, then $\C$ is $S^{-1}S$-midpointed.
In particular, weak UCP, UCP and midpointedness are equivalent concepts.
\end{lemma}

\begin{proof}
Suppose $S$-UCP fails. Then there exist $C \in \C$, $a \in C$ such that $a \in \angle C$ 
$S \Subset G$, and distinct $g_1, g_2 \in G$ such that $a \in g_iS \subset C$ for $i \in \{1, 2\}$. Let $a = g_1s_1 = g_2s_2$ Clearly $g_1 \neq g_2$ implies $s_1 \neq s_2$, and since $g_1 s_2, g_2 s_1 \in C$, we have $a s_1^{-1} s_2 = g_1 s_2 \in C$ and $a s_2^{-1} s_1 = g_2 s_1 \in C$. Thus, $C \setminus \{a\}$ contains $a h$ and $ah^{-1}$ for $h = s_1^{-1} s_2 \in S^{-1}S$. Since $a \in \angle C$, there is a set $D \in \C$ with $D \cap C = C \setminus \{a\}$, which contradicts the $S^{-1}S$-midpointedness of $\C$.

Suppose then that $\C$ is not $S^{-1}S$-midpointed, i.e.\ $g \notin \overline{\{gh, gh^{-1}\}}$ for some $g \in G, h \in S^{-1}S$, say $h = s_1^{-1}s_2$. By the definition of $\overline{\{gh, gh^{-1}\}}$, there exists a convex set $C'$ containing $\{gh, gh^{-1}\}$ but not $g$, and by the definition of a convexoid there exists $D \in \C$ containing $\{g, gh, gh^{-1}\}$. Pick an anti-shelling from $C'$ to $D$, and take $C$ to be the first set containing $g$. Then $\{gh, gh^{-1}, g\} \subset C$ and $g$ is a corner of $C$. Setting $T = \{s_1, s_2\} \subset S$, we have $g \in g s_1^{-1} T = \{g, gh\} \subset C$ and $g \in g s_2^{-1} T = \{g, gh^{-1}\} \subset C$. Thus $\C$ does not have weak $T$-UCP.
\end{proof}

\begin{definition}
If $\C \subset \FinSet(G)$, the set $\bigcup_{g \in G} g^{-1} \angle_{\C} gS$ is the set of \emph{translated lax corners} of $S$.
\end{definition}

Recall that the lax corners are elements $s \in S$ such that some convex set contains $S \setminus \{s\}$ but not $s$. The translated lax corners are obtained by taking the lax corners $\angle S$ of all translates of $S$ and translating them back inside $S$, equivalently taking the union of $\angle_{g\C} S$ over all $g \in G$. If $\C$ is invariant, $g^{-1} \angle gS = \angle S$ for all $g \in G$, i.e.\ all translated lax corners are lax corners.

\subsection{Polynomial time algorithm for convex sets on $\Z^d$}

The fact that ``the convex hull of a set of real vectors can be computed in polynomial time'' is a well-known fact in computational geometry. Nevertheless, it is non-trivial to find, in the literature, such a result that directly applies in our situation: algorithms are often specific to even prime dimensions, and even when not, they often give the extremal vertices instead of bounding half-spaces, or assume that points are in general position, and such algorithms will not really simplify the problem at hand. In any case, some postprocessing is needed to construct the actual discrete set from the description of the geometric convex hull. Thus, it seems easier to give a direct proof.

\begin{proposition}
\label{prop:PolyTime}
Consider $\Z^d$ as an encoded group, with the unary encoding. Let $\tau$ be the closure operator for the standard convex geometry of $\Z^d$. Then, given $S \Subset \Z^d$, the set $\tau(S)$ can be computed in polynomial time.
\end{proposition}

\begin{proof}
Let $R$ be the maximal value that appears in the vectors in $S$, and observe that, since we are working in unary notation, the list of all vectors in $[-R,R]^d$ is of polynomial size in the input size. Thus, it suffices to check, for every individual $\vec v \in [-R,R]^d$, whether $\vec v \in \tau(S)$.
By setting $T = S - \vec v$, we have reduced the problem to verifying in polynomial time whether $\vec 0 \in \tau(T)$ for a set $T$ of integer vectors of dimension $d$ and with entries in $[-2R, 2R]$.

Now, we recall that the convex hull of $S$ in $\R^d$ is precisely the intersection of all affine half-spaces of dimension $(d-1)$ that contain $S$. By the definition of a half-plane, it is thus enough to find out whether there is a real column vector $x \in \R^d$ such that the strict inequality $Ax > 0$ holds, where $A$ is the $|T|$-by-$d$ integer matrix whose rows are the vectors in $T$.

Such a system of inequalities can be solved for a rational matrix $A$ by using at most $O(|T|d)$ field operations on the matrix coefficients, so that each matrix entry takes part in at most $O(d)$ calculations. Thus, even with exact calculations with rational numbers, the time will be polynomial in $R$.

We outline this algorithm: If the first column of $A$ contains no negative coefficients, then it is enough to solve $B \vec y > 0$ where $B$ is the ${(\leq\!|T|)}$-by-$(d-1)$ matrix of $(d-1)$-suffixes of rows of $A$ having $0$ in the first coordinate, as $A \vec x > 0$ then holds for any $\vec x = (a, \vec y^T)^T$ with large enough $a$. A similar reduction happens if the first column of $A$ contains no positive coefficients. Thus, in these cases we can reduce $d$ to $d-1$ without applying any field operations to the input matrix and conclude by induction.

If the first column of $A$ contains both positive and negative coefficients, we observe that if $E$ is invertible and $E \vec v > 0 \iff \vec v > 0$, then $Ax > 0$ has a solution if and only if $EAx > 0$ has one. Thus, we can apply any elementary row operations with positive entries to $A$. This way we can eliminate all negative coefficients from the first column, with each individual entry taking part in only a constant number of field operations. Now apply the previous paragraph to solve the resulting problem, and by induction every coefficient will be involved in at most $O(d-1)$ additional field operations.
\end{proof}

In practical computer implementations, we have instead used Carath\'eodory's theorem and taken the union of convex hulls of simplices, as this is fast enough for small examples in dimension two and is very quick to implement.

The assumption that vectors are specified in unary notation rather than in binary is natural from the general point of view, as we are dealing with subshifts on groups -- for most groups, analogs of binary representations tend to be much less canonical, and much more difficult to work with. Specifying the inputs in unary is also necessary for polynomial-time computability of the convex hull if the output is listed explicitly: already in $\Z$ the convex hull of $\{0,n\}$ has $n+1$ elements, so the output complexity will not be polynomial in the input if the input is specified in binary.

One could of course use a more efficient coding of the output convex hull, as is typically done in computational geometry. There are some reasons why this is not natural in our context: such implicit convex sets cannot really be used as supports of patterns without revealing the hidden exponential (as the data of a pattern of such shape is exponential in the description of the shape), and one cannot write down (in binary notation) the exact number of legal patterns of such shapes in a TEP subshift in polynomial time, as the number is typically doubly exponential in an efficient description of the convex set.

\section{TEP}
\label{sec:UCEandTEP}

In this section, we define the TEP subshifts. One should keep in mind the case of the standard convex sets $\Z^d$. Recall that by Lemma~\ref{lem:UCP} they have UCP. We will construct ($S$-)UCP convex geometries on some other groups in Section~\ref{sec:ExamplesOfCG}.

\begin{definition}
Let $G$ be a group and $\C \subset \FinSet(G)$ a convexoid, $\emptyset \neq C \subset S \Subset G$ and $A$ a finite alphabet. We say a family of patterns $\T \subset A^S$ has \emph{$k$-uniform $C$-extensions} if
\[ \forall s \in C: \forall P \in A^{S \setminus \{s\}}: \exists^k a \in A: P \sqcup (s \mapsto a) \in \T. \]
We say $\T$ is \emph{($k$-)totally extremally permutive}\footnote{We immediately acknowledge that the term ``totally extremally permutive'' does not really make sense for $k > 1$, but it makes some sense for the case $k = 1$ (which is the notion we studied first), as a $1$-TEP family can be shown to set up a \emph{permutive} relation between the \emph{extremal} vertices (translated lax corners), and the function from possible contents of $S \setminus \{g\}$ to those of $g$ is indeed \emph{total}, i.e.\ defined on all of $A^{S \setminus \{g\}}$.} or \emph{($k$-)TEP} if it has uniform $g^{-1} \angle gS$-extensions for all $g \in G$. We say $\T$ has \emph{TEP} if it has $1$-TEP. If $C = \{c\}$ is a singleton, we speak of ($k$-)uniform $c$-extensions instead of $C$-extensions.
\end{definition}

We sometimes say an SFT $X \subset A^G$ is $k$-TEP if it is defined by an allowed set of patterns $\T \subset A^S$ which is $k$-TEP. This does not mean all patterns in $\T$ appear in $X$, although in the presence of a suitable convex geometry Theorem~\ref{thm:UCE} states precisely that the do.

We recall the key properties about the set of translated lax corners. First, we have $g^{-1} \angle gS \subset S$ for all $g \in G$, so $k$-uniform $S$-extensions implies $k$-TEP no matter what $\C$ is. Second, when $\C$ is an invariant convex geometry, $\bigcup_g g^{-1} \angle gS = \cor S$, the set of corners of $S$. The set $S$ is sometimes called the \emph{shape} of the TEP family or subshift.

The following describes how the number of legal patterns in a TEP subshift behaves then we step from a convex set to a larger one.

\begin{lemma}
\label{lem:UniformExtensions}
Let $G$ be a countable group, let $S \Subset G$, let $\C \subset \FinSet(G)$ be an $S$-UCP convexoid, let $A$ be a finite alphabet, let $\T \subset A^S$ have $k$-TEP, let $X$ be the SFT with allowed patterns $\T$, let $C, C \cup \{a\} \in \C$ with $a \notin C$, and let $P \in A^C$ be $\T$-legal. Let
\[ \ell = |\{Q \in A^{C \cup \{a\}} \;|\; Q|_{C} = P \mbox{ and } Q \mbox{ is $\T$-legal} \}|. \]
Then
\begin{itemize}
\item if $a \in gS \subset C \cup \{a\}$ for some $g \in G$, we have $\ell = k$, and
\item if such $g \in G$ does not exist, we have $\ell = |A|$.
\end{itemize}
\end{lemma}

\begin{proof}
By the $S$-UCP property, and the assumption that $C, C \cup \{a\}$ are convex, there is at most one $g$ such that $a \in gS \subset C \cup \{a\}$. If there is at least one such way, i.e.\ we are in the case of the first item of the lemma, then because $C$ is convex we have $a = gt \in \angle gS$, so $t \in g^{-1} \angle gS$. 
Enumerate the patterns over $S$ extending $g^{-1}(P)|_{S \setminus \{t\}}$ pattern as $Q'_1, ..., Q'_{|A|}$. By the $k$-TEP property there are exactly $\ell = k$ legal ways to extend $g^{-1}(P)|_{S \setminus \{t\}}$ to a pattern with support $S$, let these be $Q'_1, ..., Q'_k$. Translating back, we obtain that out of the patterns $Q_i = gQ'_i$ extending $P|_{gS \setminus \{a\}}$ to $gS$ exactly $Q_1,...,Q_k$ are $\T$-legal at $gS$.

It follows that there are at most $k$ extensions of $P$ to $C \cup a$. But for any $1 \leq i \leq k$, the pattern $P' \in C \cup \{a\}$ defined by $P'_g = P_g$ if $g \neq a$, $P'_g = (Q_i)_g$ otherwise, is locally legal: the contents of $gS$ is legal by the assumption on the $Q_i$, and all other translates of $S$ that are contained in $C \cup \{a\}$ are even contained in $C$ and we assumed $P$ is locally legal. Thus, there are exactly $k$ locally legal extensions, meaning $\ell = k$.

If we are in the case of the second item, i.e.\ there is no $g \in G$ such that $a \in gS \subset C \cup \{a\}$, then any extension of $P$ to $Q \in A^{C \cup \{a\}}$ is trivially $\T$-legal, so $\ell = |A|$.
\end{proof}

\subsection{Language and measure}

Our first main result states that $k$-TEP implies that every locally legal pattern supported on a convex set is globally legal.

\begin{theorem}
\label{thm:UCE}
Let $G$ be a countable group, $S \Subset G$, $\C \subset \FinSet(G)$ an $S$-UCP convexoid, $A$ a finite alphabet, $\T \subset A^S$ have $k$-TEP, and $X$ be the SFT with allowed patterns $\T$. If $C \in \C$ and $P \in A^C$ is $\T$-legal, then $P \sqsubset X$. If $G$ has decidable word problem and $\C$ is recursively enumerable, then $X$ has computable language, uniformly in the description of $\T$.
\end{theorem}

Note that this theorem applies even if $\C$ is not invariant.

\begin{proof}
Suppose $P = P_0 \in A^C$ is $\T$-legal and $C \in \C$. By Lemma~\ref{lem:GrowthSequences} there is an unbounded anti-shelling $C_0 = C, C_1, C_2, ...$. By Lemma~\ref{lem:UniformExtensions}, for each $i \in \N$ we can find at least one $\T$-legal pattern $P_{i+1} \in A^{C_{i+1}}$ (in fact, we can find at least $k$). Since $\bigcup_i C_i = G$ is an increasing union, in the limit (under cellwise convergence), we obtain a configuration $x \in A^G$. It is in $X$ because its finite subpatterns are $\T$-legal. Thus, $P \sqsubset X$.

If $G$ is computable with respect to a fixed encoding, then given a pattern $P \in A^B$ with $B \Subset G$, we can easily check whether $P$ is $\T$-legal. Namely, we simply need to consider all possible translates of $S$ which fit inside $B$, and these can be enumerated easily. If $\C$ is recursively enumerable, then given a pattern $P \in A^B$ for a finite set $B \Subset G$, we enumerate any convex $C \in \C$ such that $C \supset B$. We then simply check whether we can $\T$-legally extend $P$ to a pattern $Q \in A^C$ with $Q|_B = P$, which, as we observed, is a decidable task. This can be done if and only if $P \sqsubset X$, proving that the language is computable.
\end{proof}

\begin{theorem}
\label{thm:Measure}
Let $G$ be a countable group, $S \Subset G$, $\C \subset \FinSet(G)$ an $S$-UCP convexoid, $A$ a finite alphabet, $\T \subset A^S$ have $k$-TEP, and $X$ be the SFT with allowed patterns $\T$. Then there exists a (unique) measure $\mu$ on $X$ such that for all $C \in \C$, the marginal distribution $\mu|_C$ is uniform on $X|_C$. If $\C$ is invariant, then $\mu$ is invariant. If $G$ has a decidable word problem and $\C$ is recursively enumerable, then this measure is computable uniformly in the description of $\T$ and can be sampled perfectly.
\end{theorem}

We call the above measure the \emph{TEP measure}. Of course, it may a priori depend on $\C$. This theorem is in some sense only as interesting as the convex geometry it is applied to. In particular, it does not seem particularly interesting dynamically when $\C$ is not invariant (its counting variants proved below may be more interesting in that case).

\begin{proof}
For the construction of the measure, we use the quantitative information that the number of extensions only depends on the shapes, given by Lemma~\ref{lem:UniformExtensions}. Associate to each $C \in \C$ the uniform measure $\mu_C$ on $X|_C$ (by the previous theorem, equivalently on the set of $\T$-legal patterns of shape $C$). If $C \subset D$ and $C, D \in \C$, using Lemma~\ref{lem:GrowthSequences} we can find an anti-shelling from $C$ to $D$.

By iterated application of Lemma~\ref{lem:UniformExtensions}, all $\T$-legal patterns $P \in A^C$ have the same number of extensions $m$ to $\T$-legal patterns in $A^D$, so if we denote by $\mu_D|_C$ the distribution that $\mu_D$ induces on patterns in $A^C$, and suppose there are $\ell$ many $\T$-legal patterns in $A^C$, by uniformity of $\mu_D$ we have
\[ \mu_D|_C(P) = \mu_D([P]) = m/\ell m = 1/\ell = \mu_C(P) \]
where $[P] = \{Q \in A^D \;|\; Q|_C = P\}$. Now by basic measure theory, there is a unique measure $\mu$ of $A^G$ such that $\mu|_C = \mu_C$ for all $C$ (for example it is easy to verify the assumptions of the Hahn-Kolmogorov theorem), which concludes the construction.

It is clear that this construction gives a shift-invariant measure when $\C$ is invariant, since the property of being $\T$-legal is invariant.

If $G$ has a decidable word problem, then to show that the measure is computable and perfectly samplable, we need only show the latter. For this, enumerate an anti-shelling (using recursive enumerability of $\C$) with base $\emptyset$, and sample the symbols at the new corners $s_i$ of the sets $C_i$ uniformly from the $k$ or $|A|$ possible choices depending on which case of Lemma~\ref{lem:UniformExtensions} we are in (which is decidable at each step by decidability of the word problem).
\end{proof}

The following is clear from the proof, and will be refined in the following section.

\begin{corollary}
Let $G$ be a countable group, $\C \subset \FinSet(G)$ an $S$-UCP convexoid. For a TEP subshift $X \subset A^G$ with defining shape $S$, the number of globally admissible patterns with domain $C \in \C$ is independent of the TEP family $\T \subset A^S$, and is always a power of $|A|$.
\end{corollary}

The following proposition is optimal in the case $G = \Z$, and its motivation should become clear in Section~\ref{sec:Kaplansky}.

\begin{proposition}
\label{prop:AtLeastPoints}
Let $G$ be a countable group, $S \Subset G$, $\C \subset \FinSet(G)$ an $S$-UCP convexoid, $\T \subset A^S$ have $k$-TEP, and $X$ be the SFT with allowed patterns $\T$. If $k = 1$, then $X$ has at least $|A|^{|S|-1}$ configurations. If $k \geq 2$, then $X$ is uncountable and homeomorphic to the Cantor set.
\end{proposition}

\begin{proof}
The convex geometry $\C$ contains at least one set $C$ of cardinality exactly $|S|-1$, namely the $|S|$th set in any unbounded anti-shelling with base $\emptyset$. Any pattern in $A^C$ is $\T$-legal, because no translate of $S$ can fit inside $C$. Thus, any such pattern appears in a configuration of $X$ by Theorem~\ref{thm:UCE}.

If $k \geq 2$, Lemma~\ref{lem:UniformExtensions} gives at least two choices for each new element seen along an unbounded anti-shelling with base $\emptyset$, so there is a Cantor set of extensions for any pattern with convex domain (and any pattern is a subpattern of such a pattern). Thus $X$ is a (nonempty) perfect totally disconnected compact metrizable space, thus uncountable and homeomorphic to the Cantor set.
\end{proof}

We will show in Section~\ref{sec:NoMidpointed} that only torsion-free groups can admit midpointed convex geometries. Indeed, Proposition~\ref{prop:AtLeastPoints} fails on all non-torsion-free groups in the following sense.

\begin{example}
\label{ex:Finite}
Let $G$ have an element of finite order, and choose some $g \neq 1_G$ such that $g^n = 1_G$ for $n \geq 2$. Let $A = \Z_q$ for any $q > n$. Let $p \in A[G]$ be the element $1_G - g$ of the group ring $\Z_q[G]$. Then the configurations $x \in A^G$ satisfying $xp = 1^G$ (where $1 \in \Z_q$ is the additive generator) form an empty $1$-TEP SFT. \qee
\end{example}

\subsection{More or less algebraic examples}

In this section we give some basic examples of TEP subshifts. We begin with the linear examples, where in addition to the TEP structure we have the structure of a group shift, in fact a vector (space) shift over a finite field. We then illustrate why TEP really generalizes linear subshifts in a different direction than group shifts do, by showing an example of a TEP subshift that does not commute with any shift-commuting continuous magma operation that is either invertible or unital. In this example, non-algebraicity is easy to prove using its algebraic subsystems. The dependence of a TEP rule on the non-corner coordinates can be rather arbitrary, giving, presumably, even less algebraic examples. However, proving their non-algebraicity is more difficult.

The simplest and most important examples of TEP subshifts are the following ones, which sit in the intersection of the algebraic realm (group shifts) and the TEP realm.

\begin{example}
\label{ex:Ledrappier}
Let $G$ be a group.
let $F$ be a finite field, and let $p \in F[G]$ be any element of the group ring $F[G]$. Identify $x \in F^G$ with the formal sum $\sum_{g \in G} x_g \cdot g$. Then configurations $x \in F^G$ satisfying $x \cdot p = 0^G$ form a subshift of finite type. This gives us a TEP rule with shape $S$, where $S^{-1} = \{s^{-1} \;|\; s \in S\}$ is the support of $p$. In particular setting $G = \Z^2$, $\C$ is the standard convex geometry, $|F| = 2$ and $p = 1 + x^{-1} + y^{-1}$, we obtain that the \emph{Ledrappier subshift} $\{x \in \{0,1\}^{\Z^2} \;|\; x_{(1,0)} \equiv x_{(0,1)} + x_{(1,1)} \bmod 2\}$ is TEP.
\qee
\end{example}

Subshifts defined by $\{x \;|\; x \cdot p = 0 \}$ for $p \in F[G]$ as in the previous example are called \emph{linear TEP subshifts}. They are a classical family of group/vector shifts in symbolic dynamics. In the last Section~\ref{sec:Kaplansky} we show a connection to a conjecture of Kaplansky, and indeed what Proposition~\ref{prop:AtLeastPoints} gives for linear TEP subshifts on strongly polycyclic groups (which we later show admit $S$-UCP convex geometries for all finite sets $S$) is a known case of Kaplansky's unit conjecture.

If Theorem~\ref{thm:Measure} is applied to a linear TEP subshift, we obtain just the Haar measure. In particular the Haar measure of the Ledrappier subshift from Example~\ref{ex:Ledrappier} is its TEP measure.

\begin{theorem}
Let $G$ be a countable group, $A$ a finite group and $X \subset A^G$ a subshift that is a group under cellwise multiplication. If $X$ is simultaneously $k$-TEP with shape $S$, with respect to any $S$-UCP convexoid on $G$, then the TEP measure is equal to the Haar measure.
\end{theorem}

Note that we do not necessarily assume $X$ is linear TEP, simply that it admits each of these two types of structures.

\begin{proof}
Since $X$ is a compact group, it is unimodular, meaning the left and right Haar measures coincide. The Haar measure is the unique Borel (probability) measure $\mu$ that satisfies
\begin{itemize}
\item $\mu(xE) = \mu(E)$ for all $x \in X$ and Borel sets $E$, and
\item $\mu(U) > 0 \mbox{ for all nonempty open $U$}$.
\end{itemize}
It is thus enough to verify these for the TEP measure. Since the multiplication is cellwise, it is clear that the TEP measure satisfies $\mu(xE) = \mu(E)$ for cylinder sets $E$ with a convex domain and arbitrary $x \in G$, since a legal cylinder is translated to another legal cylinder under group translation. Thus it holds for all Borel sets, since the sets where this holds are easily seen to form a $\sigma$-algebra and every finite set is contained in a convex set of the convexoid $\C$ so the generated $\sigma$-algebra is the one of Borel sets. As for the second, any nonempty open set $U$ contains a nonempty cylinder, to which we give positive measure by the definition of the TEP measure.
\end{proof}

\begin{corollary}
Let $G$ be a countable group, $F$ a finite field and $X \subset F^G$ a linear TEP subshift for some convexoid. Then the TEP measure is the Haar measure (of $X$ as a compact subgroup of the compact group $(F^G, +)$ with cellwise addition).
\end{corollary}

Now, we move on to the less algebraic examples.

\begin{example}
\label{ex:Quasigroup}
Let $G$ be a group admitting an $S$-UCP convex geometry $\C$ for a finite set $S \Subset G$. Let $A$ be a quasigroup (\cite{BuSa81}), that is, $A$ is a latin square, the multiplication operation $a*b$ extracts the element on row $a$ and column $b$, and operations $a/b$ and $a \bs b$ are determined by the identities
\[ (a*b)/b = a, (a/b)*b = a, a \bs (a*b) = b, a * (a \bs b) = b. \]

Let $E$ be a formal identity in the variables in $S$ and constants in $A$ such that every variable $s$ that is a translated lax corner of $S$ (meaning $s \in g^{-1} \angle gS$ for some $g \in G$) appears exactly once. Define a subshift $X \subset A^G$ by
\[ X = \{ x \in A^G \;|\; \forall g \in G: E(gx|_S) \}, \]
where $E(f)$ for $f : S \to A$ means that the identity $E$ holds with valuation $f$.

Letting $\T = \{P \in A^S \;|\; E(P)\}$, we see that $X$ is a TEP subshift: $\T$ is a TEP family because if all values except the value of some translated lax corner are known, then similarly in the previous example there is a unique possible value for $s$, obtained by using the quasigroup operations of $A$ to move the unique occurrence of $s$ to the LHS of the identity and everything else to the RHS. In particular any such $X$ is nonempty if $|S| \geq 2$.

For an example of solving, $S = \{a, b, c, d\}$, suppose the translated lax corners are $a, b, c$, and $e \in A$ is a constant, then the identity
\[ d \bs ((a * d) \di (c * d)) = e \di ((b \bs e) * d) \]
is of the required type. We can solve
\[ a = ((d * (e \di ((b \bs e) * d))) * (c * d)) \di d, \]
\[ b = e \di (((d \bs ((a * d) \di (c * d))) \bs e) \di d), \]
\[ c = ((d * (e \di ((b \bs e) * d))) \bs (a * d)) \di d, \]
giving the extension rule.
\qee
\end{example}

Of course, a quasigroup is a perfectly good universal algebra, so this is in some sense an algebraic way to construct subshifts. However, even if the quasigroup $A$ in the construction of the previous example is a group, typically the TEP subshift obtained is not a subgroup of $A^G$ under cellwise multiplication (although this does happen when $G$ is abelian). In other words, even if the system is in some sense algebraic, it is not algebraic in the same sense that group shifts are.

We now analyze the most basic example of a TEP subshift coming from a nonabelian group multiplication in more detail, and show that this subshift is not even conjugate to a group shift. We need the basic number theoretic fact that Pascal's triangles modulo different primes look different. Define \emph{Pascal's triangle} as $P \in \N ^{\N^2}$ by $P_{(0,0)} = 1$, $P_{(m,0)} = 0$ for $m > 0$ and inductively $P_{(m,n+1)} = P_{(m+1,n)} + P_{(m,n)}$.
Write $B_n = \{v \in \Z^2 \;|\; |v|_1 \leq n \}$ where $|\cdot|_1$ denotes the $\ell^1$ norm.

\begin{lemma}
Let $P_2, P_3 \subset \N^2$ be the sets of positions where the Pascal triangle contains a nonzero number modulo $2$ and $3$, respectively. Then for all $n \in \N$ and $i \in \{2,3\}$ there exists $v \in P_i$ such that $(v + B_n) \cap P_{5-i} = \emptyset$.
\end{lemma}

\begin{proof}
For $i = 2$, pick $p$ such that $2^p \in [\frac{4 \cdot 3^m}{3}, \frac{5 \cdot 3^m}{3}]$ for some $m$. Arbitrarily large such $(p, m)$ can be found by taking logarithms and using the minimality of the irrational rotation by $\log_3 2$. Now, using the self-similarity of $P_2$ and $P_3$, it is easy to check that
\[ (2^{p-1}, 2^p-1) \in P_2 \mbox{ and } (2^{p-1}, 2^p-1) + B_n \cap P_3 = \emptyset \]
for $n = \frac{3^m}{6}$. For $i = 3$ the proof is similar.
\end{proof}

%

\begin{example}
\label{ex:NotConjugate}
Let $A = S_3$ (the symmetric group on the set $\{1,2,3\}$) and consider the subshift $X$ defined as follows
\[ X = \{x \in A^{\Z^2} \;|\; \forall \vv: x_{\vv + (0,-1)} = x_{\vv + (-1,0)} \circ x_{\vv + (0,0)} \}. \]
In the orientation of Example~\ref{ex:ConvexOrdering}, $X$ is the $\Z^2$-subshift over alphabet $S_3$ where in every pattern $\begin{smallmatrix} a & b \\ & c \end{smallmatrix}$ we have $c(n) = a(b(n))$ for all $n \in \{1,2,3\}$. This rule also determines a TEP family of patterns $\T \subset S_3^S$ where $S = {\{(0,-1),(-1,0),(0,0)\}}$ and $X$ is the corresponding TEP subshift. This is the subshift sampled in Figure~\ref{fig:S3triangle} (though $S$ is slightly shifted for technical reasons).

We show that the subshift $X \subset S_3^{\Z^2}$ does not admit the structure of a unital magma or the structure of a quasigroup, with continuous shift-commuting operations. We begin with the case of a unital magma structure. Denote the magma operation by $\star$, and observe that the only possible unit is $0^{\Z^2}$, as it is the only fixed point.

Consider the cellular automaton (shift-commuting continuous function) $f : S_3^{\Z} \to S_3^{\Z}$ defined by $f(x)_i = x_{i-1} \circ x_i$. Denote $a = (1\; 2\; 3), b = (1 \; 3 \; 2), A = (1 \; 2), B = (2 \; 3), * = (1 \; 3)$ and we denote the identity permutation by empty space in the figure below. We refer to $\{a,b,A,B,*\}$ as nonzero symbols, and the \emph{nonzero support} of a configuration is the set of positions where nonzero symbols occur. Let $v_1 = (-2^n,-m), v_2 = (0,0)$ and $v' = (0,-1) \in S$, and consider the half-open stripe
\[ P = (((\R v_1 + \overline{S}) \setminus \R v_1 + v') \cap \Z^2 \]
where $\overline{S}$ denotes the convex hull inside $\R^2$. This stripe spans between $v_1$ and $v_2$ and is just the right width not to contain a translate of $S$. It is clearly convex, so by Theorem~\ref{thm:UCE}, any contents is legal on this strip. It is easy to see that the TEP rule then determines the entire configuration uniquely\footnote{In standard terminology, this line is \emph{coding}, or \emph{expansive} (for the partition given by basic cylinders).}, as long as $m > 0$: slide $S$ along either side of this strip to extend the known contents (Section~\ref{sec:Contours} uses a similar argument).

Now, consider what happens below this stripe for configuration $z$ defined by
\[ z_{v_1} = B, \; z_{v_2} = a, \; \forall v \in P \setminus \{v_1, v_2\}: z_v = 0. \]
The configuration is given as the spacetime diagram of the cellular automaton $f$ (i.e.\ each row is obtained by applying the local rule of $f$, which is just the TEP rule), and its behaviors from the lone $B$ and the long $a$ are easy to understand: since $\langle B \rangle \cong \Z_2$, $\langle a \rangle \cong \Z_3$ and the rule is simply addition, below $B$ we see Pascal's triangle modulo $2$, and below $a$ we see Pascal's triangle modulo $3$, as long as these triangles do not meet. We informally refer to these as the \emph{$2$-adic and the $3$-adic components} of the spacetime diagram, and informally refer to the self-similarity-inducing time steps when the spacetime splits into copies of itselfs as \emph{splits}. Now, we pick suitable $n$ and $m$. The choices are best explained by looking at Figure~\ref{fig:AdicsFight}.


\verbfilebox[\tiny]{adicsfight3}

\begin{figure}
\begin{tikzpicture}
\node () at (0,0) {\theverbbox};
\end{tikzpicture}
\caption{As the two components meet, bees turn into stars.}
\label{fig:AdicsFight}
\end{figure}

The choice in the figure is $n = 5$, so $-2^n = -32$, and pick $m = -22$, and what is shown is part of the configuration $z$ below the stripe $P$, which is actually a spacetime diagram of the cellular automaton $f$. The stripe $P$ is marked with dots. We have picked $m$ so that at the time when the $2$-adic component touches the $3$-adic component, both components are about to split. In numbers, we chose $n$ and $m$ so that $2^n+m = 2 \cdot 3^k$ for some $k$, namely $k = 4$.\footnote{We use $3$-adic splitting times of type $2 \cdot 3^k$ instead of $3^k$ so that the figure is of a suitable size; of course what is relevant is large $3$-adic norm.}

The point is now that because $Ba = (2 \; 3)(1 \; 2 \; 3) = (1 \; 3) = *$, and $\langle * \rangle = \Z_2$, the collision has the effect that the original $B$ continues producing the $2$-adic pattern, but over $*$ instead of $B$, while the $3$-adic area disappears after the split.

It is clear that we can produce this effect for arbitrarily large $n, m$, by taking $k$ and $n$ large and setting $m = 2 \cdot 3^k - 2^n$. Now consider $x, y \in X$ defined by
\[ x_{v_1} = B, \forall v \in P \setminus \{v_1\}: x_v = 0 \]
and
\[ y_{v_2} = a, \forall v \in P \setminus \{v_2\}: y_v = 0. \]
By unitality and because the contents of $P$ determine the entire configuration, we have $x \star y = z$. By the previous lemma, we can now find positions $v$ such that the $2$-adic pattern has a nonzero symbol (that is, $z_v = *$, and $x_v = B$), and $v$ is at distance at least $h$ from any element of the nonzero support of $y$. From the continuity and shift-commutation of the magma operation, we obtain that $\star : X^2 \to X$ has a \emph{radius} $r$ such that
\[x'|_{u + B_r} = x''|_{u + B_r} \wedge y'|_{u + B_r} = y''|_{u + B_r} \implies (x' \star y')_{u} = (x'' \star y'')_{u} \]
for all $u \in \Z^2$, where $B_r$ is the ball of radius $r$. As soon as $h \geq r$, the continuations of $x$ and $y$ above $P$ have , and we have
\[ * = z_{\vec v} = (x \star y)_{\vec v} = (x \star 0^{\Z^2})_{\vec v} = x_{\vec v} = B, \]
a contradiction.

The case of a quasigroup operation is proved similarly. Denote this operation again by $\star$ and let $\di$ and $\bs$ be the division operations. Let $z$ be the same configuration as above, and let $x$ and $y$ be as above. Define $y' = 0^{\Z^2} \bs y$, observe that $y'|_P$ contains only finitely many non-$c$ symbols for some $c \in S_3$. Let $w \in X$ be the unique configuration with $w_|P = c^P$ and let $x' = x \di w$. The quasigroup axioms imply that $x'|_P$ contains only finitely many nonzero symbols. Now by analyzing what happens on $P$ (which determines the entire configuration), we see that as soon as $(-m, -2^n)$ and $(0,0)$ are sufficiently far from each other, we have
\[ x' \star y' = z. \]
Since $w$ and $0^{\Z^2}$ both have finite $\Z^2$-orbits, $x' = x \di w$ and $y' = 0^{\Z^2} \bs y$ are images of $x$ and $y$ by continuous homeomorphisms on $X$ that commute with a finite-index subgroup of $\Z^2$. Thus, zero areas of $y$ surrounding $v$ turns into a periodic area in $y'$. Just like in the unital magma case, the rule has no way of knowing when it should produce the $2$-adic pattern of of $*$s instead of $B$s. \qee
\end{example}

The example does admit the structure of a magma with shift-invariant continuous operations, by setting $x * y = 1_{S_3}^{\Z^2}$ for all $x, y$ (this is even a semigroup), but we are not aware of any nontrivial magma structures on it. There are of course TEP subshifts without fixed points (even on $\Z^2$), which then trivially cannot have a unital magma structure, but we don't know any examples that ``trivially'' do not have a quasigroup structure. Once one has a full understanding on at least some configurations of the TEP subshift, non-algebraicity can often be proved as in the above example.

Of course, multiplying quasigroup elements together is just one way to obtain TEP subshifts. Since the definition of TEP is entirely syntactic, it is easy to produce examples that are less obviously algebraic. We show one way of injecting an arbitrary function into the definition of TEP.

\begin{example}
\label{ex:Arbitrary}
Let $G$ be a group admitting an $S$-UCP convex geometry $\C$ for a finite set $S \Subset G$. Let $A$ be a quasigroup. Let $C \subset S$ be the corners of $S$. Let $(A, *)$ be a quasigroup, define $g : A^C \to A$ by multiplying the elements of $A$ together in some order (or any other function such that $g^{-1}(a)$ has $1$-uniform $C$-extensions for every $a \in A$), and pick any function $f : A^{S \setminus C} \to A$. Let $a \in A$. Then
\[ \{ P \in A^S \;|\; f(P|_{S \setminus C})*g(P|_{C} = a \} \]
is a TEP family. Since $f$ is an arbitrary function, it is hard to imagine such examples being very algebraic in general. This does not rule out that all such effects cancel out on the large scale. \qee
\end{example}

We also note a famous class of non-examples of TEP subshifts: As noted already in the introduction, the decidability aspects of TEP subshifts given by Theorem~\ref{thm:UCE} differ considerably from those of SFTs defined by allowed patterns $\T \in A^S$ with the property
\[ \forall s \in \angle S: \forall P \in A^{S \setminus \{s\}}: \exists^{\leq 1} a \in A: P \sqcup (s \mapsto a) \in \T \]
which have been studied on $\Z^d$ equipped with the standard convex geometry in \cite{GuKaZi15,FrKr19a}, in the latter under the name \emph{polygonal subshift}: with this definition, the domino problem (nonemptiness) stays undecidable, at least for $S = \{0,1\}^2$ \cite{Lu10}. (To the author's knowledge, decidability of nonemptiness is open for other shapes, in particular $S = \{(0,0), (1,0), (0,1)\}$.) It follows that there cannot be an effective procedure for turning a polygonal subshift into a conjugate TEP subshift, even on the group $\Z^2$.

\subsection{Complexity-theoretic statements and counting}

Here we state some refinements of Theorem~\ref{thm:UCE} and Theorem~\ref{thm:Measure}, which imply that for $\Z^d$ both verification of legal pattern, and counting patterns, can be done efficiently.



\begin{theorem}
\label{thm:Counting}
Suppose $G$ is a polytime encoded group and $\C \subset \FinSet(G)$ is a polynomial time computable convexoid. Let $\T \subset A^S$ be a $k$-TEP family of allowed patterns for an SFT $X$. Given $S,k,|A|$ and a finite convex set $C$, we can compute the number of patterns in $X$ of shape $C$ in polynomial time in $S,C,k,|A|$.
\end{theorem}

Note that we do not need to see the actual patterns $\T$. The numbers $k$ and $|A|$ may be given in binary even though for the groups $\Z^d$ as acting groups we always use unary notation. Note also that the convexoid need only be verifiable in polynomial time, i.e.\ it is enough that $\C$ is in the complexity class $P$ under its natural encoding as a language -- we do not a priori need to be able to find a convex set containing a given finite set in polynomial time.

\begin{proof}
This is proved exactly as Theorem~\ref{thm:UCE}. Since $\C$ is polynomial time verifiable and has the corner addition property, we can in polynomial time construct an anti-shelling (by greedily adding corners). We always have $|A|$ or $k$ distinct ways to fill the corner symbol depending on whether we can position $S$ inside the current convex set so that it touches the newest corner, and we can check in polynomial time which case we are in because $G$ is a polytime encoded group. Simply calculate the resulting product (which will be of the form $k^m |A|^n$).
\end{proof}

\begin{example}
\label{ex:CountingExample}
Consider the convex set $C$ and its anti-shelling from Example~\ref{ex:ConvexOrdering}. Given any $k$-TEP rule, we can sample the contents of $C$ uniformly (among globally legal patterns) by iterating through the positions (which are always corners of the resulting convex set) in the stated order, and
\begin{itemize}
\item whenever the defining shape $S$ fits inside the convex shape built so far, and touches the newly-added corner, choose a symbol out of the ones determined by the rule uniformly at random,
\item if it does not fit in, pick the contents uniformly at random from the entire alphabet $A$.
\end{itemize}
Since the TEP measure samples uniformly on all convex sets, the produced sample is indeed uniform. This same ordering works for TEP subshifts defined with respect to any shape $S$. With the shape $S = \begin{smallmatrix} * & * \\ & * \end{smallmatrix}$ one can verify that along the ordering of Example~\ref{ex:ConvexOrdering}, the positions where we can pick an arbitrary symbol are
\[ 1,2,3,4,5,6,7,8,10,11,12,23,26,39,44 \]
so in any $k$-TEP subshift with this shape there are exactly $|A|^{15} k^{46}$ globally legal patterns of domain $C$. With shape
$S' = \begin{smallmatrix} * & * & * \\ & & * \end{smallmatrix}$ the free positions are
\[ 1,2,...,14,15, 19, 23, 26, 30, 39, 42, 44, \]
so $22$ in total, and thus in any $k$-TEP subshift with shape $S'$ there are exactly $|A|^{22} k^{39}$ globally legal patterns of domain $C$. \qee
\end{example}

\begin{theorem}
Suppose $G$ is a polytime encoded group and $\C \subset \FinSet(G)$ is a polynomial time computable convex geometry. Given the shape $S$ defining a $k$-TEP subshift $X$ over an alphabet $A$, a convex set $C \in \C$, and a pattern $P \in A^C$, we can check in polynomial time whether $P$ is in the language of $X$. If every set $B \in \FinSet(C)$ is contained in a convex set (whose desciption is) polynomial in (the description of) $B$, then the set of all $\T$-legal patterns is in the complexity class NP.
\end{theorem}

\begin{proof}
The first claim is trivial, since given a pattern of convex shape, we simply need to check whether it is $\T$-legal, and the assumptions on the group imply this can be done in polynomial time. For the latter claim, given $P \in A^B$ for a finite set $B \Subset G$, nondeterministically guess a convex set $C \Supset B$ of polynomial size, and then a pattern $Q \in A^C$. Verify that $Q$ is $\T$-legal and $Q|_B = P$.
\end{proof}

Note that if $\C$ is a convex geometry, once we have guessed a convex set $C \supset B$ which is of polynomial size in the set of $B$, we can actually compute the convex hull of $B$ in deterministic polynomial time by using Lemma~\ref{lem:AECE} to drop elements one by one from $C$ while keeping it a convex superset of $B$. However, even if we assume that the true convex hull $C$ can be computed in polynomial time, we still have to guess the pattern $Q$, and even with $G = \Z^2$ it is not clear how to do this in polynomial time, indeed we do not know whether the language of a TEP subshift can be NP-complete for $G = \Z^2$ and the standard convex geometry.

While Theorem~\ref{thm:Counting} shows that the number of patterns of a convex shape can be computed quickly, it seems much harder to count the number of legal patterns on non-convex sets, and indeed these counts seem to behave in a complicated way. We give some pattern counts in the following example (the calculations were done by computer).

\begin{example}
\label{ex:Counts}
Let $A$, $X$ be as in Example~\ref{ex:NotConjugate}. Consider the following subsets of $\Z^2$, where $*$ marks the included elements, and $\cdot$ denotes a visible space, which we use to mark the convex hull
\[ \begin{smallmatrix} * & * & * & * \\
&*&* & \\
&  &  * & \\
& &  * & \end{smallmatrix} \;\;\;\;
\begin{smallmatrix} * & \cdot & *\\
& & \cdot \\
& & *
 \end{smallmatrix} \; \;\;\;
 \begin{smallmatrix} * & * & \cdot & *\\
& \cdot & \cdot & \cdot \\
& & \cdot & \cdot \\
& & & *
 \end{smallmatrix} \;\;\;\;
  \begin{smallmatrix} * & * & * & \cdot & *\\
&\cdot & \cdot & \cdot & \cdot \\
& & \cdot & \cdot & \cdot \\
& & & \cdot & \cdot \\
& & & & *
\end{smallmatrix} \;\;\;\;
\begin{smallmatrix}
* & \cdot & * &  &  \\
 & \cdot & \cdot & * &  \\
 & & * & \cdot & * \\
 & & & \cdot & \cdot \\
 & & & & *
\end{smallmatrix} \]
For the leftmost shape (which is convex), one can verify that the number of $\T$-legal patterns is precisely $7776 = 2^5 \cdot 3^5$, which as expected is a power of $|A| = 6$. As expected from Theorem~\ref{thm:UCE}, each of these $\T$-legal patterns extends to a legal configuration of the $4 \times 4$ square containing the shape.

By a direct calculation, one can show for the second shape $\{(0,0),(2,0),(2,-2)\}$ that the number of patterns of that shape which extend to a legal configuration (obtained by extending to a convex set and applying Theorem~\ref{thm:UCE}) is $108 = 2^2 \cdot 3^3$, for the third shape $\{(0,0),(1,0),(3,0),(3,-3)\}$ this number is $1080 = 2^3 \cdot 3^3 \cdot 5$, for the fourth shape $\{(0,0),(1,0),(2,0),(4,0),(4,-4)\}$ this number is $3456 = 2^7 \cdot 3^3$, and for the last shape the number is $5616 = 2^4 \cdot 3^3 \cdot 13$.

Even in the case of abelian group shifts over abelian groups, the number of patterns is not necessarily a power of the alphabet size $|A|$: For the sum rule over alphabet $A = \Z_6$ with the same shape defining a group shift in $\Z_6^{\Z^2}$, the corresponding numbers for these five shapes are $7776$ (of course), $108 = 2^2 \cdot 3^3$, $432 = 2^4 \cdot 3^3$ and $3888 = 2^4 \cdot 3^5$ and (again) $3888$, respectively. This also shows that the number of patterns of a given (non-convex) shape can depend on the rule.

For vector shifts such as the Ledrappier example, on the other hand, basic linear algebra shows that the number of patterns of any shape is a power of the cardinality of the underlying field. \qee
\end{example}

We have no theoretical understanding of the counts on non-convex sets. Based on our brief experimentation on the TEP subshift $X$ above, the counts typically have only $2$ and $3$ in their prime decomposition, but sporadically one sees other primes. We have seen $5$ several times, $13$ exactly once, and we have seen no other primes. We have not experimented systematically with other alphabet sizes, TEP shapes and rules, and with other groups, and have no idea whether this reflects a general trend.

\subsection{An easier way to count: contours}
\label{sec:Contours}

If a group admits a (left-)invariant order, then we associate to every convex set a subset that can be filled arbitrarily, and assuming $k$-uniform $g$-extensions for suitable $g$, there are $k$ ways to pick the value of each cell after that. More precisely, to each convex set $C$ we will associate a subset $E$ called its $S$-contour so that every possible filling of $E$ is $\T$-legal, and extends to a filling of $C$ in exactly $k^{|C| \setminus |E|}$ possible ways, whenever $\T \subset A^S$ has $k$-uniform extensions in the maximal coordinate of $S$. In other words, the ``free choices'' are all made before all the ``constrained choices''. The drawback of the method compared to using an anti-shelling are that it only works for finite sets, and the filling order depends on $S$.

Suppose $G$ is linearly ordered by a left-invariant order $<$, i.e. $h < k \iff gh < gk$ for all $g, h, k \in G$. Say that a $S\subset G$ is in \emph{good position} if $1_G \in S$ and $1_G$ is the $<$-maximal element of $S$ (i.e.\ all elements of $S$ are ``negative''). Note that we can turn any TEP family of patterns $\T \subset A^S$ into one with $S$ in good position, without changing the SFT it defines, by translating its patterns by the inverse of the maximal element of the shape $S$.

Let $S \Subset G$ be in good position. Let $C \subset G$, and let
\[ E = \{ g \in G \;|\; g S \not\subset C\}. \]
Then $E$ is called the \emph{$S$-contour} of $C$. More generally, if $S$ is not in good position, then we define the $S$-contour as the $gS$-contour for the unique $g$ such that $gS$ is in good position.

\begin{lemma}
\label{lem:Contour}
Let $\T \subset A^S$ with $S$ in good position, and suppose $\T$ has $k$-uniform $\{1_G\}$-extensions. Let $C \Subset G$ and let $E$ be the $S$-contour of $C$. Then for each pattern $P : E \to S$, there are exactly $k^{|C| \setminus |E|}$ many $\T$-legal patterns $Q : C \to A$ with $Q|_E = P$.
\end{lemma}

\begin{proof}
Every pattern $P : E \to A$ is $\T$-legal, because by the definition of $E$ there are no translates of $S$ contained in $E$. We consider a fixed such pattern and show that it has exactly $k^{|C| \setminus |E|}$ extensions. 

Let $D = C \setminus E$.
Enumerate the vectors in $D$ in the $<$-order, i.e.\ enumerate $g_1, g_2, g_3, \ldots, g_\ell$ where $g_{i+1}$ is the $<$-minimal vector in $D \setminus E_i$, where $E_i = E \cup \{g_1, \ldots, g_i\}$. Let us show that if $w \in \{1,2,...,k\}^i$ and $P_{w} : E_i \to A$ is $\T$-legal then there are exactly $k$ distinct $\T$-legal patterns $P_{wa} : E_{i+1} \to A$, $a \in \{1,2,...,k\}$, with $P_{wa}|_{E_i} = P_{w}$.


We have $g_{i+1} S \subset E_{i+1}$ because $g_{i+1} \in D$ and thus $g_{i+1} S \subset C$, and all vectors in $g_{i+1} S \setminus \{g_{i+1}\}$ are $<$-smaller than $g_{i+1}$ by left-invariance of $<$, so if they are not in $E$, they are among the $g_j$ with $j \leq i$. Since $\T$ has $k$-uniform $\{1_G\}$-extensions, we have exactly $k$ ways to fill the coordinate $g_{i+1}$ so that in the translate $g_{i+1} S$ we do not see a forbidden pattern of $X$. The patterns $P_{wa}$ are taken to be any enumeration of such patterns.

We claim that in fact the only translate $g S$ such that $g_{i+1} \in g S \subset E_{i+1}$ is the one with $g = g_{i+1}$. Namely, if $g \in E$, then $g S$ is not contained in $C$, if $g$ is among the $g_j$ with $j \leq i$ then $g S$ does not contain $g_{i+1}$, and finally if $g = g_j$ for some $j > i+1$, then $g S \ni g$ is not contained in $E_{i+1}$. This implies that all the patterns $P_{wa}$ are $\T$-legal, since any forbidden pattern in $P_{wa}$ would have to have a domain of the form $gS$ and contain the new coordinate $g_{i+1}$. 
\end{proof}

\begin{theorem}
\label{thm:Contour}
Let $G$ be a group and let $S \Subset G$ be in good position. Let $\C \subset \FinSet(G)$ be an $S$-UCP convexoid. Let $X \subset A^G$ be defined by a $k$-TEP family $\T \subset A^S$, and suppose $\T$ also has $k$-uniform $\{1_G\}$-extensions. Let $C \in \C$ and let $E$ be the $S$-contour of $C$. Then for each pattern $P : E \to S$, there are exactly $k^{|C| \setminus |E|}$ patterns $Q : C \to A$ with $Q|_E = P$ such that $Q \sqsubset X$.
\end{theorem}

\begin{proof}
Apply the previous lemma, and observe that since $C$ is convex and $\T$ is $k$-TEP, the $\T$-legal pattern $Q$ extends to a configuration of $X$.
\end{proof}

This can be used to obtain formulas for the number of globally admissible patterns of a particular shape. We give a somewhat trivial example, and count the number $(m \times n)$-rectangles that occur in a TEP subshift on $\Z^2$. Let $S \subset \Z^2$ be finite, and let $m_i = \max_{\uu, \vv \in S} |\uu_i - \vv_i| - 1$ for $i = 1, 2$. We say the \emph{width} of $A$ is $m_1$ and its \emph{height} is $m_2$.

\begin{lemma}
Let $X \subset A^{\Z^2}$ be a TEP subshift with defining shape $S$. Let $m_1$ and $m_2$ be the width and height of $S$, respectively, and suppose $n_1 \geq m_1$, $n_2 \geq m_2$. Then
\[ |X|_{[1, n_1] \times [1, n_2]}| = |A|^{n_1 m_2 + m_1 n_2 - m_1 m_2}. \]
\end{lemma}

\begin{proof}
Clearly the size of the contour is $n_1 m_2 + m_1 n_2 - m_1 m_2$.
\end{proof}

Let us also reproduce the pattern count for the convex shape in Example~\ref{ex:Counts}.

\begin{example}
The group $\Z^2$ is left-invariantly ordered by the lexicographic order. Consider the shape $S = \begin{smallmatrix} * & * \\ & * \end{smallmatrix}$ defining a TEP subshift over an alphabet $A$ with $|A| = 6$. In Example~\ref{ex:Counts}, the contour of the first shape on the list is marked with $*$s, and numbers mark the order in which the rest of the cells are filled in
\[ \begin{smallmatrix} * & 1 & 3 & * \\
& * & 2 & \\
&  &  * & \\
& &  * & \end{smallmatrix} \]
The contents of the contour positions can be picked arbitrarily, and the rest is determined uniquely, so the number of patterns is $6^5 = 7776$, as indeed the computer claimed in Example~\ref{ex:Counts}. \qee
\end{example}

One can play with all the parameters to get more exotic examples:

\begin{example}
Pick the vector $\vec w = (\pi, 1)$ (i.e.\ the irrational mathematical constant $\pi \approx 3$) so that $\vec u \leq \vec v \iff (\vec v - \vec u) \cdot \vec w \geq 0$ (where $\cdot$ denotes the dot product) is an invariant (total) order on $\Z^2$. Consider a TEP subshift with shape $S = \begin{smallmatrix} * & * & * \\
& & * \end{smallmatrix}$, i.e.\ $S = \{(0,0), (1,0), (2,-1), (2,0)\}$. Let $E$ be the intersection of the closed ball of radius $\sqrt{19}$ with $\Z^2$ (which is clearly convex). Then the contour and ordering of the other cells as in the proof of Lemma~\ref{lem:Contour} is as follows
\begin{center}
\begin{tikzpicture}
\node () at (-2.0, -0.5) {*};
\node () at (-2.0, 0.0) {*};
\node () at (-1.5, -1.5) {*};
\node () at (-2.0, 0.5) {*};
\node () at (-1.5, -1.0) {*};
\node () at (-1.5, -0.5) {*};
\node () at (-1.5, 0.0) {*};
\node () at (-1.0, -1.5) {*};
\node () at (-1.5, 0.5) {*};
\node () at (-1.0, -1.0) {*};
\node () at (-1.5, 1.0) {*};
\draw (-1.25, -0.75) rectangle (-0.75, -0.25);\node () at (-1.0, -0.5) {1};
\node () at (-0.5, -2.0) {*};
\node () at (-1.5, 1.5) {*};
\draw (-1.25, -0.25) rectangle (-0.75, 0.25);\node () at (-1.0, 0.0) {2};
\draw (-0.75, -1.75) rectangle (-0.25, -1.25);\node () at (-0.5, -1.5) {3};
\draw(-1.25, 0.25) rectangle (-0.75, 0.75);\node () at (-1.0, 0.5) {4};
\draw (-0.75, -1.25) rectangle (-0.25, -0.75);\node () at (-0.5, -1.0) {5};
\node () at (-1.0, 1.0) {*};
\draw (-0.75, -0.75) rectangle (-0.25, -0.25);\node () at (-0.5, -0.5) {6};
\node () at (0.0, -2.0) {*};
\node () at (-1.0, 1.5) {*};
\draw (-0.75, -0.25) rectangle (-0.25, 0.25);\node () at (-0.5, 0.0) {7};
\draw (-0.25, -1.75) rectangle (0.25, -1.25);\node () at (0.0, -1.5) {8};
\draw (-0.75, 0.25) rectangle (-0.25, 0.75);\node () at (-0.5, 0.5) {9};
\draw (-0.25, -1.25) rectangle (0.25, -0.75);\node () at (0.0, -1.0) {\small 10};
\draw (-0.75, 0.75) rectangle (-0.25, 1.25);\node () at (-0.5, 1.0) {\small 11};
\draw (-0.25, -0.75) rectangle (0.25, -0.25);\node () at (0.0, -0.5) {\small 12};
\node () at (0.5, -2.0) {*};
\draw (-0.75, 1.25) rectangle (-0.25, 1.75);\node () at (-0.5, 1.5) {\small 13};
\draw (-0.25, -0.25) rectangle (0.25, 0.25);\node () at (0.0, 0.0) {\small 14};
\draw (0.25, -1.75) rectangle (0.75, -1.25);\node () at (0.5, -1.5) {\small 15};
\node () at (-0.5, 2.0) {*};
\draw (-0.25, 0.25) rectangle (0.25, 0.75);\node () at (0.0, 0.5) {\small 16};
\draw (0.25, -1.25) rectangle (0.75, -0.75);\node () at (0.5, -1.0) {\small 17};
\draw (-0.25, 0.75) rectangle (0.25, 1.25);\node () at (0.0, 1.0) {\small 18};
\draw (0.25, -0.75) rectangle (0.75, -0.25);\node () at (0.5, -0.5) {\small 19};
\draw (-0.25, 1.25) rectangle (0.25, 1.75);\node () at (0.0, 1.5) {\small 20};
\draw (0.25, -0.25) rectangle (0.75, 0.25);\node () at (0.5, 0.0) {\small 21};
\node () at (1.0, -1.5) {*};
\node () at (0.0, 2.0) {*};
\draw (0.25, 0.25) rectangle (0.75, 0.75);\node () at (0.5, 0.5) {\small 22};
\draw (0.75, -1.25) rectangle (1.25, -0.75);\node () at (1.0, -1.0) {\small 23};
\draw (0.25, 0.75) rectangle (0.75, 1.25);\node () at (0.5, 1.0) {\small 24};
\draw (0.75, -0.75) rectangle (1.25, -0.25);\node () at (1.0, -0.5) {\small 25};
\draw (0.25, 1.25) rectangle (0.75, 1.75);\node () at (0.5, 1.5) {\small 26};
\draw (0.75, -0.25) rectangle (1.25, 0.25);\node () at (1.0, 0.0) {\small 27};
\node () at (1.5, -1.5) {*};
\draw (0.25, 1.75) rectangle (0.75, 2.25);\node () at (0.5, 2.0) {\small 28};
\draw (0.75, 0.25) rectangle (1.25, 0.75);\node () at (1.0, 0.5) {\small 29};
\draw (1.25, -1.25) rectangle (1.75, -0.75);\node () at (1.5, -1.0) {\small 30};
\draw (0.75, 0.75) rectangle (1.25, 1.25);\node () at (1.0, 1.0) {\small 31};
\draw (1.25, -0.75) rectangle (1.75, -0.25);\node () at (1.5, -0.5) {\small 32};
\draw (0.75, 1.25) rectangle (1.25, 1.75);\node () at (1.0, 1.5) {\small 33};
\draw (1.25, -0.25) rectangle (1.75, 0.25);\node () at (1.5, 0.0) {\small 34};
\draw (1.25, 0.25) rectangle (1.75, 0.75);\node () at (1.5, 0.5) {\small 35};
\draw (1.25, 0.75) rectangle (1.75, 1.25);\node () at (1.5, 1.0) {\small 36};
\node () at (2.0, -0.5) {*};
\draw (1.25, 1.25) rectangle (1.75, 1.75);\node () at (1.5, 1.5) {\small 37};
\draw (1.75, -0.25) rectangle (2.25, 0.25);\node () at (2.0, 0.0) {\small 38};
\draw (1.75, 0.25) rectangle (2.25, 0.75);\node () at (2.0, 0.5) {\small 39};
\draw[step = 0.5, shift={(0.25,0.25)}] (-2.5,-2.5) grid (2,2);
\end{tikzpicture}
\end{center}
We see that there are $22$ free choices and $39$ constrained choices, confirming the number we obtained in Example~\ref{ex:CountingExample} by sampling along an anti-shelling. \qee
\end{example}

\begin{example}
\label{ex:FreeContour}
Just as happens with $\Z^2$, the space of left-invariant orderings of the two-generator free group $F_2$ is a Cantor set \cite{Na10}, so there are many choices for the order. Also, through a suitable interpretation, just like $\Z^2$, the free group admits a natural lexicographic order which is invariant by translations from both sides, which we summarize (see also \cite{DeNari14}): Apply the \emph{Magnus transformation} $M : F_2 \to \Z \langle a, b \rangle$ where $\Z \langle a, b \rangle$ is the (noncommutative) algebra of formal $\Z$-linear combinations of words in $\{a, b\}^*$, and the mapping is the group homomorphism (to the multiplicative semigroup of $\Z \langle a, b \rangle$) induced by $a \mapsto 1+a$ on the positive generators. Observe that indeed $1+a$ is invertible, and $M(a^{-1}) = 1 - a + a^2 - a^3 + \cdots$. Now define $u \leq v \iff M(u) \leq M(v)$, where power series are compared by ordering the variables first by length, then by the lexicographic order, and finally comparing the sequence of coefficients lexicographically. Let us call this the \emph{Magnus ordering}.

An invariant order is of course determined by the positive elements. Let us argue in pictures and embed $F_2$ into $\Z^2$ by its Cayley graph with respect to the free generators, so that it is a $4$-regular tree. Right multiplication by the generator $a$ is represented by a right edge, and right multiplication by $b$ represents is represented by an upward edge. Figure~\ref{fig:Magnus} shows the ball of radius $4$, with a black circle at nodes $u \in F_2$ satisfying $u \geq 1_{F_2}$ in the Magnus ordering.

\begin{figure}
\begin{center}
\begin{tikzpicture}
\input{magnus}
\end{tikzpicture}
\end{center}
\caption{The ball of radius $4$ on the free group, with ``nonnegative'' nodes marked by black circles.}
\label{fig:Magnus}
\end{figure}
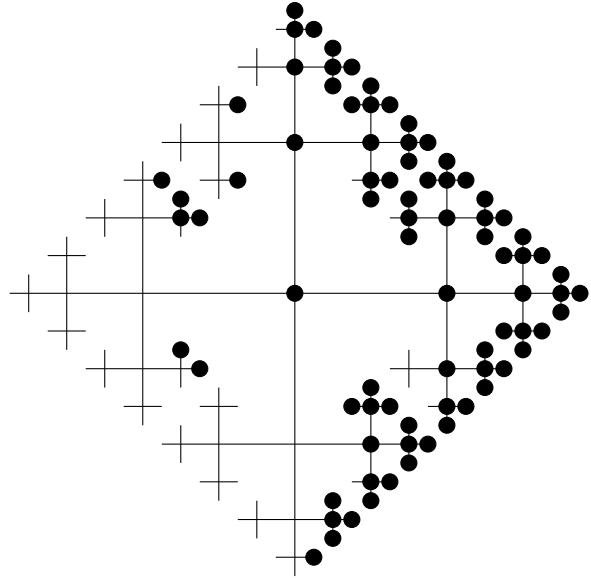

Now, consider the shape $S = \{1_{F_2}, a, b, a^{-1}, b^{-1}\}$,. We see from the figure that $a$ is the maximal element in the Magnus ordering, since $a^{-1} S$ is the only translate of $S$ that touches the identity and does not contain positive elements. Suppose thus that $\T \subset A^{a^{-1} S}$ has $1$-uniform $\{1_{F_2}\}$-extensions.

Next, we pick a set that we wish to tile. Though Lemma~\ref{lem:Contour} does not require convexity, it is natural to pick a convex set for some notion of convexity, so that if $\T$ is also TEP, the $\T$-legal patterns obtained are actually in the language of the subshift. In Section~\ref{sec:FreeGroup} we will define the tree convex sets, and the ball of radius $4$ is itself tree convex. Thus, let us tile that set. Figure~\ref{fig:Garbanzo} shows the $a^{-1}S$-contour with large black circles, and numbers show the order in which the Magnus ordering suggests we fill the rest. 

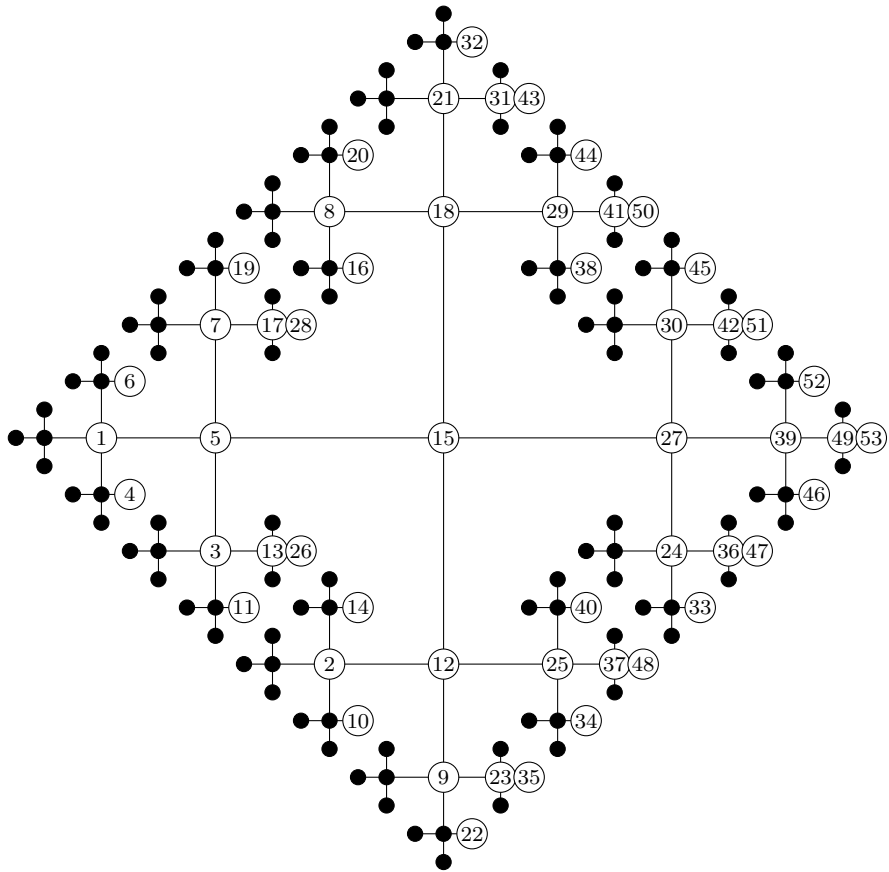
\begin{figure}
\begin{tikzpicture}
\input{garbanzo}
\end{tikzpicture}
\caption{The $a^{-1}S$-contour and the order in which the Magnus ordering suggests we fill the rest.}
\label{fig:Garbanzo}
\end{figure}

One can verify that this indeed works out, i.e.\ if the contour itself can be picked arbitrarily without introducing a tiling error, and if we ``slide'' the shape $a^{-1}S$ along the ordering, the value of each cell is uniquely determined. In particular there are $|A|^{108}$ legal configurations of this shape. \qee
\end{example}

\begin{example}
\label{ex:CAExample}
In this example we assume some familiarity with cellular automata.
Consider a surjective cellular automaton on a full shift on $\Z$, and its spacetime subshift $X \subset A^{\Z^2}$, with time increasing downward. The local rule of the cellular automaton gives a family of allowed patterns of shape $S = (\{m,m+1,\cdots, n\} \times \{0\}) \cup \{(0,-1)\}$ which has $1$-uniform $(0,-1)$-extensions.

One can define a natural measure on its spacetime subshift, namely the measure which is uniform on the rows (the rows that appear in configurations are just the limit set of the cellular automaton, which is $A^\Z$ by surjectivity). Namely, it is classical that the resulting measure on the spacetime subshift is invariant under the $\Z^2$-action. For bipermutive cellular automata this is a particular case of a TEP subshift, and indeed the measure uniform on the lines is precisely the TEP measure, by Theorem~\ref{thm:Measure} and the fact the horizontal lines are convex sets.

For a non-bipermutive surjective cellular automaton the measure constructed this way will \emph{not} always have a uniform distribution on convex sets: Let $f : \Z_3^\Z \to \Z_3^\Z$ be the cellular automaton $f(x)_i = \min(1, x_i) + x_{i+1}$. This cellular automaton is surjective because it is right-permutive. If the cellular automaton is run downward, the distribution on the convex shape $((1,0), (0,-1))$ is that the contents $(a,b)$ has probability $2/9$ if $b = a+1$ (because there are two ways to choose the contents of $(0,0)$) and $1/9$ otherwise (then necessarily $b = a$ and we must choose the symbol $0$ in position $(0,0)$).

Note also that given any set, we can find a locally legal filling of it as in the proof of Lemma~\ref{lem:Contour}, by filling its contour with respect to any ordering having $(0,-1)$ as the maximal element of $S$ and then applying the local rule to the cells in the ordering we have chosen. This does \emph{not} always produce globally legal patterns by the same example as in the previous paragraph, since $((1,0), (0,-1))$ is its own $\{(0,0),(1,0),(0,-1)\}$-contour, but $(0,2)$ is not globally legal on it. \qee
\end{example}

\subsection{Subshift and solitaire of independence}
\label{sec:Solitaire}

In this section, we restrict to TEP subshifts, i.e.\ $k$-TEP for $k = 1$.

We introduce the independence subshift, and a closure property of it we call the solitaire (of independence). This allows one to construct more sets of the type studied in the previous section, where one can pick the contents arbitrarily.

The main ``applications'' of the solitaire that we are aware of (beyond just finding independent sets, of course) would require getting our hands dirty with some discrete geometry of $\Z^d$, and this is beyond the scope of this paper. Nevertheless, we find this solitaire rather fascinating, and feel it is worth presenting in the present text.

\begin{definition}
Let $G$ be a group, 
and let $X \subset A^G$ be a subshift. 
For $y \in \{0,1\}^G$ write $\supp(y) = \{g \in G \;|\; y_g = 1\}$ for its \emph{support}, and define the \emph{(topological) independence subshift} $I(X) \subset \{0,1\}^G$ by
\[ y \in I(X) \iff \forall P \in A^{\supp(y)}: \exists x \in X: x|_{\supp(y)} = P. \]
If $\mu \in \mathcal{M}(A^G)$ is an invariant measure, let the \emph{$\mu$-independence subshift} $I_{\mu}(X) \subset \{0,1\}^G$ be the set defined by
\[ y \in I_{\mu}(X) \iff \forall B \Subset \supp(y): \mu|_B \mbox{ is uniform on } X|_B. \]
\end{definition}

It is easy to show that indeed $I(X)$ and $I_\mu(X)$ are always subshifts, and also that they are \emph{down}, meaning $(y \in I(X) \wedge \forall g \in G: y'_g \leq y_g) \implies y' \in I(X)$. It is convenient to identify binary configurations with their supports when working with the solitaire, and we move freely back and forth. For $y \in I(X)$, we call $\supp(y)$ (or $y$ itself) an \emph{independent set}.

\begin{definition}
Let $G$ be a group, and fix $T \subset S \Subset G$. We give $\{0,1\}^G$ the structure of an undirected graph, the \emph{($(S,T)$-)solitaire graph} by using the following edges: $(y, y') \in E_{S,T} \subset (\{0,1\}^G)^2$ if there exist $g \in G$ and $a, b \in gT$ such that $a \neq b$ and
\[ y|_{G \setminus \{a, b\}} = y'|_{G \setminus \{a, b\}}, \]
\[ y|_{gS \setminus \{a, b\}} = y'|_{gS \setminus \{a, b\}} = 1^{gS \setminus \{a, b\}}, \]
\[ y_a \neq y_b \wedge y'_a \neq y'_b. \]
\end{definition}


We call two configurations $y, y' \in \{0,1\}^\Z$ \emph{($(S,T)$-)solitaire connected} if they are in the same component of the ($(S,T)$-)solitaire graph, and we call edges \emph{($(S,T)$-)solitaire moves}. By default, if $T$ is not mentioned and there is a fixed $S$-UCP convexoid on $G$, we assume $T = \bigcup_g g^{-1} \angle g S$ is the set of translated lax corners. There is a strong relation between fillings of sets in the same solitaire component in a TEP subshift.

\begin{lemma}
\label{lem:OtherSwapping}
Let $G$ be a group, $S \Subset G$. Suppose $X \subset A^G$ is defined by a family $\T \subset A^S$ which has $1$-uniform $T$-extensions. Suppose $y, y' \in \{0,1\}^G$ are $(S,T)$-solitaire connected, and let $N, N'$ be the supports of $y, y'$ respectively. Then there is a unique homeomorphism $\phi : X|_N \to X|_{N'}$ such that $x|_N = P \iff x|_{N'} = \phi(P)$ for all $x \in X$.
\end{lemma}

(Observe that the symmetric difference of $N$ and $N'$ is necessarily finite by definition of a solitaire move and connectedness.)

\begin{proof}
Uniqueness is obvious once existence is proved. Existence follows from proving existence of $\phi$ for a single step (by composing the $\phi$-maps for individual steps, since ``$\iff$'' is transitive). For existence of $\phi$ for a single solitaire step, observe that $1$-uniform $T$-extensions imply that, when considering configurations $x \in X$, the contents of $x|_{gS \setminus \{a, b\}}$ for $a, b \in gT$ put up a bijection between possible contents of $x_{ga}$ and $x_{gb}$, in the sense that knowing one implies the contents of the other, given that $x|_{gS \setminus \{a, b\}}$ is known. Thus, if $gS \setminus \{b\} \subset \supp(y)$ for some $y \in I(X)$, and the configuration $y'$ is obtained from $y$ by replacing $a$ by $b$ in the support, then the homeomorphism $\phi$ simply performs this deduction.
\end{proof}

\begin{lemma}
\label{lem:Swapping}
Let $G$ be a f.g.\ group, $S \Subset G$ and fix an $S$-UCP convexoid on $G$. Suppose $X \subset A^G$ is TEP with shape $S$ (resp.\ and $\mu \in \mathcal{M}(X)$ is any invariant measure). Then $I(X)$ (resp.\ $I_\mu(X)$) is a union of connected components of the $(S,T)$-solitaire graph, where $T$ is the set of translated lax corners.
\end{lemma}

\begin{proof}
First he topological claim. It suffices to show that if $(y, y') \in E_{S,T}$ where $T$ is the set of translated lax corners, and $y \in I(X)$, then also $y' \in I(X)$. Observe that if $y$ and $y'$ are connected by solitaire moves, then their supports $N, N'$ differ in finitely many positions only. If the differences are contained in the ball $B_n$ of $G$ and $k$ is sufficiently large depending on the modulus of continuity of $\phi$, we see that $\phi$ also gives a bijection between $X|_M$ and $X|_{M'}$ for $M = N \cap B_{n+k}$ and $M' = N \cap B_{n+k}$, which are of the same cardinality. Thus
\[ y \in I(X) \implies X|_M = A^M \implies X|_{M'} = A^{M'} \implies y' \in I(X). \]
For the measure-theoretic claim, observe that bijections on finite sets preserve the uniform measure.
\end{proof}

Examples of independent sets can be produced very easily: whenever we obtain independent sets from contours, or by following any anti-shelling, we can apply the solitaire to produce more independent sets. The following example looks at the triangle shape on $\Z^2$.

\begin{example}
\label{ex:Solitaire}
Consider an arbitrary TEP subshift $X$ with the shape $S = \begin{smallmatrix} * & * \\ & * \end{smallmatrix}$. It is easy to see that the configuration $y$ with support $\Z \times \{0\}$ is in $I(X)$. This follows directly from Theorem~\ref{thm:UCE} because $\Z \times \{0\}$ is convex. Alternatively, it follows from basic theory of cellular automata by considering the spacetime subshift, similarly as in Example~\ref{ex:CAExample}. 

By applying the solitaire to this configuration $y$, one obtains a large family of independent sets. As an example, one can show that $X|_B = A^B$, where $B$ is the following subset of $\Z^2$.
\begin{center}
\pgfplotstableread{independent.cvs}{\matrixfile}
\begin{tikzpicture}[scale = 0.25]

  \pgfplotstablegetrowsof{\matrixfile} 
  \pgfmathtruncatemacro{\totrow}{\pgfplotsretval}
  \pgfplotstablegetcolsof{\matrixfile} 
  \pgfmathtruncatemacro{\totcol}{\pgfplotsretval}
  
  \pgfplotstableforeachcolumn\matrixfile\as\col{
    \pgfplotstableforeachcolumnelement{\col}\of\matrixfile\as\colcnt{%
      \ifnum\colcnt=0
        \fill[white]($ -\pgfplotstablerow*(0,\cellht) + \col*(\cellwd,0) $) rectangle +(\cellwd,\cellht);
      \fi
      \ifnum\colcnt=1
        \fill[black!50!white]($ -\pgfplotstablerow*(0,\cellht) + \col*(\cellwd,0) $) rectangle+(\cellwd,\cellht);
      \fi
    }
  }
 \draw[black!20!white] (-3, -9) grid (30,1);
\end{tikzpicture}
\end{center}
We generated this fact by applying the solitaire at random to the configuration with support $\{0,1,...,29\} \times \{0\}$. It follows that this set is solitaire-connected to a particular translate of the line $\{0,1,...,29\} \times \{0\}$, and its contents are in bijection with the contents of such a line.

For the triangle shape $S$, it is not hard to show that the solitaire-connected component of every finite-support configuration $y \in I(X)$ is finite, indeed the smallest convex set obtained by scaling the triangle and discretizing, which contains the support of the original set $\supp(y)$, cannot be increased by an application of a solitaire move. One can thus compute the entire connected component of a finite-support configuration.

The size of the solitaire-component of the set $\{0,1,...,m-1\} \times \{0,1,...,n-1\}$ is given for small $m,n$ in Table~\ref{tab:OrbitSizes}. Only the case $n = 1$ is about independent sets of course. The column $n = 1$ of the sequence looks rather exotic, and was not in the OEIS database in 2016 when we performed these calculations. Sequence A295928, added in 2017 \cite{TEPSolitaire}, has a similar-sounding definition and agrees with these entries. \qee
\end{example}

\setlength\tabcolsep{0.13cm}
\begin{table}
\centering
\begin{tabular}{|c||c|c|c|c|c|c|c|}
\hline
$m$\textbackslash$n$  & 1       & 2        & 3        & 4        & 5        & 6        & 7      \\
\hline
\hline
1                     & 1       & 3        & 16       & 122      & 1188     & 13844    & 185448 \\
\hline
2                     & 3       & 15       & 207      & 6252     & 339027   & 28920151 &        \\
\hline
3                     & 16      & 207      & 4971     & 292370   & 37248312 &          &        \\
\hline
4                     & 122     & 6252     & 292370   & 30354021 &          &          &        \\
\hline
5                     & 1188    & 339027   & 37248312 &          &          &          &        \\
\hline
6                     & 13844   & 28920151 &          &          &          &          &        \\
\hline
7                     & 185448  &          &          &          &          &          &        \\
\hline
8                     & 2781348 &          &          &          &          &          &        \\
\hline
\end{tabular}
\vspace*{0.3cm}
\caption{The size of the connected component of a rectangle of shape $m \times n$ in the solitaire with the triangle shape.}
\label{tab:OrbitSizes}
\end{table}

\begin{example}
There can be long-range dependencies in the homeomorphism $\phi$ even if the supports of $y, y'$ have a small difference, in the following sense. Consider $G = \Z^2$, $S = \{(0,0), (1,0), (0,1), (1,1)\}$ and the configurations
\[ y = \begin{smallmatrix} 
* & \cdot & \cdot & \cdot & \cdot & \cdot & \cdot & \cdot & \cdot & * \\
* & * & * & * & * & * & * & * & * & \cdot \\
  &   &   &   &   &   &   &   &   & * \\
\end{smallmatrix} \]
and
\[ y' = \begin{smallmatrix} 
  &   &   &   &   &   &   &   &   & * \\
* & * & * & * & * & * & * & * & * & \cdot \\
* & \cdot & \cdot & \cdot & \cdot & \cdot & \cdot & \cdot & \cdot & * \\
\end{smallmatrix} \]
where $*$ denotes $1$ and $\cdot$ denotes a $0$ in the convex hull of the set of $1$s.
These patterns are both easily seen to be in the connected component of the convex-support configuration
\[ \begin{smallmatrix} 
  &   &   &   &   &   &   &   &   & * \\
* & * & * & * & * & * & * & * & * & * \\
  &   &   &   &   &   &   &   &   & * \\
\end{smallmatrix}, \]
thus all three configurations are in $I(X)$ whenever $X$ is a TEP subshift with shape $S$. Now consider the bijection $\phi$ puts up between fillings of supports of $y$ and $y'$. If all coordinates but the two where the supports differ have been filled, we have a bijection between the possible fillings. It is easy to see that (in any TEP $X$ with shape $S$) this bijection (i.e.\ the dependency between the two cells where $y$ and $y'$ differ) depends on at least the top and bottom right coortinates, and may even depend on all the coordinates. \qee
\end{example}

We ask some questions about independent sets and the solitaire. The first question is a matter of linear algebra for the actual Ledrappier subshift, but seems difficult for general TEP subshifts with the triangle shape such as $X$ from Example~\ref{ex:Counts}. The second is a more open-ended question, and does not seem easy even for the Ledrappier subshift.

\begin{question}
Is the language of $I(X)$ polynomial-time verifiable for a TEP subshift $X$? 
\end{question}

\begin{question}
Can the connected components of $I(X)$ can be characterized? To what extent are the independent sets connected by moves of the solitaire? Is Lemma~\ref{lem:OtherSwapping} optimal in some sense?
\end{question}

\begin{question}
\label{q:Approachable}
Let $S \Subset \Z^2$ and $A \Subset \Z^2$ (given in unary). Is the connected solitaire-component of the configuration $y$ with support $A$ recognizable in polynomial time? What if $S$ is the Ledrappier shape? What if $A = \{0,1,...,n-1\} \times \{0\}$?
\end{question}

The experience of the author with Question~\ref{q:Approachable} is that randomly generated elements of the connected component of a line (such as the one seen in Example~\ref{ex:Solitaire}) can usually be rather easily renormalized to the original line by playing a leisurely round of the solitaire. However, we do this in an ad hoc fashion, and do not have a general algorithm.

In the case of linear TEP subshifts, the independence subshift corresponds to a standard object from matroid theory (known as a \emph{column matroid}). We recall the connection and the simple proof. Here, a \emph{matroid} on a countable set $G$ is a family $I \subset \FinSet(G)$ such that $\emptyset \in I$, $B \subset A \in I \implies B \in I$, and the \emph{augmentation property}
\[ A, B \in I \wedge |A| > |B| \implies \exists a \in A \setminus B: B \cup \{a\} \in A \]
holds.

\begin{proposition}
If $X$ is linear TEP and $\mu$ the Haar measure, then the finite supports of configurations in $I(X)$ and $I_\mu(X)$ form a matroid.
\end{proposition}

\begin{proof}
We show the finite augmentation property. Let $F$ be a finite field. Let $A, B$ be two independent sets with $|A| > |B|$. Each vector $v \in X|_{A \setminus B} \subset F^{A \setminus B}$ appears as the restriction of some vector $w_v \in X|_{A \cup B}$ such that $w_v|_{A \cap B} = 0^{A \cap B}$, by the independence of $A$. 
Define the functions
\[ \chi(v) = w_v : F^{A \setminus B} \to F^{A \cup B} \]
and
\[ \pi(w) = w_v|_{B \setminus A} : F^{A \cup B} \to F^{B \setminus A}. \]
Since $|B| < |A|$, also $|B \setminus A| < |A \setminus B|$, and thus $\pi(\chi(v)) = \pi(\chi(u))$ for some $v \neq u$. It follows that the support of $w = \chi(v) - \chi(u)$ is contained in $A \setminus B$. Choose a coordinate $i$ such that $w_i \neq 0$. Then $B \cup \{i\}$ is independent.
\end{proof}

For a general TEP subshift $X$, we suspect $I(X)$ is not always a matroid.

\section{Examples of midpointed convex geometries}
\label{sec:ExamplesOfCG}

\begin{question}
Which groups admit an invariant midpointed convex geometry?
\end{question}

In this section, we construct examples of $S$-midpointed convex geometries for arbitrary finite sets $S$. By Lemma~\ref{lem:UCP}, we obtain UCP convex geometries from this, thus the results of Section~\ref{sec:UCEandTEP} can be directly applied.

The following results are shown: For the Heisenberg group and the free group, we give invariant (fully!) midpointed convex geometries. The midpointed convex geometry of the Heisenberg group in Section~\ref{sec:Heisenberg} is a special case of a construction of Yves de Cornulier \cite{ConvexMO}. We also show that $\Z^2$ admits a midpointed invariant convex geometry that properly contains the standard one.

We show that not all torsion-free abelian groups, nor f.g. metabelian groups, admit such convex geometries. Nevertheless, for all strongly polycyclic groups and indeed for a large class of groups obtained from group extensions (such as the Baumslag-Solitar group $\Z[1/2] \rtimes \Z$) we construct $S$-midpointed invariant convex geometries for all finite sets $S$. For an (a priori) even larger class, we construct $S$-midpointed convex geometries (that are not necessarily invariant), and show that this amounts to simply ordering the group in a suitable way.

\subsection{Finitely-generated free abelian groups}

On $\Z^d$, we have the standard convex geometry. The relevant facts about this geometry were already mentioned in Section~\ref{sec:Convexity}. On $\Z$, this is the only midpointed convex geometry that is invariant (this follows easily from by Lemma~\ref{lem:ProgsInside}, proved later). On $\Z^2$, this is not the end of the story:

\begin{proposition}
There exists a invariant midpointed convex geometry on $\Z^2$ that properly contains the standard convex geometry.
\end{proposition}

\begin{proof}
Let $B$ be any set with $|B| = 3$ which is midpointed, is not convex for the standard geometry $\C$ but its $2$-subsets are, and its convex hull with respect to $\C$ has exactly one new element $\vec u$. For example, the set $B = \{(0,0), (1,1), (2, 5)\} \subset \Z^2$ has these properties; its convex hull contains only one new element $\vec u = (1,2)$. Add $B$ and all its translates to the standard convex geometry, call the new family of sets $\D$. It is easy to see that $\D$ is closed under intersections, because all sets of cardinality at most $2$ contained in $B$ are in $\C$, thus in $\D$. Every finite set is also still contained in a convex set of $\C$, thus of $\D$.

By Lemma~\ref{lem:AECE} it is enough to show prove the corner addition property for $\D$. Suppose thus $C \subsetneq D$ and both are in $\D$. If neither has a translate equal to $B$, the claim is clear. If $C = \vec v + B$ for some $\vec v \in \Z^2$, without loss of generality (by translating) we may assume $C = B$. We observe that necessarily $\vec u \in D$ because every convex set of $\C$ containing $B$ contains its convex hull, and any proper superset of $B$ in $\D$ is in $\C$. Thus, $\D \ni C \cup \{\vec u\} \subset D$ as required. If $D = \vec v + B$, the claim is easy to show using the fact all midpointed sets of cardinality $2$ contained in $B$ are in $\C$.
\end{proof}

On the other hand, in Section~\ref{sec:NoMidpointed} we will see that there is no invariant midpointed convex geometry on $\Z^2$ that properly contains the standard convex geometry and the set $B' = \{(0,0), (3,-1), (2,3)\} \subset \Z^2$, even though this set in itself is midpointed.

By a compactness argument, if there exists a invariant midpointed convex geometry on a group $G$, there exists a maximal one (under inclusion).

\begin{question}
What are the maximal extensions of the standard convex geometry of $\Z^d$? More generally, what are the maximal invariant midpointed convex geometries of $\Z^d$? What about convexoids?
\end{question}

It seems plausible that the ``pseudoconvex sets'' that can coexist with the standard ones must all be ``close'' to standard convex sets on a large scale, although we have no precise result of this form.

\subsection{Free groups}
\label{sec:FreeGroup}

First, we construct a convex geometry on a general tree. The only standard convex geometry we know on a tree is the family of geodesically convex sets. Unfortunately this family is not midpointed when considered on Cayley graphs of free groups.

Our convex geometry will instead be obtained by requiring that if the geodesic between two vertices of a convex set $C$ goes through the center of a ball, and the ball does not contain those vertices, then the ball is contained in $C$.

A \emph{tree} is a simple undirected graph $(V, E)$, $E \subset \{\{u, v\} \;|\; u, v \in V, u \neq v\}$ which is connected and does not contain a cycle. In a tree, there is a unique path of minimal length, i.e.\ a \emph{geodesic} $u = u_0, u_1, ..., u_k = v$ with $(u_i, u_{i+1}) \in E$ for all applicable $i$, between any two vertices $u, v \in V$. The vertices of the geodesic are contained in every path between $u$ and $v$. Write $\geod(u, v)$ for this path, and $d(u, v)$ for the length of this path (number of edges). It is useful to interepret $\geod(u, v)$ also topologically, as a subset of the geometric realization of the tree (as a $1$-complex), giving a meaning to unions of geodesics.

\begin{definition}
Let $T = (V, E)$ be a tree. Define the \emph{tree convex sets} $\C_T \subset \mathcal{P}(V)$ as the family of sets $C \Subset V$ such that
\[ \forall u, w, t \in C, v \in V: v \in \geod(u, w) \wedge d(v, t) < \min(d(v, u), d(v, w)) \implies t \in C. \]
\end{definition}

Observe that tree convex sets in the above sense are geodesically convex, since by taking $v \in \geod(u, w), v \notin \{u, w\}, t = v$, we have $0 = d(v, t) < 1 \leq \min(d(v, u), d(v, w))$. The converse does not hold.

\begin{theorem}
\label{thm:TreeConvex}
Let $T = (V, E)$ be a tree. Then the tree convex sets of $T$ form a convex geometry.
\end{theorem}

Before proving this, we introduce a bit of notation and prove a simple lemma. Fix a tree $T = (V, E)$ and for $u, w \in T$ write
\[ [uw] = \{u, w\} \cup \{t \in V \;|\; \exists v \in \geod(u, w): d(v, t) < \min(d(v, u), d(v, w)) \}. \]

\begin{lemma}
\label{lem:Tech}
\[ \forall u, w, t, p, s: (t \in [uw] \wedge s \in [tp]) \implies (s \in [wp] \cup [up] \cup [uw]) \]
\end{lemma}

\begin{proof}
We have five cases, depending on where the geodesic from $p$ to $t$ deviates from the triangle formed by $u, w, t$, and where the path from $s$ to the geodesic between $p$ and $t$ branches off, up to swapping the roles of $u$ and $w$. These cases are listed as (a)-(e) of Figure~\ref{fig:Cases}. We also explain them in words: define $v$ as the last common node of $\geod(t, u)$ and $\geod(t, w)$, define $r$ as the last common node of $\geod(t, p)$ and $\geod(t,v) \cup \geod(v,u) \cup \geod(v,w)$ (where the latter is interpreted as the union of the vertex sets). Let $q$ be the last common node of $\geod(s, t)$ and $\geod(s, p)$.
 
\begin{figure}
\begin{subfigure}{.18\textwidth}
\begin{center}
{\includegraphics[scale=0.3,trim={5.7cm 0 17.5cm 0},clip]{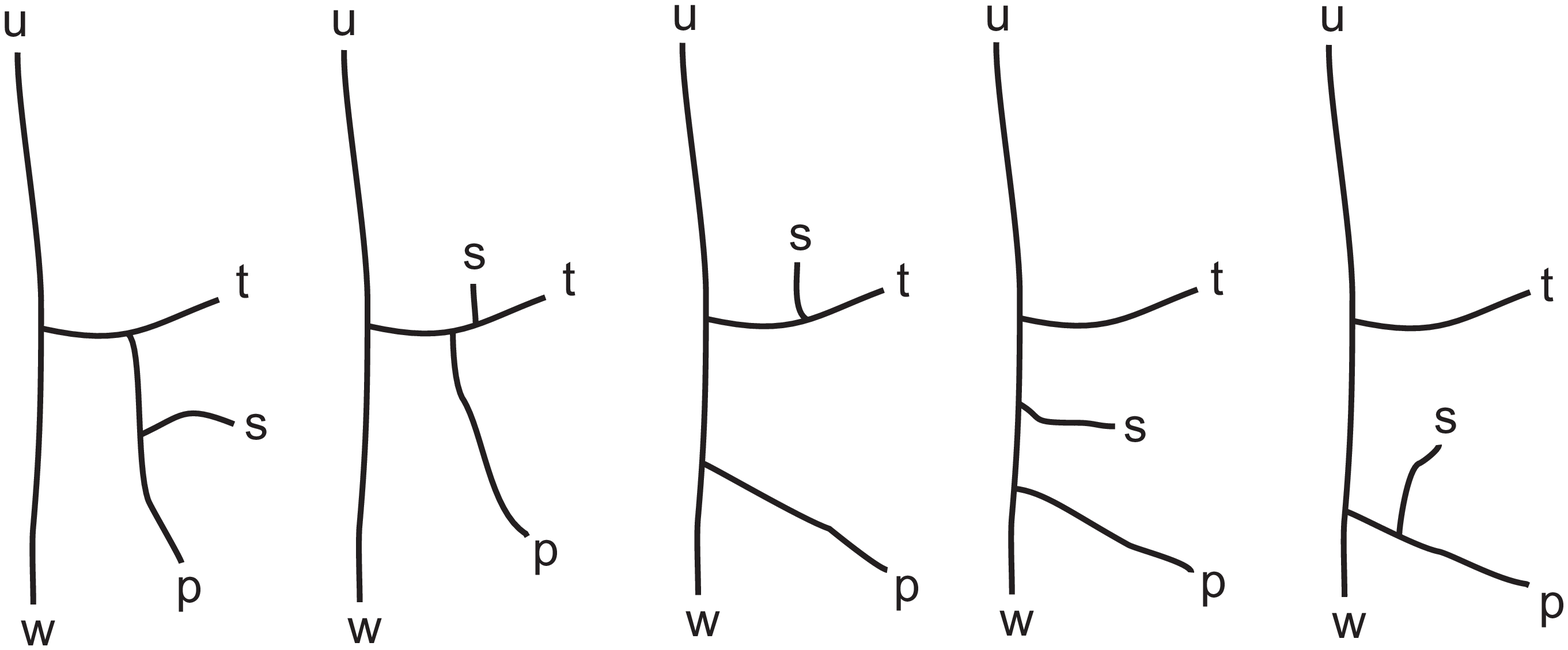}}
\end{center}
\caption{}
\end{subfigure}
\begin{subfigure}{.18\textwidth}
\begin{center}
{\includegraphics[scale=0.3,trim={0cm 0 22.7cm 0},clip]{fivecases}}
\end{center}
\caption{}
\end{subfigure}
\begin{subfigure}{.18\textwidth}
\begin{center}
{\includegraphics[scale=0.3,trim={11.8cm 0 11.2cm 0},clip]{fivecases}}
\end{center}
\caption{}
\end{subfigure}
\begin{subfigure}{.18\textwidth}
\begin{center}
{\includegraphics[scale=0.3,trim={17.3cm 0 6cm 0},clip]{fivecases}}
\end{center}
\caption{}
\end{subfigure}
\begin{subfigure}{.18\textwidth}
\begin{center}
{\includegraphics[scale=0.3,trim={23.2cm 0 0 0},clip]{fivecases}}
\end{center}
\caption{}
\end{subfigure}
\caption{The cases of Lemma~\ref{lem:Tech}.}
\label{fig:Cases}
\end{figure}

Now, (a) and (b) are the cases where $r$ lies on $\geod(t, v)$, and (a) and (b) are respectively the cases where $q$ lies on $\geod(t, r)$ or $\geod(r, p)$ respectively. The cases (c), (d) and (e) take care of the situations where $r$ does not lie on $\geod(t, v)$. By symmetry, we may assume it lies on $\geod(v, w)$. Then we have three cases where $q$ can lie, $\geod(t, v)$, $\geod(v, q)$ or $\geod(q, p)$, and these are respectively the cases (c), (d), (e). (We allow the length of the geodesics to be zero, so this covers all situations.)

By the known inequalities between lengths,
\begin{itemize}
\item in case (a), we have $s \in [uw]$,
\item in case (b), we have $s \in [up]$,
\item in case (c), we have $s \in [uw]$,
\item in case (d), we have $s \in [up]$, and
\item in case (e), we have $s \in [up]$.
\end{itemize}

Indeed, in case (a), $d(q, s) < d(q, t)$, $t \in [uw]$ and geodesics from $s$ to $u$ and $w$ go through $q$, so $s \in [uw]$. In case (b),
\begin{align*}
d(t, s) &= d(t, r) + d(r, s) \leq d(t, v) + d(r, s) \\
&< d(u, v) + d(r, s) \leq d(u, r) + d(r, s) = d(u, s)
\end{align*}
and $r$ separates $u$ and $t$ from $s$ and $p$, so we may replace $t$ by $u$ in $s \in [tp]$. In case (c), since $s \in [tp]$ we have
\begin{align*}
d(s, v) &= d(s, q) + d(q, v) < d(t, q) + d(q, v) = d(t, v) \\
& < \min(d(u, v), d(w, v)).
\end{align*}
Cases (d) and (e) are proved like case (b).
\end{proof}

We are now ready to prove Theorem~\ref{thm:TreeConvex}.

\begin{proof}[Proof of Theorem~\ref{thm:TreeConvex}]
For $C \Subset V$, define
\[ \tau(C) = C \cup \{t \in V \;|\; \exists u, w \in C, v \in \geod(u, w): d(v, t) < \min(d(v, u), d(v, w)) \}. \]
Equivalently, $\tau(S) = \bigcup_{u, w \in C} [uw]$.
It is clear that $\tau(C) = C$ for any tree convex set $C \Subset V$.

We show that $\tau$ is a closure operator. The only non-trivial thing to verify is idempotency. Suppose $s \in [tt']$ for some $t, t' \in \tau(C)$. It is enough to show that $s \in \tau(C)$. If $t, t' \in C$, we are done, so suppose $t \in \tau(C) \setminus C$. We then have $t \in [uw]$ for some $u, w \in C$. By the previous lemma, $s \in [wt'] \cup [ut'] \cup [uw]$. If $s \in [uw]$, then again $s \in \tau(C)$ and we are done, so by symmetry assume $s \in [ut']$. If $t' \in C$ we are done, so assume $t' = [u'w']$ for some $u', w' \in C$. Again by the previous lemma $s \in [u'w'] \cup [uu'] \cup [uw'] \subset \tau(C)$ as required.

Now, let us show the anti-exchange axiom. Suppose $C \in \C_T$ and $u, a \notin C$, $u \neq a$. Suppose $a \in \tau(C \cup \{u\})$. Then there exists $w \in C$ and $v \in V$ such that $d(v, a) < \min(d(v, u), d(v, w))$. We may assume $\geod(v, a), \geod(v, u), \geod(v, w)$ are edge-disjoint paths by picking the branching point $v$ at minimal distance from $a$.

Suppose now for a contradiction that we had some $c \in C$ and $b \in V$ such that $d(b, u) < \min(d(b, a), d(b, c))$ and $\geod(b, a), \geod(b, u), \geod(b, c)$ are edge-disjoint. Consider the geodesic $\geod(a, b)$, in particular where it branches off $\geod(v, a) \cup \geod(v, u) \cup \geod(v, w)$. We see that it must agree with some initial segment of $\geod(a, v)$, and then possibly some initial segment of $\geod(v, u)$ -- it is impossible for it to continue along $\geod(v, w)$ as then $\geod(b, u)$ would intersect $\geod(b, a)$.

It is impossible for $\geod(a, b)$ to branch off $\geod(a, v)$ before reaching $v$, as in this case we would have
\[ d(b, u) > d(v, u) > d(v, a) > d(b, a). \]

If $\geod(a, b)$ does reach $v$, then $b \in \geod(v, u)$, and
\[ d(b, a) = d(b, v) + d(v, a) < d(b, v) + d(v, w) = d(b, w), \]
thus $d(b, u) < \min(d(b, c), d(b, a)) < \min(d(b, c), d(b, w))$, and thus $u \in C$, a contradiction.
\end{proof}

Recall that elements of the free group with generating set $A$ are in bijection with reduced words over an alphabet $A^{\pm} = A \cup \{a^{-1} \;|\; a \in A\}$ where $A$ is the free generating set, where a word is \emph{reduced} if $aa^{-1}$ and $a^{-1}a$, $a \in A$, do not occur as subwords \cite{LySc15}. The group operation is concatenation followed by reducing the word by removing, or \emph{canceling}, such subwords $aa^{-1}, a^{-1}a$ (in arbitrary order).

\begin{theorem}
For all $n \in \N$, the free group $F_n$ admits an invariant midpointed convex geometry.
\end{theorem}

\begin{proof}
The Cayley graph of $G = F_n$ over a free generating set $A^{\pm}$ is a $2|A|$-regular tree. The tree convex sets $\C$ for this tree are defined in terms of the geodesic metric, thus are invariant under any tree automorphism, in particular the group translations.

We show that $\C$ is midpointed. Since it is invariant, it is enough to show $1_G \in \overline{\{g, g^{-1}\}}$ for any $g \in G$. Write $g$ in reduced form as a word over $A^{\pm}$ and then as $g = uvu^{-1}$ for $u$ of maximal length (where $u^{-1}$ is the word obtained by reversing the word and changing the exponent of each letter). We have $|v| \geq 1$ and $g^{-1} = uv^{-1}u^{-1}$. 

Now, $d(1_G, u) = |u|$ and $d(u, g) = |vu^{-1}| = |v^{-1}u^{-1}| = d(u, g^{-1}) > |u|$, so indeed $1_G \in \tau(\{g, g^{-1}\})$.
\end{proof}


We give an analog of Example~\ref{ex:ConvexOrdering} for the free group.

\begin{example}
The ball $C$ of radius $4$ w.r.t.\ the free generators of the free group is easily seen to be tree convex (every ball is). As in Example~\ref{ex:ConvexOrdering}, we can construct an anti-shelling by increasing our set at random by adding random elements of $C$, without breaking convexity. One such sequence is shown in Figure~\ref{fig:TreeConvexOrdering}.

\begin{figure}
\begin{center}
\begin{tikzpicture}[scale=1.05]
\input{tree_convex_ordering}
\end{tikzpicture}
\end{center}
\caption{A free group analog of Figure~\ref{fig:ConvexOrdering}. The freely choosable cells for the shape $\{1_{F_2}, a, b, a^{-1}, b^{-1}\}$ have a thicker border.}
\label{fig:TreeConvexOrdering}
\end{figure}
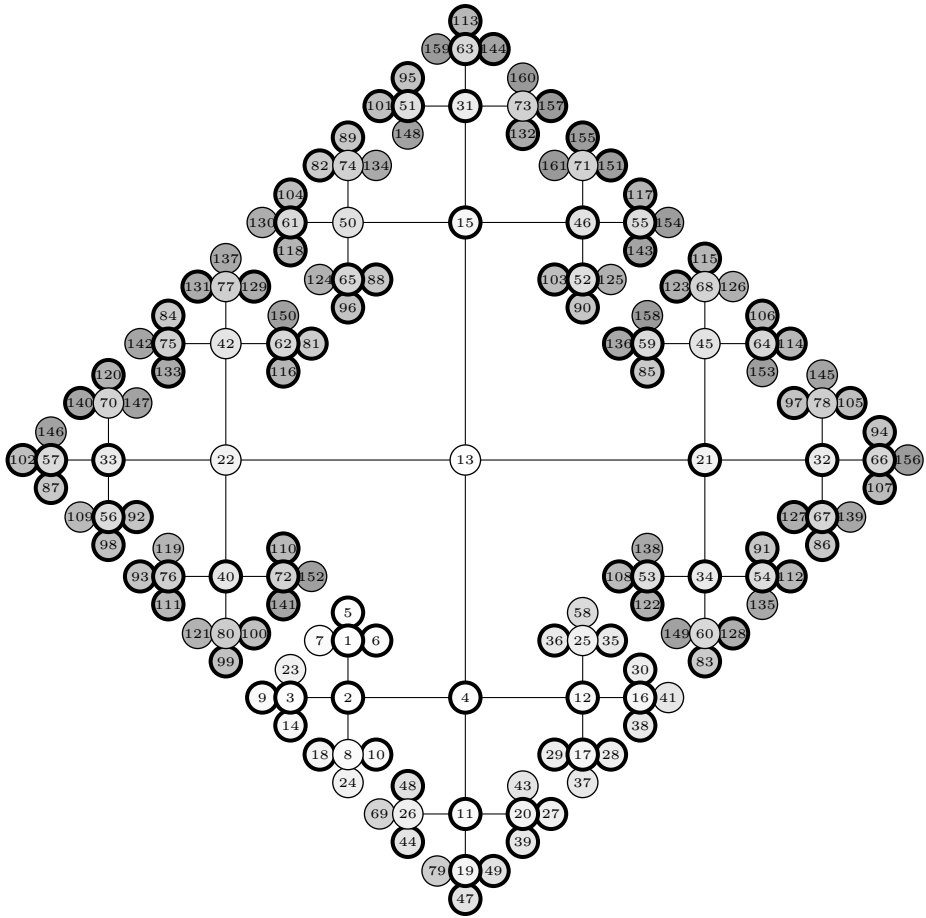

We can use this sequence to list all legal patterns of shape $C$, for any TEP subshift (with any $S$). For example, this applies to shapes $S = \{1_G, g\}$, implying that the ordering must have the property that the induced ordering of every left coset of every cyclic subgroups $\langle g \rangle$ sees an anti-shelling of $\Z$, i.e.\ a \emph{unimodal sequence} where the values first decrease and then increase. For example on the central horizontal line $\langle a \rangle \cap C$ we see $(102, 57, 33, 22, 13, 21, 32, 66, 156)$.

Using the shape $S = \{1_{F_2}, a, b, a^{-1}, b^{-1}\}$, and considering any TEP subshift with respect to that shape, one can check that values of $108$ cells can be picked freely, and the rest are determined by a translate of $S$ (of course, each by exactly one translate of $S$). This reconfirms that there are $|A|^{108}$ legal configurations, as we also saw in Example~\ref{ex:FreeContour}. Because $C$ is tree convex, they all indeed extend to a legal configuration on the entire free group.

The shape $S = \{1_{F_2}, a, b, a^{-1}, b^{-1}\}$ is rather special, and one could actually even use the convex geometry of geodesically convex shapes when working with it: pairs $C \subset D$ of geodesically closed sets with $|D| = |C|+1$ amount to adding a protruding edge to $C$, and $S$ can clearly be positioned at most one way so that it touches the new vertex. With most shapes, for example the shape $T = \{1_{F_2}, a, b, ab\}$, the geodesically convex sets would not work, i.e.\ it is possible to find an anti-shelling of $C$ with geodesically convex sets such that at some point, a single coordinate is determined by two distinct translates of $T$. For example, this happens when $ab$ is added to $\{1_{F_2}, b, a, a^2, a^2b\}$. \qee
\end{example}

\subsection{The Heisenberg group}
\label{sec:Heisenberg}

For most of our applications it suffices to prove that a group admits $S$-midpointed (invariant) convex geometries for a particular set $S$, and we construct such convex geometries for all strongly polycyclic groups in Section~\ref{sec:GroupExtensions}, which covers the Heisenberg group. Of course, having the same convex geometry work for all $S$ at once is desireable for aesthetic reasons, especially as Theorem~\ref{thm:Measure} is to some extent only as interesting as the convex geometry it is applied to. The construction in Section~\ref{sec:GroupExtensions} (in fact necessarily in that generality) does not achieve this, i.e.\ the construction depends on $S$.

\begin{question}
Do all strongly polycyclic groups admit an invariant midpointed convex geometry?
\end{question}

We asked in MathOverflow \cite{ConvexMO} whether strongly polycyclic groups, and especially the discrete Heisenberg group, admit invariant midpointed convex geometries. Yves de Cornulier proved the following theorem.

\begin{theorem}
Every finitely generated torsion-free $2$-step nilpotent group admits an invariant midpointed convex geometry.
\end{theorem}

The convex sets of this construction are very natural ones, they are obtained by embedding the group into a continuous Lie group, and taking as convex sets the intersections of the group with the images of standard convex sets of $\R^d$ in the exponential map. We show what this means concretely for the Heisenberg group, and give a proof from first principles in this special case.

\begin{proposition}
Let $H = \langle x, y, z \;|\; z = [x, y], [x, z] = [y, z] = 1 \rangle$ be the discrete Heisenberg group. Then $G$ admits an invariant midpointed convex geometry.
\end{proposition}

\begin{proof}
We recall the representation of the Heisenberg group in exponential coordinates. First, it is well-known that the discrete Heisenberg group is isomorphic to the group of matrices of the form
$\left(\begin{smallmatrix} 1 & a & c \\ 0 & 1 & b \\ 0 & 0 & 1 \end{smallmatrix}\right)$
where $a, b, c \in \Z$. We can see it as a (cocompact) lattice in the \emph{continuous Heisenberg group} obtained by replacing $a, b, c$ by real numbers.

The Lie algebra of the Heisenberg group can be identified with real matrices where the diagonal and subdiagonal are zero, and the exponential map amounts to
\[ \exp\left(\begin{matrix} 0 & a & c \\ 0 & 0 & b \\ 0 & 0 & 0 \end{matrix}\right) = \left(\begin{matrix} 1 & a & c + \frac{ab}2 \\ 0 & 1 & b \\ 0 & 0 & 1 \end{matrix}\right), \]
which is clearly bijective. Conjugating the group operation through this map, we obtain that the continuous Heisenberg group can be seen as $\R^3$ with the following group operation
\[ (a, b, c) * (a', b', c') = (a + a', b + b', c + c' + \frac{ab' - a'b}2). \]
(This is also a special case of the Baker--Campbell--Hausdorff formula.)

Conjugated through the exponential map, the discrete Heisenberg group is the subgroup where $a$, $b$ and $c + \frac{ab}2$ are integers, i.e.\ the set $H = \{(a, b, c) \in \Z^2 \times \frac12\Z \;:\; c \in \Z \iff 2 | ab \}$. The group $(H, *)$ may be called the representation of the discrete Heisenberg group in exponential coordinates. It is a cocompact lattice in the continuous Heisenberg group $(\R^3, *)$.

Now, let $\C \subset \mathcal{P}(\R^3)$ be the family of compact convex sets in the standard sense of $\R^3$. Observe that, apart from the restriction that the sets $C \in \C$ be finite, $\C$ satisfies our axioms of a convex geometry. Thus it is easy to see that $\D = \{C \cap H \;|\; C \in \C\} \subset \mathcal{P}(H)$ is a convex geometry.

To see that $\D$ is invariant with respect to $*$, observe that $\vec v \mapsto \vec u * \vec v$ is affine in $\vec v$ for fixed $\vec u$. Thus, if $D \in \D$ then $D = C \cap H$ for some $C \in \C$, and $\vec u * C \subset \R^3$ is convex as an affine image of a convex set. If $\vec u \in H$, we have $\vec u * D = (\vec u * C) \cap H$ since $H$ is invariant under $H$-translations (since it is a subgroup), thus we have $\vec u * D \in \D$.

To see that $\D$ is midpointed, we observe that (for example by direct computation) every element of $\R^3$ has a unique square root with respect to $*$, and this turns out to be $\sqrt{(a, b, c)} = (a/2, b/2, c/2)$. Thus the claim follows since convex sets of $\R^3$ are midpointed.
\end{proof}

\begin{example}
Consider the Heisenberg group $(H, *)$ in exponential coordinates as in the above proof. Clearly
\[ K = \langle (0,1,0), (0,0,1) \rangle = \{0\} \times \Z^2 = H \cap (\{0\} \times \R^2), \]
so the subgroup $K$ is convex. Suppose now that $S \Subset H$ and $\forall g \in H: gS \not\subset K$. Then Theorem~\ref{thm:Measure} implies that every TEP subshift $X \subset A^H$ with shape $S$ admits an invariant measure that samples the contents of $K$ from the uniform Bernoulli distribution of $A^K$.

In the special case of a shape satisfying $S \subset (\{0\} \times \Z^2) \cup \{(1,0,0)\}$, one can easily deduce the statement of Theorem~\ref{thm:UCE} that the restriction to $K$ is full, i.e. $X|_K = A^K$, from a cellular automata style argument. We sketch this argument: Given any values for the subgroup $K$, the values in the coset $(\{0\} \times \Z^2) * (1,0,0)$ are obtained by first applying a (linear) shear map and then a classical two-dimensional cellular automaton rule. That rule is totally extremally permutive in the sense of \cite{SaTo13b}, and thus surjective, thus it is also surjective when composed by the shear map. A compactness argument shows that the restriction to $\langle (0,1,0), (0,0,1) \rangle$ is indeed full. \qee
\end{example}

\subsection{Groups without midpointed convex geometries}
\label{sec:NoMidpointed}

\begin{lemma}
\label{lem:ProgsInside}
Let $G$ be a group and let $\C \subset \FinSet(G)$ be a midpointed convexoid. Then for all $C \in \C$ and for all $g,h \in G$, if $g \in C$ and $gh^n \in C$ for $n > 0$, then $gh^i \in C$ for all $i \in \{1, ..., n-1\}$.
\end{lemma}

\begin{proof}
If $n = 1$ there is nothing to prove. The full claim clearly follows from the claim for $n$ at most equal to the order of $h$. 
Suppose then that $n$ is at most the order of $h$, so the elements $gh^i$ are distinct for $i \in \{1, ..., n-1\}$. Suppose the lemma is false and $D \supset \{g, gh, gh^2, ..., gh^n\}$ and pick an anti-shelling from $C$ to $D$. Let $E$ be the last set in this sequence which does not yet contain all of $\{g, gh, gh^2, ..., gh^n\}$. Then $E$ is convex and $gh^i \notin E$ for a unique $i$, meaning $gh^i \notin \overline{\{gh^{i-1}, gh^{i+1}\}}$, contradicting midpointedness.
\end{proof}

\begin{proposition}
If $G$ is not torsion-free then it does not admit any midpointed convexoid.
\end{proposition}

\begin{proof}
Suppose $h \in G$ satisfies $h^n = 1_G$ for $n \geq 2$. By the previous lemma, any convex set containing an element $g$ would have to already contain $g \langle h \rangle$. But every convexoid contains an element of cardinality one.
\end{proof}

Say $g \in G$ has \emph{infinitely many roots} if for infinitely many $n \in \N$ there exist $h \in G$ such that $h^n = g$.

\begin{proposition}
If $G$ is a group where some element $g \in G \setminus \{1_G\}$ has infinitely many roots. Then $G$ does not admit a midpointed convexoid.
\end{proposition}

\begin{proof}
If $G$ is not torsion-free this follows from the previous proposition, so suppose $G$ is torsion-free. Suppose there is a midpointed convexoid on such $G$ and let $C \supset \{g, g'\}$ be any convex set. By torsion-freeness, the solutions $h$ to $h^n = g^{-1}g'$ are distinct for distinct $n \in \N$, so since there are infinitely many roots for $g^{-1}g'$, we can find an equation $h^n = g^{-1}g'$ such that $gh \notin C$. We have $g \in C, gh \notin C, h^n = g' \in C$, contradicting Lemma~\ref{lem:ProgsInside}.
\end{proof}

\begin{example}
The abelian groups $(\Z[\frac1n], +)$ do not admit any midpointed convex geometries. It follows that the finitely-generated torsion-free metabelian groups $\Z[\frac1n] \rtimes \Z$ (where $\Z$ acts by multiplication by $n$) do not admit midpointed convex geometries either. No \emph{divisible group} (meaning all elements have roots of all orders) admits a midpointed convex geometry, for example $(\Q, +)$ does not (of course it also contains $(\Z[\frac1n], +)$). \qee
\end{example}

\begin{example}
There is no midpointed convexoid on $\Z^2$ which contains both the standard convex geometry and the set $B' = \{(0,0), (3,-1), (2,3)\} \subset \Z^2$. Suppose there were one, call it $\C$. The standard convex hull of $B'$ is $C = B' \cup \{(1,0), (2,0), (1,1), (2,1), (2,2)\}$, and there is an anti-shelling from $B'$ to $C$ since these sets are convex in $\C$. It is easy to see that we must first add $(1,0)$ to $B'$ as every other choice breaks midpointedness. After adding $(1,0)$, every possible choice contradicts Lemma~\ref{lem:ProgsInside}. \qee
\end{example}

\subsection{Strongly polycyclics, direct limits, group extensions}
\label{sec:GroupExtensions}

\begin{lemma}
\label{lem:Subgroup}
Let $H \leq G$ and $S \Subset H$. Then $G$ admits an invariant $S$-midpointed convex geometry if and only if $H$ does.
\end{lemma}

\begin{proof}
By Lemma~\ref{lem:ToSubset}, if $\C$ is an invariant $S$-midpointed convex geometry on $G$, then $\{C \cap H \;|\; C \in \C\}$ is a convex geometry on $H$. Its $H$-invariance and $S$-midpointedness are easy to check. If $\C$ is an invariant $S$-midpointed convex geometry on $H$, then we obtain one on $G$ by applying the closure operation independently in left $H$-cosets, in the sense of Lemma~\ref{lem:FromPartition}, where on $gH$ we of course pick $g\C$ as the convex geometry. The $S$-midpointedness of the resulting convex geometry $\D$ is easy to check.

For $G$-invariance of $\D$, pick left coset representatives $(g_i)_i$ for $H$ and observe that for $D \in \D$, 
\[ g D \cap g_i H = g (g_jH \cap D) \cap g_i H \]
for some $g_j$, because the natural action $G \curvearrowright G/H$ permutes the left cosets. If this set is empty, we are done. Otherwise, write $g g_j = g_k h$ for $h \in H$ so
\[ g (g_jH \cap D) \cap g_i H = g_k h H \cap g D \cap g_i H = g_k H \cap g D \cap g_i H \neq \emptyset \]
if and only if $g_k = g_i$, and then $g D = g_i h D \in g_i \C$ because $\C$ is $H$-invariant.
\end{proof}

\begin{lemma}
\label{lem:Locally}
Let $G = \bigcup G_n$ for an increasing union of groups $G_i \leq G_{i+1}$. Suppose each $G_i$ admits an invariant $S$-midpointed convex geometry for each $S \Subset G_i$. Then $G$ admits an invariant $S$-midpointed convex geometry for each $S \Subset G$.
\end{lemma}

\begin{proof}
Let $S \Subset G$. Then $S \Subset G_i$ for some $i$, and $G_i$ admits an invariant $S$-midpointed convex geometry. Apply the previous lemma with $H = G_i$.
\end{proof}

The following covers for example the rationals and dyadic rationals, which do not admit any midpointed convex geometry by the previous section.

\begin{proposition}
Let $G$ be a torsion-free abelian group. Then $G$ admits an invariant $S$-midpointed convex geometry for each $S \Subset G$.
\end{proposition}

\begin{proof}
A torsion-free abelian group is by definition locally torsion-free finitely-generated abelian, i.e.\ locally-($\Z^d$ for some $d$). The group $\Z^d$ admits an invariant midpointed convex geometry, a fortiori it admits an invariant $S$-midpointed convex geometry for each $S \Subset G_i$, and the previous lemma applies.
\end{proof}

\begin{theorem}
\label{thm:GroupExtensions}
Let $1 \rightarrow K \rightarrow G \overset{\pi}{\rightarrow} H \rightarrow 1$ be exact. If $K$ admits invariant $S$-midpointed convex geometries for all finite sets $S \Subset K$, and $H$ is torsion-free abelian, then $G$ admits invariant $S$-midpointed convex geometries for all finite sets $S \Subset G$. 
\end{theorem}

\begin{proof}
We first observe that it is enough to show this for $H = \Z$: By Lemma~\ref{lem:Locally} it is enough to show this for all finitely-generated subgroups, thus we may assume $G$, and thus also $H$, is finitely-generated, i.e.\ $H \cong \Z^d$. We prove the csae $H = \Z$, from which the general case follows by induction. Since $\Z$ is free, the extension splits and we may suppose $G = K \rtimes \Z$ with $K \triangleleft G$ a normal subgroup.

Let $S \Subset G$, which we may assume generates $G$. We construct an invariant $S$-midpointed convex geometry on $G$.  We may suppose $\pi(S) \subset \N$ by possibly replacing some elements of $S$ with their inverses. On $K$, fix an $(S \cap K)$-midpointed invariant convex geometry. As convex sets pick sets $C \Subset G$ with the following two properties:
\begin{itemize}
\item for all $g \in G$, $g^{-1}(gK \cap C) \subset K$ is convex in $K$, and
\item if $c, e \in C$, $d \in G$, $\pi(c) < \pi(d) < \pi(e)$ and $c^{-1}d \in S$, then $d \in C$.
\end{itemize}

Invariance of this family is easy to show. The first property is an alternative description of the construction in Lemma~\ref{lem:Subgroup}, se we can talk about convex sets on cosets $gK$ (which are also easily seen to also be convex in $G$ in the sense of the previous definition). We claim that this gives a convex geometry. It is clear that the empty set is convex, and it is easy to show that the intersection of two convex sets is convex. We show the anti-exchange axiom. By Lemma~\ref{lem:AECE}, it is enough to show the corner addition property. Let thus $C \subsetneq D$, $C, D \in \C$.


Suppose first $(\pi^{-1}(\pi(C)) \cap D) \setminus C \neq \emptyset$, then pick any element $a$ of this set which maximizes $\pi(a)$, and is such that $(C \cap aK) \cup \{a\}$ is convex in $aK$ with respect to the convex geometry of $K$ (this is possible since the convex geometry of $K$, thus that of $aK$, has the corner addition property). Then $C \cup \{a\}$ is convex in $G$, i.e.\ the two properties stated above hold. Namely, the first property was explicitly enforced. Suppose the second property fails for some triple $c,d,e$. Then we must have $c = a$, since the choices of $e$ only deal with the $\pi$-projection, which was not changed by the addition of $a$. However, since $\pi(a)$ was taken to be maximal and $D$ is convex, we must in fact have $d \in C$, a contradiction. Thus, the second property holds.

Suppose then that $\pi^{-1}(\pi(C)) \cap D = C$. Observe that it follows from the second property and our choice of $S$ that all convex sets of $G$ have intervals as their $\pi$-projections. Thus $\pi(C), \pi(D) \subset \Z$ are intervals, and we are in the case where $C$ and $D$ agree when restricted to the preimage of the interval $\pi(C)$ in $G$. If $\min (\pi(D)) < \min(\pi(C))$, then pick any element $a \in D \setminus C$ which maximizes $\pi(a)$ under the constraint $\pi(a) < \min \pi(C)$, i.e.\ any element that extends the $\pi(C \cup \{a\})$ by adding a new minimum. We must have $\pi(a) = \min(\pi(C)) - 1$ since $\pi(D)$ is an interval. Then $C \cup \{a\}$ is convex -- singletons are closed in the convex geometry of $K$, thus in that of the coset $aK$, so the only problem could be that the second property fails for some triple $c, d, e$ with $c = a$, and $e \in C$. Since $D$ is convex and $e \in C$, we have $d \in D$, and thus in fact $d \in C$ since $\pi(d)$ is in the interval where $C$ and $D$ agree. Thus the second property must in fact hold.

If $\max (\pi(D)) > \max(\pi(C))$, then pick any element $a \in D \setminus C$ which minimizes $\pi(a)$ under the constraint $\pi(a) > \max \pi(C)$. Again $C \cup \{a\}$ is convex -- singletons are closed in the convex geometry of $K$, so the only problem could be that the second property fails for some triple $c, d, e$ with $e = a$. Since $D$ is convex, we have $d \in D$, thus $d \in C$ because $\pi(d)$ is again in the interval where $C$ and $D$ agree. Thus the second property must hold in this case as well.

Finally, we show $S$-midpointedness. Suppose $C$ is convex, and $gh, gh^{-1} \in C$ with $h \in S$. If $\pi(h) = 0$, $g \in C$ follows from the $S$-midpointedness of the convex geometry of $K$ and the first property of our convex sets. If $\pi(h) > 0$, then consider the triple $c = gh^{-1}, d = g, e = gh$. We have $\pi(c) < \pi(d) < \pi(e)$, $c, e \in C$, and $c^{-1}d = h \in S$. Thus, $d \in C$ by the second property of convexity, and thus our convex geometry is $S$-midpointed.
%
\end{proof}


\begin{corollary}
\label{cor:Polycyclics}
Let $G$ be a strongly polycyclic group. Then, for every finite set $S \Subset G$, there exists an invariant $S$-midpointed decidable convex geometry.
\end{corollary}

\begin{proof}
By definition, a strongly polycyclic group is obtained by repeated $\Z$-extensions $1 \to K \to G \to \Z \to 1$, and the previous theorem applies. Strongly polycyclic groups are well-known to have decidable word problems, and it is easy to see from the proof that the convex geometries obtained are decidable.
\end{proof}

The following is direct from the previous theorem, Theorem~\ref{thm:UCE} and Lemma~\ref{lem:UCP}.

\begin{corollary}
Let $G$ be a strongly polycyclic group $S \Subset G$, $\T \subset A^S$ have $k$-uniform $S$-extensions for some $k$, and $X$ be the corresponding $k$-TEP subshift. Then $X$ has a decidable language.
\end{corollary}

\begin{example}
The above theorem applies to $\Z^d$, giving another construction of $S$-midpointed convex geometries for all $S$ for these groups (but these convex geometries are not midpointed). It also implies that the Baumslag-Solitar groups $\Z[\frac1n] \rtimes \Z$ (which by the previous section do not admit any midpointed convex geometries) admit invariant $S$-midpointed convex geometries for all finite sets $S$. The same is true for the wreath product $\Z \wr \Z$. By the results of Section~\ref{sec:FreeGroup} we have that $F_2 \times \Z$ admits $S$-midpointed invariant convex geometries for all finite sets $S$. In each case, it is clear from the construction that the convex geometries are decidable, giving decidability of languages of $k$-TEP subshifts. \qee
\end{example}

The following theorem summarizes our results about invariant $S$-midpointed convex geometries.

\begin{theorem}
\label{thm:ClassTransInv}
Let $\mathcal{G}$ be the smallest family of groups such that
\begin{itemize}
\item free groups and torsion-free abelian groups are in $\mathcal{G}$,
\item $\mathcal{G}$ is closed under direct unions,
\item $\mathcal{G}$ is closed under taking subgroups
\item $\mathcal{G}$ is closed under group extensions by actions of torsion-free abelian groups.
\end{itemize}
Then every group $G \in \mathcal{G}$ admits an invariant $S$-midpointed convex geometry for each $S \Subset G$.
\end{theorem}

\begin{question}
Which groups admit an invariant $S$-midpointed convex geometry for each $S \Subset G$? What about convexoids?
\end{question}

\subsection{Non-invariant convex geometries}
\label{sec:NonTranslationInvariant}

In this section, we show that constructing not necessarily invariant $S$-midpointed convex geometries is equivalent to ordering the group in a way that avoids lacking midpoints. This is the bare minimum needed to apply Theorem~\ref{thm:UCE} and Theorem~\ref{thm:Measure}. As an application, we show that groups admitting such convex geometries are closed under group extensions, and give some examples not (obviously) covered by our invariant constructions.

\begin{definition}
Let $<$ be a total order on a subset $A$ of a group $G$, and let $S \subset G$. We say $<$ is an \emph{$S$-midpointed order} if $\forall g \in A, h \in S: \{gh^{-1}, gh\} \subset A \implies g \leq \max(gh^{-1}, gh)$. It is \emph{midpointed} if $S = G$.
\end{definition}

\begin{lemma}
\label{lem:IffOrdering}
Let $G$ be a countably infinite group and $S \subset G$. Then $G$ admits an $S$-midpointed convex geometry if and only if it admits an $S$-midpointed ordering of order type $\omega$.
\end{lemma}

\begin{proof}
From an $S$-midpointed convex geometry $\C$, we obtain a midpointed ordering by taking any unbounded anti-shelling and collecting the corners. Namely this gives an ordering $G = \{g_1, g_2, ...\}$ such that $\{g_1, g_2, ..., g_n\}$ is convex for all $n$, and the fact $\C$ is $S$-midpointed directly translates into this ordering being $S$-midpointed.

Conversely, if $G$ admits an $S$-midpointed ordering of order type $\omega$, then the lower sets of this order are easily seen to yield an $S$-midpointed convex geometry.
\end{proof}

\begin{definition}
Let $G$ be a countably infinite set and for each $i \in \N$ let $F_i$ be a subset of $G$. We say that two orderings of $G$ are \emph{$(F_i)_i$-consistent} if for all $i$, the orderings they determine on $F_i$ are equal.
\end{definition}

\begin{lemma}
Let $G$ be a countably infinite set, let $(G, <)$ be a total well-order, and let $(F_t)_t$ be a family of finite subsets such that each $g \in G$ appears in finitely many of the $F_i$. Then there exists an ordering $\prec$ of $G$ of order type $\omega$, which is $(F_t)_t$-consistent with $<$.
\end{lemma}

\begin{proof}
Let $G = \{g_1, g_2, g_3, ...\}$ be some well-order of order type $\omega$. Construct another order $G = \{h_1, h_2, h_3, ...\}$ by always setting $h_i = g_j$ where $j$ is minimal such that for all $F_t \ni g_j$, all the elements $g_k \in F_t$ which appear before $g_j$ in the order $<$ already appear in $\{h_1, h_2, ..., h_{i-1}\}$.

First, we observe that this process never stops, i.e.\ that an infinite sequence $(h_i)_i$ is indeed constructed: the first element $g_j$ of the initial $\omega$-segment of $<$ which has not yet been added is always available for adding, since all its $<$-predecessors have been added.

Next, we claim that $(h_i)_i$ indeed enumerates $G$. Suppose not, and let $g_j$ be minimal in the order $<$ such that $g_j$ is never added as $h_i$. After finitely many enumeration steps, we never add $g_{j'}$ with $j' < j$ as $h_i$ (because there are finitely many such $j'$), thus the only possible reason $g_j$ is not enumerated on a particular step $i$ is that for some $F_t \ni g_j$, some $g_k \in F_t$ that appears before $g_j$ in the order $<$ does not appear as $\{h_1, h_2, ..., h_{i-1}\}$. But all such $g_k$ (of which there are only finitely many) are eventually added into the order, since $g_j$ was taken to be $<$-minimal, a contradiction. Thus, the sequence $(h_i)_i$ is an ordering of $G$ with order type $\omega$.

Next, we show that $(h_i)_i$ is $(F_t)_t$-consistent. Suppose not, and for some $t$, we have $h_i, h_{i'} \in F_t$ with $i < i'$ but $h_i > h_{i'}$. This means that at step $i$, we enumerated $h_i$ even though $h_{i'} < h_i$ had not yet been enumerated, contrary to the process.
\end{proof}

\begin{lemma}
\label{lem:Reordering}
Let $S \Subset G$ be finite. Then a countably infinite subset $A \subset G$ of a group admits an $S$-midpointed well-order if and only if it admits an $S$-midpointed order of order-type $\omega$.
\end{lemma}
 
\begin{proof}
The non-trivial direction is to show that the order type of a a midpointed well-order can be changed to $\omega$. We observe that, setting $F_{g,h} = \{gh^i \;|\; i \in \{-1,0,1\}\} \cap A$ for $g \in A, h \in S$, we obtain a countable family of finite sets such that each $a \in A$ appears in only finitely many of them. Any ordering that is $(F_{g, h})_{g, h}$-consistent with an $S$-midpointed order is an $S$-midpointed order. Thus the claim follows from the previous lemma.
\end{proof}

\begin{lemma}
\label{lem:UselessLemma}
Let $G$ be a countable group, $1 \rightarrow K \rightarrow G \overset{\pi}{\rightarrow} H \rightarrow 1$ be an exact sequence and $S \Subset G$ finite. If $H$ admits a $\pi(S)$-midpointed convex geometry, and for all $g \in G$, $K$ admits an $(S \cap K)$-midpointed convex geometry, then $G$ admits an $S$-midpointed convex geometry.
\end{lemma}

\begin{proof}
We argue with orderings, using Lemma~\ref{lem:IffOrdering}. Pick a section $h_1, h_2, ...$ for $\pi$, ordered according to a $\pi(S)$-midpointed ordering of $H$, and for $k, k' \in K$, order $G$ by $h_i k < h_j k'$ when $i < j$ or $i = j$ and $k < k'$ in the $(S \cap K)$-midpointed ordering of $K$. The order type is $\omega^2$.

This ordering is $S$-midpointed: If $g > \max(gh^{-1}, gh)$ in this ordering, then we must have $\pi(h) = 1_H$, since we ordered the cosets according to a $\pi(S)$-midpointed ordering of $H$. But if $\pi(h) = 1_H$, then $h \in S \cap K$, and $g > \max(gh^{-1}, gh)$ contradicts the fact we used an $(S \cap K)$-midpointed ordering of $K$ on the individual cosets.

Since $\omega^2$ is a well-order, the result follows from Lemma~\ref{lem:Reordering}.
\end{proof}

In particular, by $\forall$-quantifying the sets $\pi(S)$ and $S \cap K$, we get a version of Theorem~\ref{thm:GroupExtensions} in the non-invariant setting. Note that here there are no restrictions on the extensions.

\begin{theorem}
Let $1 \rightarrow K \rightarrow G \overset{\pi}{\rightarrow} H \rightarrow 1$ be exact. If $K$ and $H$ admit $S$-midpointed convex geometries for all finite sets, then so does $G$.
\end{theorem}

By the results of Section~\ref{sec:FreeGroup} and the previous theorem we have for example that the wreath product $\Z \wr F_2$ (the semidirect product where the free group acts on $\Z^{F_2}$ by translation) admits $S$-midpointed convex geometries for all finite sets $S$. Though we include no precise decidability statements, it is clear that this convex geometry can be constructed so that the convex sets are a decidable family, thus TEP subshifts on this group have decidable languages. This is not, at least for any obvious reason, covered by the results of the previous section.

We have the following non-invariant analog of Lemma~\ref{lem:Locally} (with a similar proof).

\begin{lemma}
Let $G = \bigcup G_n$ for an increasing union of groups $G_i \leq G_{i+1}$. Suppose each $G_i$ admits an $S$-midpointed convex geometry for each $S \Subset G_i$. Then $G$ admits an $S$-midpointed convex geometry for each $S \Subset G$.
\end{lemma}

The following theorem summarizes our results about non-invariant $S$-midpointed convex geometries.

\begin{theorem}
\label{thm:Class}
Let $\mathcal{H}$ be the smallest family of groups such that
\begin{itemize}
\item free groups and torsion-free abelian groups are in $\mathcal{H}$,
\item $\mathcal{H}$ is closed under direct limits,
\item $\mathcal{H}$ is closed under subgroups,
\item $\mathcal{H}$ is closed under group extensions.
\end{itemize}
Then every group $G \in \mathcal{H}$ admits an $S$-midpointed convex geometry for each $S \Subset G$.
\end{theorem}

\begin{question}
Which groups admit an $S$-midpointed convex geometry for each $S \Subset G$?
\end{question}

\section{Kaplansky's and Gottshalk's conjectures}
\label{sec:Kaplansky}

The definition of a TEP subshift depends on a particular convexoid, but because the translated lax corners of a shape $S$ are always contained in it, assuming $k$-uniform $S$-extensions we have uniform extensions in the translated lax corners with respect to any convexoid. Thus, when discussing groups without a fixed convexoid, it is natural to redefine TEP to mean an SFT defined by allowed patterns $\T \subset A^S$ having $k$-uniform $S$-extensions, and we take this approach in the present section.

By Proposition~\ref{prop:AtLeastPoints}, the existence of $S$-UCP convex geometries for all finite sets $S$ (equivalently, $S$-midpointed convex geometries for all finite sets $S$) implies that all TEP subshifts with shape $|S| \geq 2$ have more than one configuration, and all $k$-TEP subshifts with $k \geq 2$ are uncountable.

\begin{question}
\label{q:AtLeastA}
Are TEP subshifts with shape $|S| \geq 2$ nonempty on all torsion-free groups? Do the always have at least $|A|$ configurations (where $A$ is the alphabet)?
\end{question}

We do not know the answer, but we show that proving that they always have at least two configurations should be difficult if it is true: in the linear case, whether a TEP subshift has at least two configurations is directly related to Kaplansky's conjectures and Gottshalk's surjunctivity conjecture.

An element $p \in R$ of a ring $R$ admits a \emph{weak inverse} if $pqp = p$ for some $q \in R$. The element $p \in F[G]$ defines a linear TEP subshift $X_p = \{x \in F^G \;|\; xp = 0\}$, where $x \in F^G$ is identified with the formal sum $\sum_{g \in G} x_g \cdot g$. It also defines a linear map $f_p : F^G \to F^G$ by $f_p(x) = x \cdot p$, $X_p = \ker f_p$. This is clearly continuous and shift-commuting for the left shift $gx_h = x_{g^{-1}h}$. We say $p$ is injective (resp. surjective, bijective) if this map is injective (resp. surjective, bijective).

We name some statements about $p \in F[G]$ with $|\supp(p)| \geq 2$, for $F$ a field, $G$ a group.

\begin{itemize}
\item $T \iff \;$ ``$X_p$ contains at least $|F|$ configurations''
\item $O \iff \;$ ``$f_p$ is not injective''
\item $U \iff \;$ ``$p$ does not have a right inverse in $F[G]$''
\item $S \iff \;$ ``$f_p$ injective implies $f_p$ surjective''
\item $W \iff \;$ ``$p$ does not have a weak inverse''
\end{itemize}

The property $T$ corresponds to the second part of Question~\ref{q:AtLeastA}.

Kaplansky's unit conjecture states that $U$ holds for any torsion-free group $G$. Gottshalk's surjunctivity conjecture \cite{Go73} implies that $S$ holds universally (and indeed even for non-linear cellular automata in place of $f_p$). It is not clear to us what the strength of $W$ is, though clearly it implies $U$.  Note that all of these statements are true for $G = \Z$, $F$ any field and $p \in F[G]$ arbitrary (with support size at least two). We record some connections between these.

\begin{proposition}
Given any group ring element $p \in F[G]$, $F$ a finite field and $G$ torsion-free, $X_p$ has at least $|F|$ configurations if and only if it has at least two configurations if and only if $f_p$ is not injective.
\end{proposition}

\begin{proof}
Since $f_p$ is linear and $X_p = \ker f_p$, $f_p$ is injective if and only if $X_p = 0$ if and only if $\dim X_p \not\geq 1$.
\end{proof}

\begin{lemma}
For any fixed group $G$, field $F$ and element $p \in F[G]$, we have
\[ W \implies O \iff T \implies U \]
\[ S \wedge U \implies T \]
\end{lemma}

\begin{proof}
The equivalence $O \iff T$ is the proposition above. For the implication $W \implies O$, we show the contrapositive $\neg O \implies \neg W$. If $f_p$ is injective then $f_p : F^G \to f_p(F^G)$ is bijective so by compactness it admits a continuous inverse $g : f_p(F^G) \to F^G$ which is automatically shift-invariant and linear, thus we can write $g(x) = x \cdot q$ for some $q \in F[G]$ (valid on the image of $f_p$). We have $x \cdot pqp = f(g(f(x))) = f(x) = x \cdot p$ for all $x \in F^G$, in particular by applying this to the configuration with $1 \in F$ at identity and $0 \in F$ elsewhere (so the formal series $1 \cdot 1_G$) we obtain $pqp = p$, and $W$ does not hold.

Next we show $O \implies U$, again let us show the contrapositive $\neg U \implies \neg O$ instead. Suppose $pq = 1_G$. Then $f_q(f_p(x)) = xpq = x$, so $f_p$ is injective, that is, $\neg O$.

Suppose then $S \wedge U$. We show that $O$ holds (since $O \iff T$). If $\neg O$, then $f_p$ is one-to-one, and $S$ then implies is it also bijective. By the above proof of $\neg O \implies \neg W$, we have $pqp = p$ for some $q \in F[G]$. Then $f_q$ must in fact be the inverse of $f_p$, so we have $pq = 1$.
\end{proof}

We restate the above observations (except the one about $W$) in words:

\begin{proposition}
Consider any group ring element $p \in F[G]$ with support size at least two, $F$ a finite field and $G$ a group.
If $X_p$ has at least two legal configurations, then Kaplansky's unit conjecture holds for $p$.
If Gottshalk's surjunctivity conjecture holds for $f_p$ and Kaplansky's unit conjecture holds for $p$, then $X_p$ has at least $|F|$ legal configurations.
\end{proposition}

%
%
%
%

\subsection*{Acknowledgements}

The author thanks especially Pierre Guillon, Benjamin Hellouin de Menibus and Guillaume Theyssier for discussions. The $\Z^2$-case of the TEP notion and the uniform measure was developed for \cite{HeSaTh17}, but did not fit. Benjamin suggested putting the ``contour'' on the left, I was stubbornly trying to put it in the middle.

\bibliographystyle{plain}
\bibliography{../../../../bib/bib}{}

\end{document}

%% file: FreeLedra.tex
\draw (0, 0) -- (2, 0);
\draw (2, 0) -- (3.0, 0.0);
\draw (3.0, 0.0) -- (3.5, 0.0);
\draw (3.5, 0.0) -- (3.75, 0.0);
\draw (3.5, 0.0) -- (3.5, 0.25);
\filldraw (3.5,0.25) circle (3pt);
\draw (3.5, 0.0) -- (3.5, -0.25);
\filldraw (3.5,-0.25) circle (3pt);
\draw (3.0, 0.0) -- (3.0, 0.5);
\filldraw (3.0,0.5) circle (3pt);
\draw (3.0, 0.5) -- (3.25, 0.5);
\filldraw (3.25,0.5) circle (3pt);
\draw (3.0, 0.5) -- (2.75, 0.5);
\draw (3.0, 0.5) -- (3.0, 0.75);
\draw (3.0, 0.0) -- (3.0, -0.5);
\filldraw (3.0,-0.5) circle (3pt);
\draw (3.0, -0.5) -- (3.25, -0.5);
\filldraw (3.25,-0.5) circle (3pt);
\draw (3.0, -0.5) -- (2.75, -0.5);
\filldraw (2.75,-0.5) circle (3pt);
\draw (3.0, -0.5) -- (3.0, -0.75);
\filldraw (3.0,-0.75) circle (3pt);
\draw (2, 0) -- (2.0, 1.0);
\draw (2.0, 1.0) -- (2.5, 1.0);
\draw (2.5, 1.0) -- (2.75, 1.0);
\filldraw (2.75,1.0) circle (3pt);
\draw (2.5, 1.0) -- (2.5, 1.25);
\filldraw (2.5,1.25) circle (3pt);
\draw (2.5, 1.0) -- (2.5, 0.75);
\draw (2.0, 1.0) -- (1.5, 1.0);
\filldraw (1.5,1.0) circle (3pt);
\draw (1.5, 1.0) -- (1.25, 1.0);
\draw (1.5, 1.0) -- (1.5, 1.25);
\filldraw (1.5,1.25) circle (3pt);
\draw (1.5, 1.0) -- (1.5, 0.75);
\draw (2.0, 1.0) -- (2.0, 1.5);
\filldraw (2.0,1.5) circle (3pt);
\draw (2.0, 1.5) -- (2.25, 1.5);
\draw (2.0, 1.5) -- (1.75, 1.5);
\filldraw (1.75,1.5) circle (3pt);
\draw (2.0, 1.5) -- (2.0, 1.75);
\draw (2, 0) -- (2.0, -1.0);
\draw (2.0, -1.0) -- (2.5, -1.0);
\filldraw (2.5,-1.0) circle (3pt);
\draw (2.5, -1.0) -- (2.75, -1.0);
\draw (2.5, -1.0) -- (2.5, -0.75);
\draw (2.5, -1.0) -- (2.5, -1.25);
\filldraw (2.5,-1.25) circle (3pt);
\draw (2.0, -1.0) -- (1.5, -1.0);
\draw (1.5, -1.0) -- (1.25, -1.0);
\draw (1.5, -1.0) -- (1.5, -0.75);
\filldraw (1.5,-0.75) circle (3pt);
\draw (1.5, -1.0) -- (1.5, -1.25);
\filldraw (1.5,-1.25) circle (3pt);
\draw (2.0, -1.0) -- (2.0, -1.5);
\filldraw (2.0,-1.5) circle (3pt);
\draw (2.0, -1.5) -- (2.25, -1.5);
\draw (2.0, -1.5) -- (1.75, -1.5);
\filldraw (1.75,-1.5) circle (3pt);
\draw (2.0, -1.5) -- (2.0, -1.75);
\draw (0, 0) -- (-2, 0);
\filldraw (-2,0) circle (3pt);
\draw (-2, 0) -- (-3.0, 0.0);
\draw (-3.0, 0.0) -- (-3.5, 0.0);
\filldraw (-3.5,0.0) circle (3pt);
\draw (-3.5, 0.0) -- (-3.75, 0.0);
\draw (-3.5, 0.0) -- (-3.5, 0.25);
\filldraw (-3.5,0.25) circle (3pt);
\draw (-3.5, 0.0) -- (-3.5, -0.25);
\draw (-3.0, 0.0) -- (-3.0, 0.5);
\filldraw (-3.0,0.5) circle (3pt);
\draw (-3.0, 0.5) -- (-2.75, 0.5);
\draw (-3.0, 0.5) -- (-3.25, 0.5);
\filldraw (-3.25,0.5) circle (3pt);
\draw (-3.0, 0.5) -- (-3.0, 0.75);
\draw (-3.0, 0.0) -- (-3.0, -0.5);
\filldraw (-3.0,-0.5) circle (3pt);
\draw (-3.0, -0.5) -- (-2.75, -0.5);
\filldraw (-2.75,-0.5) circle (3pt);
\draw (-3.0, -0.5) -- (-3.25, -0.5);
\draw (-3.0, -0.5) -- (-3.0, -0.75);
\draw (-2, 0) -- (-2.0, 1.0);
\draw (-2.0, 1.0) -- (-1.5, 1.0);
\filldraw (-1.5,1.0) circle (3pt);
\draw (-1.5, 1.0) -- (-1.25, 1.0);
\filldraw (-1.25,1.0) circle (3pt);
\draw (-1.5, 1.0) -- (-1.5, 1.25);
\draw (-1.5, 1.0) -- (-1.5, 0.75);
\draw (-2.0, 1.0) -- (-2.5, 1.0);
\filldraw (-2.5,1.0) circle (3pt);
\draw (-2.5, 1.0) -- (-2.75, 1.0);
\filldraw (-2.75,1.0) circle (3pt);
\draw (-2.5, 1.0) -- (-2.5, 1.25);
\filldraw (-2.5,1.25) circle (3pt);
\draw (-2.5, 1.0) -- (-2.5, 0.75);
\filldraw (-2.5,0.75) circle (3pt);
\draw (-2.0, 1.0) -- (-2.0, 1.5);
\filldraw (-2.0,1.5) circle (3pt);
\draw (-2.0, 1.5) -- (-1.75, 1.5);
\draw (-2.0, 1.5) -- (-2.25, 1.5);
\filldraw (-2.25,1.5) circle (3pt);
\draw (-2.0, 1.5) -- (-2.0, 1.75);
\draw (-2, 0) -- (-2.0, -1.0);
\filldraw (-2.0,-1.0) circle (3pt);
\draw (-2.0, -1.0) -- (-1.5, -1.0);
\filldraw (-1.5,-1.0) circle (3pt);
\draw (-1.5, -1.0) -- (-1.25, -1.0);
\draw (-1.5, -1.0) -- (-1.5, -0.75);
\filldraw (-1.5,-0.75) circle (3pt);
\draw (-1.5, -1.0) -- (-1.5, -1.25);
\filldraw (-1.5,-1.25) circle (3pt);
\draw (-2.0, -1.0) -- (-2.5, -1.0);
\filldraw (-2.5,-1.0) circle (3pt);
\draw (-2.5, -1.0) -- (-2.75, -1.0);
\filldraw (-2.75,-1.0) circle (3pt);
\draw (-2.5, -1.0) -- (-2.5, -0.75);
\draw (-2.5, -1.0) -- (-2.5, -1.25);
\filldraw (-2.5,-1.25) circle (3pt);
\draw (-2.0, -1.0) -- (-2.0, -1.5);
\draw (-2.0, -1.5) -- (-1.75, -1.5);
\filldraw (-1.75,-1.5) circle (3pt);
\draw (-2.0, -1.5) -- (-2.25, -1.5);
\filldraw (-2.25,-1.5) circle (3pt);
\draw (-2.0, -1.5) -- (-2.0, -1.75);
\filldraw (-2.0,-1.75) circle (3pt);
\draw (0, 0) -- (0, 2);
\filldraw (0,2) circle (3pt);
\draw (0, 2) -- (1.0, 2.0);
\draw (1.0, 2.0) -- (1.5, 2.0);
\filldraw (1.5,2.0) circle (3pt);
\draw (1.5, 2.0) -- (1.75, 2.0);
\filldraw (1.75,2.0) circle (3pt);
\draw (1.5, 2.0) -- (1.5, 2.25);
\filldraw (1.5,2.25) circle (3pt);
\draw (1.5, 2.0) -- (1.5, 1.75);
\filldraw (1.5,1.75) circle (3pt);
\draw (1.0, 2.0) -- (1.0, 2.5);
\draw (1.0, 2.5) -- (1.25, 2.5);
\filldraw (1.25,2.5) circle (3pt);
\draw (1.0, 2.5) -- (0.75, 2.5);
\filldraw (0.75,2.5) circle (3pt);
\draw (1.0, 2.5) -- (1.0, 2.75);
\draw (1.0, 2.0) -- (1.0, 1.5);
\draw (1.0, 1.5) -- (1.25, 1.5);
\filldraw (1.25,1.5) circle (3pt);
\draw (1.0, 1.5) -- (0.75, 1.5);
\filldraw (0.75,1.5) circle (3pt);
\draw (1.0, 1.5) -- (1.0, 1.25);
\draw (0, 2) -- (-1.0, 2.0);
\filldraw (-1.0,2.0) circle (3pt);
\draw (-1.0, 2.0) -- (-1.5, 2.0);
\filldraw (-1.5,2.0) circle (3pt);
\draw (-1.5, 2.0) -- (-1.75, 2.0);
\filldraw (-1.75,2.0) circle (3pt);
\draw (-1.5, 2.0) -- (-1.5, 2.25);
\filldraw (-1.5,2.25) circle (3pt);
\draw (-1.5, 2.0) -- (-1.5, 1.75);
\draw (-1.0, 2.0) -- (-1.0, 2.5);
\filldraw (-1.0,2.5) circle (3pt);
\draw (-1.0, 2.5) -- (-0.75, 2.5);
\filldraw (-0.75,2.5) circle (3pt);
\draw (-1.0, 2.5) -- (-1.25, 2.5);
\draw (-1.0, 2.5) -- (-1.0, 2.75);
\filldraw (-1.0,2.75) circle (3pt);
\draw (-1.0, 2.0) -- (-1.0, 1.5);
\draw (-1.0, 1.5) -- (-0.75, 1.5);
\filldraw (-0.75,1.5) circle (3pt);
\draw (-1.0, 1.5) -- (-1.25, 1.5);
\filldraw (-1.25,1.5) circle (3pt);
\draw (-1.0, 1.5) -- (-1.0, 1.25);
\filldraw (-1.0,1.25) circle (3pt);
\draw (0, 2) -- (0.0, 3.0);
\draw (0.0, 3.0) -- (0.5, 3.0);
\filldraw (0.5,3.0) circle (3pt);
\draw (0.5, 3.0) -- (0.75, 3.0);
\draw (0.5, 3.0) -- (0.5, 3.25);
\draw (0.5, 3.0) -- (0.5, 2.75);
\filldraw (0.5,2.75) circle (3pt);
\draw (0.0, 3.0) -- (-0.5, 3.0);
\draw (-0.5, 3.0) -- (-0.75, 3.0);
\draw (-0.5, 3.0) -- (-0.5, 3.25);
\filldraw (-0.5,3.25) circle (3pt);
\draw (-0.5, 3.0) -- (-0.5, 2.75);
\filldraw (-0.5,2.75) circle (3pt);
\draw (0.0, 3.0) -- (0.0, 3.5);
\draw (0.0, 3.5) -- (0.25, 3.5);
\filldraw (0.25,3.5) circle (3pt);
\draw (0.0, 3.5) -- (-0.25, 3.5);
\filldraw (-0.25,3.5) circle (3pt);
\draw (0.0, 3.5) -- (0.0, 3.75);
\draw (0, 0) -- (0, -2);
\draw (0, -2) -- (1.0, -2.0);
\draw (1.0, -2.0) -- (1.5, -2.0);
\filldraw (1.5,-2.0) circle (3pt);
\draw (1.5, -2.0) -- (1.75, -2.0);
\draw (1.5, -2.0) -- (1.5, -1.75);
\filldraw (1.5,-1.75) circle (3pt);
\draw (1.5, -2.0) -- (1.5, -2.25);
\draw (1.0, -2.0) -- (1.0, -1.5);
\filldraw (1.0,-1.5) circle (3pt);
\draw (1.0, -1.5) -- (1.25, -1.5);
\filldraw (1.25,-1.5) circle (3pt);
\draw (1.0, -1.5) -- (0.75, -1.5);
\filldraw (0.75,-1.5) circle (3pt);
\draw (1.0, -1.5) -- (1.0, -1.25);
\filldraw (1.0,-1.25) circle (3pt);
\draw (1.0, -2.0) -- (1.0, -2.5);
\draw (1.0, -2.5) -- (1.25, -2.5);
\draw (1.0, -2.5) -- (0.75, -2.5);
\draw (1.0, -2.5) -- (1.0, -2.75);
\draw (0, -2) -- (-1.0, -2.0);
\draw (-1.0, -2.0) -- (-1.5, -2.0);
\filldraw (-1.5,-2.0) circle (3pt);
\draw (-1.5, -2.0) -- (-1.75, -2.0);
\filldraw (-1.75,-2.0) circle (3pt);
\draw (-1.5, -2.0) -- (-1.5, -1.75);
\draw (-1.5, -2.0) -- (-1.5, -2.25);
\draw (-1.0, -2.0) -- (-1.0, -1.5);
\filldraw (-1.0,-1.5) circle (3pt);
\draw (-1.0, -1.5) -- (-0.75, -1.5);
\draw (-1.0, -1.5) -- (-1.25, -1.5);
\filldraw (-1.25,-1.5) circle (3pt);
\draw (-1.0, -1.5) -- (-1.0, -1.25);
\draw (-1.0, -2.0) -- (-1.0, -2.5);
\draw (-1.0, -2.5) -- (-0.75, -2.5);
\draw (-1.0, -2.5) -- (-1.25, -2.5);
\draw (-1.0, -2.5) -- (-1.0, -2.75);
\draw (0, -2) -- (0.0, -3.0);
\draw (0.0, -3.0) -- (0.5, -3.0);
\filldraw (0.5,-3.0) circle (3pt);
\draw (0.5, -3.0) -- (0.75, -3.0);
\draw (0.5, -3.0) -- (0.5, -2.75);
\filldraw (0.5,-2.75) circle (3pt);
\draw (0.5, -3.0) -- (0.5, -3.25);
\draw (0.0, -3.0) -- (-0.5, -3.0);
\filldraw (-0.5,-3.0) circle (3pt);
\draw (-0.5, -3.0) -- (-0.75, -3.0);
\draw (-0.5, -3.0) -- (-0.5, -2.75);
\filldraw (-0.5,-2.75) circle (3pt);
\draw (-0.5, -3.0) -- (-0.5, -3.25);
\draw (0.0, -3.0) -- (0.0, -3.5);
\draw (0.0, -3.5) -- (0.25, -3.5);
\draw (0.0, -3.5) -- (-0.25, -3.5);
\draw (0.0, -3.5) -- (0.0, -3.75);

%% file: FreeU.tex
\filldraw (0,0) circle (3pt);
\draw (0, 0) -- (2, 0);
\draw (2, 0) -- (3.0, 0.0);
\filldraw (3.0,0.0) circle (3pt);
\draw (3.0, 0.0) -- (3.5, 0.0);
\filldraw (3.5,0.0) circle (3pt);
\draw (3.5, 0.0) -- (3.75, 0.0);
\filldraw (3.75,0.0) circle (3pt);
\draw (3.5, 0.0) -- (3.5, 0.25);
\draw (3.5, 0.0) -- (3.5, -0.25);
\draw (3.0, 0.0) -- (3.0, 0.5);
\draw (3.0, 0.5) -- (3.25, 0.5);
\draw (3.0, 0.5) -- (2.75, 0.5);
\filldraw (2.75,0.5) circle (3pt);
\draw (3.0, 0.5) -- (3.0, 0.75);
\filldraw (3.0,0.75) circle (3pt);
\draw (3.0, 0.0) -- (3.0, -0.5);
\draw (3.0, -0.5) -- (3.25, -0.5);
\filldraw (3.25,-0.5) circle (3pt);
\draw (3.0, -0.5) -- (2.75, -0.5);
\filldraw (2.75,-0.5) circle (3pt);
\draw (3.0, -0.5) -- (3.0, -0.75);
\filldraw (3.0,-0.75) circle (3pt);
\draw (2, 0) -- (2.0, 1.0);
\filldraw (2.0,1.0) circle (3pt);
\draw (2.0, 1.0) -- (2.5, 1.0);
\draw (2.5, 1.0) -- (2.75, 1.0);
\filldraw (2.75,1.0) circle (3pt);
\draw (2.5, 1.0) -- (2.5, 1.25);
\draw (2.5, 1.0) -- (2.5, 0.75);
\filldraw (2.5,0.75) circle (3pt);
\draw (2.0, 1.0) -- (1.5, 1.0);
\filldraw (1.5,1.0) circle (3pt);
\draw (1.5, 1.0) -- (1.25, 1.0);
\filldraw (1.25,1.0) circle (3pt);
\draw (1.5, 1.0) -- (1.5, 1.25);
\filldraw (1.5,1.25) circle (3pt);
\draw (1.5, 1.0) -- (1.5, 0.75);
\draw (2.0, 1.0) -- (2.0, 1.5);
\filldraw (2.0,1.5) circle (3pt);
\draw (2.0, 1.5) -- (2.25, 1.5);
\filldraw (2.25,1.5) circle (3pt);
\draw (2.0, 1.5) -- (1.75, 1.5);
\filldraw (1.75,1.5) circle (3pt);
\draw (2.0, 1.5) -- (2.0, 1.75);
\draw (2, 0) -- (2.0, -1.0);
\draw (2.0, -1.0) -- (2.5, -1.0);
\filldraw (2.5,-1.0) circle (3pt);
\draw (2.5, -1.0) -- (2.75, -1.0);
\filldraw (2.75,-1.0) circle (3pt);
\draw (2.5, -1.0) -- (2.5, -0.75);
\filldraw (2.5,-0.75) circle (3pt);
\draw (2.5, -1.0) -- (2.5, -1.25);
\draw (2.0, -1.0) -- (1.5, -1.0);
\filldraw (1.5,-1.0) circle (3pt);
\draw (1.5, -1.0) -- (1.25, -1.0);
\filldraw (1.25,-1.0) circle (3pt);
\draw (1.5, -1.0) -- (1.5, -0.75);
\filldraw (1.5,-0.75) circle (3pt);
\draw (1.5, -1.0) -- (1.5, -1.25);
\draw (2.0, -1.0) -- (2.0, -1.5);
\filldraw (2.0,-1.5) circle (3pt);
\draw (2.0, -1.5) -- (2.25, -1.5);
\draw (2.0, -1.5) -- (1.75, -1.5);
\filldraw (1.75,-1.5) circle (3pt);
\draw (2.0, -1.5) -- (2.0, -1.75);
\draw (0, 0) -- (-2, 0);
\draw (-2, 0) -- (-3.0, 0.0);
\draw (-3.0, 0.0) -- (-3.5, 0.0);
\draw (-3.5, 0.0) -- (-3.75, 0.0);
\draw (-3.5, 0.0) -- (-3.5, 0.25);
\filldraw (-3.5,0.25) circle (3pt);
\draw (-3.5, 0.0) -- (-3.5, -0.25);
\filldraw (-3.5,-0.25) circle (3pt);
\draw (-3.0, 0.0) -- (-3.0, 0.5);
\filldraw (-3.0,0.5) circle (3pt);
\draw (-3.0, 0.5) -- (-2.75, 0.5);
\draw (-3.0, 0.5) -- (-3.25, 0.5);
\filldraw (-3.25,0.5) circle (3pt);
\draw (-3.0, 0.5) -- (-3.0, 0.75);
\filldraw (-3.0,0.75) circle (3pt);
\draw (-3.0, 0.0) -- (-3.0, -0.5);
\draw (-3.0, -0.5) -- (-2.75, -0.5);
\filldraw (-2.75,-0.5) circle (3pt);
\draw (-3.0, -0.5) -- (-3.25, -0.5);
\draw (-3.0, -0.5) -- (-3.0, -0.75);
\draw (-2, 0) -- (-2.0, 1.0);
\filldraw (-2.0,1.0) circle (3pt);
\draw (-2.0, 1.0) -- (-1.5, 1.0);
\filldraw (-1.5,1.0) circle (3pt);
\draw (-1.5, 1.0) -- (-1.25, 1.0);
\filldraw (-1.25,1.0) circle (3pt);
\draw (-1.5, 1.0) -- (-1.5, 1.25);
\filldraw (-1.5,1.25) circle (3pt);
\draw (-1.5, 1.0) -- (-1.5, 0.75);
\draw (-2.0, 1.0) -- (-2.5, 1.0);
\draw (-2.5, 1.0) -- (-2.75, 1.0);
\draw (-2.5, 1.0) -- (-2.5, 1.25);
\draw (-2.5, 1.0) -- (-2.5, 0.75);
\draw (-2.0, 1.0) -- (-2.0, 1.5);
\filldraw (-2.0,1.5) circle (3pt);
\draw (-2.0, 1.5) -- (-1.75, 1.5);
\draw (-2.0, 1.5) -- (-2.25, 1.5);
\draw (-2.0, 1.5) -- (-2.0, 1.75);
\draw (-2, 0) -- (-2.0, -1.0);
\draw (-2.0, -1.0) -- (-1.5, -1.0);
\filldraw (-1.5,-1.0) circle (3pt);
\draw (-1.5, -1.0) -- (-1.25, -1.0);
\filldraw (-1.25,-1.0) circle (3pt);
\draw (-1.5, -1.0) -- (-1.5, -0.75);
\filldraw (-1.5,-0.75) circle (3pt);
\draw (-1.5, -1.0) -- (-1.5, -1.25);
\draw (-2.0, -1.0) -- (-2.5, -1.0);
\draw (-2.5, -1.0) -- (-2.75, -1.0);
\filldraw (-2.75,-1.0) circle (3pt);
\draw (-2.5, -1.0) -- (-2.5, -0.75);
\draw (-2.5, -1.0) -- (-2.5, -1.25);
\draw (-2.0, -1.0) -- (-2.0, -1.5);
\filldraw (-2.0,-1.5) circle (3pt);
\draw (-2.0, -1.5) -- (-1.75, -1.5);
\filldraw (-1.75,-1.5) circle (3pt);
\draw (-2.0, -1.5) -- (-2.25, -1.5);
\filldraw (-2.25,-1.5) circle (3pt);
\draw (-2.0, -1.5) -- (-2.0, -1.75);
\filldraw (-2.0,-1.75) circle (3pt);
\draw (0, 0) -- (0, 2);
\draw (0, 2) -- (1.0, 2.0);
\draw (1.0, 2.0) -- (1.5, 2.0);
\draw (1.5, 2.0) -- (1.75, 2.0);
\filldraw (1.75,2.0) circle (3pt);
\draw (1.5, 2.0) -- (1.5, 2.25);
\filldraw (1.5,2.25) circle (3pt);
\draw (1.5, 2.0) -- (1.5, 1.75);
\filldraw (1.5,1.75) circle (3pt);
\draw (1.0, 2.0) -- (1.0, 2.5);
\filldraw (1.0,2.5) circle (3pt);
\draw (1.0, 2.5) -- (1.25, 2.5);
\filldraw (1.25,2.5) circle (3pt);
\draw (1.0, 2.5) -- (0.75, 2.5);
\draw (1.0, 2.5) -- (1.0, 2.75);
\draw (1.0, 2.0) -- (1.0, 1.5);
\filldraw (1.0,1.5) circle (3pt);
\draw (1.0, 1.5) -- (1.25, 1.5);
\filldraw (1.25,1.5) circle (3pt);
\draw (1.0, 1.5) -- (0.75, 1.5);
\filldraw (0.75,1.5) circle (3pt);
\draw (1.0, 1.5) -- (1.0, 1.25);
\filldraw (1.0,1.25) circle (3pt);
\draw (0, 2) -- (-1.0, 2.0);
\filldraw (-1.0,2.0) circle (3pt);
\draw (-1.0, 2.0) -- (-1.5, 2.0);
\filldraw (-1.5,2.0) circle (3pt);
\draw (-1.5, 2.0) -- (-1.75, 2.0);
\draw (-1.5, 2.0) -- (-1.5, 2.25);
\draw (-1.5, 2.0) -- (-1.5, 1.75);
\filldraw (-1.5,1.75) circle (3pt);
\draw (-1.0, 2.0) -- (-1.0, 2.5);
\draw (-1.0, 2.5) -- (-0.75, 2.5);
\draw (-1.0, 2.5) -- (-1.25, 2.5);
\filldraw (-1.25,2.5) circle (3pt);
\draw (-1.0, 2.5) -- (-1.0, 2.75);
\draw (-1.0, 2.0) -- (-1.0, 1.5);
\draw (-1.0, 1.5) -- (-0.75, 1.5);
\draw (-1.0, 1.5) -- (-1.25, 1.5);
\filldraw (-1.25,1.5) circle (3pt);
\draw (-1.0, 1.5) -- (-1.0, 1.25);
\filldraw (-1.0,1.25) circle (3pt);
\draw (0, 2) -- (0.0, 3.0);
\filldraw (0.0,3.0) circle (3pt);
\draw (0.0, 3.0) -- (0.5, 3.0);
\filldraw (0.5,3.0) circle (3pt);
\draw (0.5, 3.0) -- (0.75, 3.0);
\filldraw (0.75,3.0) circle (3pt);
\draw (0.5, 3.0) -- (0.5, 3.25);
\filldraw (0.5,3.25) circle (3pt);
\draw (0.5, 3.0) -- (0.5, 2.75);
\draw (0.0, 3.0) -- (-0.5, 3.0);
\filldraw (-0.5,3.0) circle (3pt);
\draw (-0.5, 3.0) -- (-0.75, 3.0);
\draw (-0.5, 3.0) -- (-0.5, 3.25);
\filldraw (-0.5,3.25) circle (3pt);
\draw (-0.5, 3.0) -- (-0.5, 2.75);
\draw (0.0, 3.0) -- (0.0, 3.5);
\filldraw (0.0,3.5) circle (3pt);
\draw (0.0, 3.5) -- (0.25, 3.5);
\filldraw (0.25,3.5) circle (3pt);
\draw (0.0, 3.5) -- (-0.25, 3.5);
\draw (0.0, 3.5) -- (0.0, 3.75);
\filldraw (0.0,3.75) circle (3pt);
\draw (0, 0) -- (0, -2);
\draw (0, -2) -- (1.0, -2.0);
\filldraw (1.0,-2.0) circle (3pt);
\draw (1.0, -2.0) -- (1.5, -2.0);
\draw (1.5, -2.0) -- (1.75, -2.0);
\draw (1.5, -2.0) -- (1.5, -1.75);
\filldraw (1.5,-1.75) circle (3pt);
\draw (1.5, -2.0) -- (1.5, -2.25);
\draw (1.0, -2.0) -- (1.0, -1.5);
\draw (1.0, -1.5) -- (1.25, -1.5);
\filldraw (1.25,-1.5) circle (3pt);
\draw (1.0, -1.5) -- (0.75, -1.5);
\draw (1.0, -1.5) -- (1.0, -1.25);
\draw (1.0, -2.0) -- (1.0, -2.5);
\draw (1.0, -2.5) -- (1.25, -2.5);
\draw (1.0, -2.5) -- (0.75, -2.5);
\draw (1.0, -2.5) -- (1.0, -2.75);
\draw (0, -2) -- (-1.0, -2.0);
\draw (-1.0, -2.0) -- (-1.5, -2.0);
\filldraw (-1.5,-2.0) circle (3pt);
\draw (-1.5, -2.0) -- (-1.75, -2.0);
\draw (-1.5, -2.0) -- (-1.5, -1.75);
\draw (-1.5, -2.0) -- (-1.5, -2.25);
\filldraw (-1.5,-2.25) circle (3pt);
\draw (-1.0, -2.0) -- (-1.0, -1.5);
\filldraw (-1.0,-1.5) circle (3pt);
\draw (-1.0, -1.5) -- (-0.75, -1.5);
\draw (-1.0, -1.5) -- (-1.25, -1.5);
\filldraw (-1.25,-1.5) circle (3pt);
\draw (-1.0, -1.5) -- (-1.0, -1.25);
\draw (-1.0, -2.0) -- (-1.0, -2.5);
\draw (-1.0, -2.5) -- (-0.75, -2.5);
\draw (-1.0, -2.5) -- (-1.25, -2.5);
\filldraw (-1.25,-2.5) circle (3pt);
\draw (-1.0, -2.5) -- (-1.0, -2.75);
\filldraw (-1.0,-2.75) circle (3pt);
\draw (0, -2) -- (0.0, -3.0);
\draw (0.0, -3.0) -- (0.5, -3.0);
\draw (0.5, -3.0) -- (0.75, -3.0);
\draw (0.5, -3.0) -- (0.5, -2.75);
\draw (0.5, -3.0) -- (0.5, -3.25);
\filldraw (0.5,-3.25) circle (3pt);
\draw (0.0, -3.0) -- (-0.5, -3.0);
\draw (-0.5, -3.0) -- (-0.75, -3.0);
\draw (-0.5, -3.0) -- (-0.5, -2.75);
\draw (-0.5, -3.0) -- (-0.5, -3.25);
\draw (0.0, -3.0) -- (0.0, -3.5);
\draw (0.0, -3.5) -- (0.25, -3.5);
\filldraw (0.25,-3.5) circle (3pt);
\draw (0.0, -3.5) -- (-0.25, -3.5);
\filldraw (-0.25,-3.5) circle (3pt);
\draw (0.0, -3.5) -- (0.0, -3.75);
\filldraw (0.0,-3.75) circle (3pt);

%% file: convex_ordering.tex
\node () at (0.5,1.5) {\footnotesize 1};
\node () at (-1.5,0.0) {\footnotesize 2};
\node () at (-1.0,0.5) {\footnotesize 3};
\node () at (-2.0,0.0) {\footnotesize 4};
\node () at (-2.0,-0.5) {\footnotesize 5};
\node () at (-0.5,1.0) {\footnotesize 6};
\node () at (-1.5,0.5) {\footnotesize 7};
\node () at (0.0,1.0) {\footnotesize 8};
\node () at (-2.0,0.5) {\footnotesize 9};
\node () at (-1.0,1.0) {\footnotesize 10};
\node () at (0.0,1.5) {\footnotesize 11};
\node () at (0.5,2.0) {\footnotesize 12};
\node () at (-0.5,1.5) {\footnotesize 13};
\node () at (0.0,2.0) {\footnotesize 14};
\node () at (-0.5,0.5) {\footnotesize 15};
\node () at (-1.5,1.0) {\footnotesize 16};
\node () at (-1.0,0.0) {\footnotesize 17};
\node () at (0.5,1.0) {\footnotesize 18};
\node () at (-1.5,-0.5) {\footnotesize 19};
\node () at (0.0,0.5) {\footnotesize 20};
\node () at (-1.0,1.5) {\footnotesize 21};
\node () at (-1.5,1.5) {\footnotesize 22};
\node () at (1.0,1.5) {\footnotesize 23};
\node () at (1.0,1.0) {\footnotesize 24};
\node () at (-0.5,2.0) {\footnotesize 25};
\node () at (1.5,1.5) {\footnotesize 26};
\node () at (-0.5,0.0) {\footnotesize 27};
\node () at (0.5,0.5) {\footnotesize 28};
\node () at (-1.0,-0.5) {\footnotesize 29};
\node () at (-1.5,-1.0) {\footnotesize 30};
\node () at (0.0,0.0) {\footnotesize 31};
\node () at (-0.5,-0.5) {\footnotesize 32};
\node () at (1.5,1.0) {\footnotesize 33};
\node () at (1.0,0.5) {\footnotesize 34};
\node () at (0.5,0.0) {\footnotesize 35};
\node () at (-1.0,-1.0) {\footnotesize 36};
\node () at (1.5,0.5) {\footnotesize 37};
\node () at (0.0,-0.5) {\footnotesize 38};
\node () at (-1.5,-1.5) {\footnotesize 39};
\node () at (1.0,0.0) {\footnotesize 40};
\node () at (-0.5,-1.0) {\footnotesize 41};
\node () at (-1.0,-1.5) {\footnotesize 42};
\node () at (0.5,-0.5) {\footnotesize 43};
\node () at (2.0,0.5) {\footnotesize 44};
\node () at (1.5,0.0) {\footnotesize 45};
\node () at (0.0,-1.0) {\footnotesize 46};
\node () at (1.0,-0.5) {\footnotesize 47};
\node () at (2.0,0.0) {\footnotesize 48};
\node () at (1.5,-0.5) {\footnotesize 49};
\node () at (0.5,-1.0) {\footnotesize 50};
\node () at (1.0,-1.0) {\footnotesize 51};
\node () at (2.0,-0.5) {\footnotesize 52};
\node () at (-0.5,-1.5) {\footnotesize 53};
\node () at (0.0,-1.5) {\footnotesize 54};
\node () at (-0.5,-2.0) {\footnotesize 55};
\node () at (0.5,-1.5) {\footnotesize 56};
\node () at (0.0,-2.0) {\footnotesize 57};
\node () at (1.5,-1.0) {\footnotesize 58};
\node () at (1.0,-1.5) {\footnotesize 59};
\node () at (0.5,-2.0) {\footnotesize 60};
\node () at (1.5,-1.5) {\footnotesize 61};

%% file: convex_as.tex
\draw[black!100!white] (6,-6) -- (2,-7);
\draw[black!100!white] (8,-3) -- (2,-7);
\draw[black!100!white] (4,-4) -- (1,-6);
\draw[black!100!white] (8,-5) -- (7,-7);
\draw[black!100!white] (4,-3) -- (1,-5);
\draw[black!100!white] (2,-3) -- (5,-1);
\draw[black!100!white] (8,-5) -- (5,-8);
\draw[black!100!white] (3,-2) -- (5,-1);
\draw[black!100!white] (1,-7) -- (0,-5);
\draw[black!100!white] (5,-4) -- (2,-6);
\draw[black!100!white] (6,-4) -- (2,-7);
\draw[black!100!white] (5,-1) -- (1,-4);
\draw[black!100!white] (7,-1) -- (7,-3);
\draw[black!100!white] (5,-7) -- (3,-8);
\draw[black!100!white] (7,-2) -- (5,-4);
\draw[black!100!white] (1,-4) -- (0,-4);
\draw[black!100!white] (3,-6) -- (1,-7);
\draw[black!100!white] (6,-2) -- (3,-4);
\draw[black!100!white] (5,-8) -- (3,-8);
\draw[black!100!white] (1,-3) -- (5,-1);
\draw[black!100!white] (8,-5) -- (6,-7);
\draw[black!100!white] (8,-5) -- (6,-6);
\draw[black!100!white] (2,-2) -- (5,-1);
\draw[black!100!white] (0,-3) -- (4,-1);
\draw[black!100!white] (6,-1) -- (4,-3);
\draw[black!100!white] (6,-4) -- (3,-6);
\draw[black!100!white] (0,-3) -- (0,-3);
\draw[black!100!white] (1,-4) -- (1,-4);
\draw[black!100!white] (0,-3) -- (5,0);
\draw[black!100!white] (5,-1) -- (0,-5);
\draw[black!100!white] (0,-3) -- (2,-2);
\draw[black!100!white] (0,-3) -- (2,-1);
\draw[black!100!white] (8,-4) -- (2,-7);
\draw[black!100!white] (5,-2) -- (4,-3);
\draw[black!100!white] (2,-1) -- (4,0);
\draw[black!100!white] (4,-8) -- (3,-8);
\draw[black!100!white] (5,-1) -- (4,-2);
\draw[black!100!white] (3,-4) -- (1,-5);
\draw[black!100!white] (3,-8) -- (1,-7);
\draw[black!100!white] (7,-1) -- (8,-3);
\draw[black!100!white] (4,-7) -- (1,-7);
\draw[black!100!white] (1,-6) -- (0,-5);
\draw[black!100!white] (7,-3) -- (2,-6);
\draw[black!100!white] (0,-3) -- (5,-1);
\draw[black!100!white] (1,-2) -- (4,0);
\draw[black!100!white] (1,-1) -- (4,0);
\draw[black!100!white] (5,-2) -- (1,-5);
\draw[black!100!white] (0,-3) -- (1,-2);
\draw[black!100!white] (0,-3) -- (1,-1);
\draw[black!100!white] (7,-3) -- (4,-5);
\draw[black!100!white] (3,0) -- (5,0);
\draw[black!100!white] (8,-4) -- (7,-5);
\draw[black!100!white] (7,-3) -- (6,-4);
\draw[black!100!white] (0,-3) -- (3,-1);
\draw[black!100!white] (6,-5) -- (2,-7);
\draw[black!100!white] (6,-2) -- (1,-5);
\draw[black!100!white] (2,-4) -- (0,-5);
\draw[black!100!white] (5,-1) -- (3,-3);
\draw[black!100!white] (7,-5) -- (2,-7);
\draw[black!100!white] (0,-5) -- (0,-4);
\draw[black!100!white] (5,-1) -- (5,-1);
\draw[black!100!white] (8,-3) -- (7,-4);
\draw[black!100!white] (7,-7) -- (5,-8);
\draw[black!100!white] (0,-4) -- (0,-4);
\draw[black!100!white] (4,-1) -- (5,-1);
\draw[black!100!white] (5,-3) -- (1,-6);
\draw[black!100!white] (7,-3) -- (1,-7);
\draw[black!100!white] (8,-5) -- (4,-8);
\draw[black!100!white] (7,-3) -- (5,-5);
\draw[black!100!white] (5,-6) -- (2,-7);
\draw[black!100!white] (8,-3) -- (8,-5);
\draw[black!100!white] (8,-3) -- (8,-4);
\draw[black!100!white] (3,-1) -- (5,0);
\draw[black!100!white] (7,-2) -- (6,-3);
\draw[black!100!white] (7,-1) -- (6,-2);
\draw[black!100!white] (4,-2) -- (0,-5);
\draw[black!100!white] (3,-7) -- (1,-7);
\draw[black!100!white] (5,-1) -- (2,-4);
\draw[black!100!white] (5,-2) -- (2,-4);
\draw[black!100!white] (4,-5) -- (2,-6);
\draw[black!100!white] (4,0) -- (5,0);
\draw[black!100!white] (0,-4) -- (5,-1);
\draw[black!100!white] (5,0) -- (6,-1);
\draw[black!100!white] (2,-5) -- (0,-5);
\draw[black!100!white] (0,-4) -- (1,-3);
\draw[black!100!white] (8,-5) -- (7,-6);
\draw[black!100!white] (7,-1) -- (3,-5);
\draw[black!100!white] (7,-2) -- (3,-5);
\draw[black!100!white] (0,-4) -- (3,-2);
\draw[black!100!white] (8,-5) -- (5,-7);
\draw[black!100!white] (3,-3) -- (0,-5);
\draw[black!100!white] (1,-4) -- (2,-3);
\draw[black!100!white] (7,-1) -- (7,-2);
\draw[black!100!white] (0,-3) -- (4,0);
\draw[black!100!white] (2,-6) -- (1,-6);
\draw[black!100!white] (3,-5) -- (1,-6);
\draw[black!100!white] (7,-1) -- (5,-3);
\draw[black!100!white] (0,-5) -- (0,-3);
\draw[black!100!white] (7,-4) -- (4,-6);
\draw[black!100!white] (5,0) -- (7,-1);
\draw[black!100!white] (7,-4) -- (2,-7);
\draw[black!100!white] (5,0) -- (5,-1);
\draw[black!100!white] (5,0) -- (5,-2);
\draw[black!100!white] (1,-1) -- (3,0);
\draw[black!100!white] (5,-5) -- (2,-7);
\draw[black!100!white] (7,-6) -- (4,-8);
\draw[black!100!white] (5,-3) -- (1,-5);
\draw[black!100!white] (6,-6) -- (3,-7);
\draw[black!100!white] (8,-5) -- (4,-7);
\draw[black!100!white] (5,-4) -- (1,-6);
\draw[black!100!white] (8,-4) -- (6,-6);
\draw[black!100!white] (1,-4) -- (5,-1);
\draw[black!100!white] (4,-6) -- (2,-7);
\draw[black!100!white] (6,-7) -- (4,-8);
\draw[black!100!white] (6,-1) -- (6,-2);
\draw[black!100!white] (1,-5) -- (0,-5);
\draw[black!100!white] (6,-4) -- (1,-7);
\draw[black!100!white] (7,-5) -- (5,-6);
\draw[black!100!white] (8,-5) -- (3,-8);
\draw[black!100!white] (7,-1) -- (4,-4);
\draw[black!100!white] (5,-3) -- (2,-5);
\draw[black!100!white] (6,-3) -- (3,-5);
\draw[black!100!white] (8,-3) -- (6,-5);
\draw[black!100!white] (2,-7) -- (1,-7);

%% file: magnus.tex
\filldraw (0,0) circle (3pt);
\draw (0, 0) -- (2, 0);
\filldraw (2,0) circle (3pt);
\draw (2, 0) -- (3.0, 0.0);
\filldraw (3.0,0.0) circle (3pt);
\draw (3.0, 0.0) -- (3.5, 0.0);
\filldraw (3.5,0.0) circle (3pt);
\draw (3.5, 0.0) -- (3.75, 0.0);
\filldraw (3.75,0.0) circle (3pt);
\draw (3.5, 0.0) -- (3.5, 0.25);
\filldraw (3.5,0.25) circle (3pt);
\draw (3.5, 0.0) -- (3.5, -0.25);
\filldraw (3.5,-0.25) circle (3pt);
\draw (3.0, 0.0) -- (3.0, 0.5);
\filldraw (3.0,0.5) circle (3pt);
\draw (3.0, 0.5) -- (3.25, 0.5);
\filldraw (3.25,0.5) circle (3pt);
\draw (3.0, 0.5) -- (2.75, 0.5);
\filldraw (2.75,0.5) circle (3pt);
\draw (3.0, 0.5) -- (3.0, 0.75);
\filldraw (3.0,0.75) circle (3pt);
\draw (3.0, 0.0) -- (3.0, -0.5);
\filldraw (3.0,-0.5) circle (3pt);
\draw (3.0, -0.5) -- (3.25, -0.5);
\filldraw (3.25,-0.5) circle (3pt);
\draw (3.0, -0.5) -- (2.75, -0.5);
\filldraw (2.75,-0.5) circle (3pt);
\draw (3.0, -0.5) -- (3.0, -0.75);
\filldraw (3.0,-0.75) circle (3pt);
\draw (2, 0) -- (2.0, 1.0);
\filldraw (2.0,1.0) circle (3pt);
\draw (2.0, 1.0) -- (2.5, 1.0);
\filldraw (2.5,1.0) circle (3pt);
\draw (2.5, 1.0) -- (2.75, 1.0);
\filldraw (2.75,1.0) circle (3pt);
\draw (2.5, 1.0) -- (2.5, 1.25);
\filldraw (2.5,1.25) circle (3pt);
\draw (2.5, 1.0) -- (2.5, 0.75);
\filldraw (2.5,0.75) circle (3pt);
\draw (2.0, 1.0) -- (1.5, 1.0);
\filldraw (1.5,1.0) circle (3pt);
\draw (1.5, 1.0) -- (1.25, 1.0);
\draw (1.5, 1.0) -- (1.5, 1.25);
\filldraw (1.5,1.25) circle (3pt);
\draw (1.5, 1.0) -- (1.5, 0.75);
\filldraw (1.5,0.75) circle (3pt);
\draw (2.0, 1.0) -- (2.0, 1.5);
\filldraw (2.0,1.5) circle (3pt);
\draw (2.0, 1.5) -- (2.25, 1.5);
\filldraw (2.25,1.5) circle (3pt);
\draw (2.0, 1.5) -- (1.75, 1.5);
\filldraw (1.75,1.5) circle (3pt);
\draw (2.0, 1.5) -- (2.0, 1.75);
\filldraw (2.0,1.75) circle (3pt);
\draw (2, 0) -- (2.0, -1.0);
\filldraw (2.0,-1.0) circle (3pt);
\draw (2.0, -1.0) -- (2.5, -1.0);
\filldraw (2.5,-1.0) circle (3pt);
\draw (2.5, -1.0) -- (2.75, -1.0);
\filldraw (2.75,-1.0) circle (3pt);
\draw (2.5, -1.0) -- (2.5, -0.75);
\filldraw (2.5,-0.75) circle (3pt);
\draw (2.5, -1.0) -- (2.5, -1.25);
\filldraw (2.5,-1.25) circle (3pt);
\draw (2.0, -1.0) -- (1.5, -1.0);
\draw (1.5, -1.0) -- (1.25, -1.0);
\draw (1.5, -1.0) -- (1.5, -0.75);
\draw (1.5, -1.0) -- (1.5, -1.25);
\draw (2.0, -1.0) -- (2.0, -1.5);
\filldraw (2.0,-1.5) circle (3pt);
\draw (2.0, -1.5) -- (2.25, -1.5);
\filldraw (2.25,-1.5) circle (3pt);
\draw (2.0, -1.5) -- (1.75, -1.5);
\draw (2.0, -1.5) -- (2.0, -1.75);
\filldraw (2.0,-1.75) circle (3pt);
\draw (0, 0) -- (-2, 0);
\draw (-2, 0) -- (-3.0, 0.0);
\draw (-3.0, 0.0) -- (-3.5, 0.0);
\draw (-3.5, 0.0) -- (-3.75, 0.0);
\draw (-3.5, 0.0) -- (-3.5, 0.25);
\draw (-3.5, 0.0) -- (-3.5, -0.25);
\draw (-3.0, 0.0) -- (-3.0, 0.5);
\draw (-3.0, 0.5) -- (-2.75, 0.5);
\draw (-3.0, 0.5) -- (-3.25, 0.5);
\draw (-3.0, 0.5) -- (-3.0, 0.75);
\draw (-3.0, 0.0) -- (-3.0, -0.5);
\draw (-3.0, -0.5) -- (-2.75, -0.5);
\draw (-3.0, -0.5) -- (-3.25, -0.5);
\draw (-3.0, -0.5) -- (-3.0, -0.75);
\draw (-2, 0) -- (-2.0, 1.0);
\draw (-2.0, 1.0) -- (-1.5, 1.0);
\filldraw (-1.5,1.0) circle (3pt);
\draw (-1.5, 1.0) -- (-1.25, 1.0);
\filldraw (-1.25,1.0) circle (3pt);
\draw (-1.5, 1.0) -- (-1.5, 1.25);
\filldraw (-1.5,1.25) circle (3pt);
\draw (-1.5, 1.0) -- (-1.5, 0.75);
\draw (-2.0, 1.0) -- (-2.5, 1.0);
\draw (-2.5, 1.0) -- (-2.75, 1.0);
\draw (-2.5, 1.0) -- (-2.5, 1.25);
\draw (-2.5, 1.0) -- (-2.5, 0.75);
\draw (-2.0, 1.0) -- (-2.0, 1.5);
\draw (-2.0, 1.5) -- (-1.75, 1.5);
\filldraw (-1.75,1.5) circle (3pt);
\draw (-2.0, 1.5) -- (-2.25, 1.5);
\draw (-2.0, 1.5) -- (-2.0, 1.75);
\draw (-2, 0) -- (-2.0, -1.0);
\draw (-2.0, -1.0) -- (-1.5, -1.0);
\draw (-1.5, -1.0) -- (-1.25, -1.0);
\filldraw (-1.25,-1.0) circle (3pt);
\draw (-1.5, -1.0) -- (-1.5, -0.75);
\filldraw (-1.5,-0.75) circle (3pt);
\draw (-1.5, -1.0) -- (-1.5, -1.25);
\draw (-2.0, -1.0) -- (-2.5, -1.0);
\draw (-2.5, -1.0) -- (-2.75, -1.0);
\draw (-2.5, -1.0) -- (-2.5, -0.75);
\draw (-2.5, -1.0) -- (-2.5, -1.25);
\draw (-2.0, -1.0) -- (-2.0, -1.5);
\draw (-2.0, -1.5) -- (-1.75, -1.5);
\draw (-2.0, -1.5) -- (-2.25, -1.5);
\draw (-2.0, -1.5) -- (-2.0, -1.75);
\draw (0, 0) -- (0, 2);
\filldraw (0,2) circle (3pt);
\draw (0, 2) -- (1.0, 2.0);
\filldraw (1.0,2.0) circle (3pt);
\draw (1.0, 2.0) -- (1.5, 2.0);
\filldraw (1.5,2.0) circle (3pt);
\draw (1.5, 2.0) -- (1.75, 2.0);
\filldraw (1.75,2.0) circle (3pt);
\draw (1.5, 2.0) -- (1.5, 2.25);
\filldraw (1.5,2.25) circle (3pt);
\draw (1.5, 2.0) -- (1.5, 1.75);
\filldraw (1.5,1.75) circle (3pt);
\draw (1.0, 2.0) -- (1.0, 2.5);
\filldraw (1.0,2.5) circle (3pt);
\draw (1.0, 2.5) -- (1.25, 2.5);
\filldraw (1.25,2.5) circle (3pt);
\draw (1.0, 2.5) -- (0.75, 2.5);
\filldraw (0.75,2.5) circle (3pt);
\draw (1.0, 2.5) -- (1.0, 2.75);
\filldraw (1.0,2.75) circle (3pt);
\draw (1.0, 2.0) -- (1.0, 1.5);
\filldraw (1.0,1.5) circle (3pt);
\draw (1.0, 1.5) -- (1.25, 1.5);
\filldraw (1.25,1.5) circle (3pt);
\draw (1.0, 1.5) -- (0.75, 1.5);
\draw (1.0, 1.5) -- (1.0, 1.25);
\filldraw (1.0,1.25) circle (3pt);
\draw (0, 2) -- (-1.0, 2.0);
\draw (-1.0, 2.0) -- (-1.5, 2.0);
\draw (-1.5, 2.0) -- (-1.75, 2.0);
\draw (-1.5, 2.0) -- (-1.5, 2.25);
\draw (-1.5, 2.0) -- (-1.5, 1.75);
\draw (-1.0, 2.0) -- (-1.0, 2.5);
\draw (-1.0, 2.5) -- (-0.75, 2.5);
\filldraw (-0.75,2.5) circle (3pt);
\draw (-1.0, 2.5) -- (-1.25, 2.5);
\draw (-1.0, 2.5) -- (-1.0, 2.75);
\draw (-1.0, 2.0) -- (-1.0, 1.5);
\draw (-1.0, 1.5) -- (-0.75, 1.5);
\filldraw (-0.75,1.5) circle (3pt);
\draw (-1.0, 1.5) -- (-1.25, 1.5);
\draw (-1.0, 1.5) -- (-1.0, 1.25);
\draw (0, 2) -- (0.0, 3.0);
\filldraw (0.0,3.0) circle (3pt);
\draw (0.0, 3.0) -- (0.5, 3.0);
\filldraw (0.5,3.0) circle (3pt);
\draw (0.5, 3.0) -- (0.75, 3.0);
\filldraw (0.75,3.0) circle (3pt);
\draw (0.5, 3.0) -- (0.5, 3.25);
\filldraw (0.5,3.25) circle (3pt);
\draw (0.5, 3.0) -- (0.5, 2.75);
\filldraw (0.5,2.75) circle (3pt);
\draw (0.0, 3.0) -- (-0.5, 3.0);
\draw (-0.5, 3.0) -- (-0.75, 3.0);
\draw (-0.5, 3.0) -- (-0.5, 3.25);
\draw (-0.5, 3.0) -- (-0.5, 2.75);
\draw (0.0, 3.0) -- (0.0, 3.5);
\filldraw (0.0,3.5) circle (3pt);
\draw (0.0, 3.5) -- (0.25, 3.5);
\filldraw (0.25,3.5) circle (3pt);
\draw (0.0, 3.5) -- (-0.25, 3.5);
\draw (0.0, 3.5) -- (0.0, 3.75);
\filldraw (0.0,3.75) circle (3pt);
\draw (0, 0) -- (0, -2);
\draw (0, -2) -- (1.0, -2.0);
\filldraw (1.0,-2.0) circle (3pt);
\draw (1.0, -2.0) -- (1.5, -2.0);
\filldraw (1.5,-2.0) circle (3pt);
\draw (1.5, -2.0) -- (1.75, -2.0);
\filldraw (1.75,-2.0) circle (3pt);
\draw (1.5, -2.0) -- (1.5, -1.75);
\filldraw (1.5,-1.75) circle (3pt);
\draw (1.5, -2.0) -- (1.5, -2.25);
\filldraw (1.5,-2.25) circle (3pt);
\draw (1.0, -2.0) -- (1.0, -1.5);
\filldraw (1.0,-1.5) circle (3pt);
\draw (1.0, -1.5) -- (1.25, -1.5);
\filldraw (1.25,-1.5) circle (3pt);
\draw (1.0, -1.5) -- (0.75, -1.5);
\filldraw (0.75,-1.5) circle (3pt);
\draw (1.0, -1.5) -- (1.0, -1.25);
\filldraw (1.0,-1.25) circle (3pt);
\draw (1.0, -2.0) -- (1.0, -2.5);
\filldraw (1.0,-2.5) circle (3pt);
\draw (1.0, -2.5) -- (1.25, -2.5);
\filldraw (1.25,-2.5) circle (3pt);
\draw (1.0, -2.5) -- (0.75, -2.5);
\draw (1.0, -2.5) -- (1.0, -2.75);
\filldraw (1.0,-2.75) circle (3pt);
\draw (0, -2) -- (-1.0, -2.0);
\draw (-1.0, -2.0) -- (-1.5, -2.0);
\draw (-1.5, -2.0) -- (-1.75, -2.0);
\draw (-1.5, -2.0) -- (-1.5, -1.75);
\draw (-1.5, -2.0) -- (-1.5, -2.25);
\draw (-1.0, -2.0) -- (-1.0, -1.5);
\draw (-1.0, -1.5) -- (-0.75, -1.5);
\draw (-1.0, -1.5) -- (-1.25, -1.5);
\draw (-1.0, -1.5) -- (-1.0, -1.25);
\draw (-1.0, -2.0) -- (-1.0, -2.5);
\draw (-1.0, -2.5) -- (-0.75, -2.5);
\draw (-1.0, -2.5) -- (-1.25, -2.5);
\draw (-1.0, -2.5) -- (-1.0, -2.75);
\draw (0, -2) -- (0.0, -3.0);
\draw (0.0, -3.0) -- (0.5, -3.0);
\filldraw (0.5,-3.0) circle (3pt);
\draw (0.5, -3.0) -- (0.75, -3.0);
\filldraw (0.75,-3.0) circle (3pt);
\draw (0.5, -3.0) -- (0.5, -2.75);
\filldraw (0.5,-2.75) circle (3pt);
\draw (0.5, -3.0) -- (0.5, -3.25);
\filldraw (0.5,-3.25) circle (3pt);
\draw (0.0, -3.0) -- (-0.5, -3.0);
\draw (-0.5, -3.0) -- (-0.75, -3.0);
\draw (-0.5, -3.0) -- (-0.5, -2.75);
\draw (-0.5, -3.0) -- (-0.5, -3.25);
\draw (0.0, -3.0) -- (0.0, -3.5);
\draw (0.0, -3.5) -- (0.25, -3.5);
\filldraw (0.25,-3.5) circle (3pt);
\draw (0.0, -3.5) -- (-0.25, -3.5);
\draw (0.0, -3.5) -- (0.0, -3.75);

%% file: garbanzo.tex
\draw (0, 0) -- (3, 0);
\draw (3, 0) -- (4.5, 0.0);
\draw (4.5, 0.0) -- (5.25, 0.0);
\draw (5.25, 0.0) -- (5.625, 0.0);
\draw (5.25, 0.0) -- (5.25, 0.375);
\draw (5.25, 0.0) -- (5.25, -0.375);
\draw (4.5, 0.0) -- (4.5, 0.75);
\draw (4.5, 0.75) -- (4.875, 0.75);
\draw (4.5, 0.75) -- (4.125, 0.75);
\draw (4.5, 0.75) -- (4.5, 1.125);
\draw (4.5, 0.0) -- (4.5, -0.75);
\draw (4.5, -0.75) -- (4.875, -0.75);
\draw (4.5, -0.75) -- (4.125, -0.75);
\draw (4.5, -0.75) -- (4.5, -1.125);
\draw (3, 0) -- (3.0, 1.5);
\draw (3.0, 1.5) -- (3.75, 1.5);
\draw (3.75, 1.5) -- (4.125, 1.5);
\draw (3.75, 1.5) -- (3.75, 1.875);
\draw (3.75, 1.5) -- (3.75, 1.125);
\draw (3.0, 1.5) -- (2.25, 1.5);
\draw (2.25, 1.5) -- (1.875, 1.5);
\draw (2.25, 1.5) -- (2.25, 1.875);
\draw (2.25, 1.5) -- (2.25, 1.125);
\draw (3.0, 1.5) -- (3.0, 2.25);
\draw (3.0, 2.25) -- (3.375, 2.25);
\draw (3.0, 2.25) -- (2.625, 2.25);
\draw (3.0, 2.25) -- (3.0, 2.625);
\draw (3, 0) -- (3.0, -1.5);
\draw (3.0, -1.5) -- (3.75, -1.5);
\draw (3.75, -1.5) -- (4.125, -1.5);
\draw (3.75, -1.5) -- (3.75, -1.125);
\draw (3.75, -1.5) -- (3.75, -1.875);
\draw (3.0, -1.5) -- (2.25, -1.5);
\draw (2.25, -1.5) -- (1.875, -1.5);
\draw (2.25, -1.5) -- (2.25, -1.125);
\draw (2.25, -1.5) -- (2.25, -1.875);
\draw (3.0, -1.5) -- (3.0, -2.25);
\draw (3.0, -2.25) -- (3.375, -2.25);
\draw (3.0, -2.25) -- (2.625, -2.25);
\draw (3.0, -2.25) -- (3.0, -2.625);
\draw (0, 0) -- (-3, 0);
\draw (-3, 0) -- (-4.5, 0.0);
\draw (-4.5, 0.0) -- (-5.25, 0.0);
\draw (-5.25, 0.0) -- (-5.625, 0.0);
\draw (-5.25, 0.0) -- (-5.25, 0.375);
\draw (-5.25, 0.0) -- (-5.25, -0.375);
\draw (-4.5, 0.0) -- (-4.5, 0.75);
\draw (-4.5, 0.75) -- (-4.125, 0.75);
\draw (-4.5, 0.75) -- (-4.875, 0.75);
\draw (-4.5, 0.75) -- (-4.5, 1.125);
\draw (-4.5, 0.0) -- (-4.5, -0.75);
\draw (-4.5, -0.75) -- (-4.125, -0.75);
\draw (-4.5, -0.75) -- (-4.875, -0.75);
\draw (-4.5, -0.75) -- (-4.5, -1.125);
\draw (-3, 0) -- (-3.0, 1.5);
\draw (-3.0, 1.5) -- (-2.25, 1.5);
\draw (-2.25, 1.5) -- (-1.875, 1.5);
\draw (-2.25, 1.5) -- (-2.25, 1.875);
\draw (-2.25, 1.5) -- (-2.25, 1.125);
\draw (-3.0, 1.5) -- (-3.75, 1.5);
\draw (-3.75, 1.5) -- (-4.125, 1.5);
\draw (-3.75, 1.5) -- (-3.75, 1.875);
\draw (-3.75, 1.5) -- (-3.75, 1.125);
\draw (-3.0, 1.5) -- (-3.0, 2.25);
\draw (-3.0, 2.25) -- (-2.625, 2.25);
\draw (-3.0, 2.25) -- (-3.375, 2.25);
\draw (-3.0, 2.25) -- (-3.0, 2.625);
\draw (-3, 0) -- (-3.0, -1.5);
\draw (-3.0, -1.5) -- (-2.25, -1.5);
\draw (-2.25, -1.5) -- (-1.875, -1.5);
\draw (-2.25, -1.5) -- (-2.25, -1.125);
\draw (-2.25, -1.5) -- (-2.25, -1.875);
\draw (-3.0, -1.5) -- (-3.75, -1.5);
\draw (-3.75, -1.5) -- (-4.125, -1.5);
\draw (-3.75, -1.5) -- (-3.75, -1.125);
\draw (-3.75, -1.5) -- (-3.75, -1.875);
\draw (-3.0, -1.5) -- (-3.0, -2.25);
\draw (-3.0, -2.25) -- (-2.625, -2.25);
\draw (-3.0, -2.25) -- (-3.375, -2.25);
\draw (-3.0, -2.25) -- (-3.0, -2.625);
\draw (0, 0) -- (0, 3);
\draw (0, 3) -- (1.5, 3.0);
\draw (1.5, 3.0) -- (2.25, 3.0);
\draw (2.25, 3.0) -- (2.625, 3.0);
\draw (2.25, 3.0) -- (2.25, 3.375);
\draw (2.25, 3.0) -- (2.25, 2.625);
\draw (1.5, 3.0) -- (1.5, 3.75);
\draw (1.5, 3.75) -- (1.875, 3.75);
\draw (1.5, 3.75) -- (1.125, 3.75);
\draw (1.5, 3.75) -- (1.5, 4.125);
\draw (1.5, 3.0) -- (1.5, 2.25);
\draw (1.5, 2.25) -- (1.875, 2.25);
\draw (1.5, 2.25) -- (1.125, 2.25);
\draw (1.5, 2.25) -- (1.5, 1.875);
\draw (0, 3) -- (-1.5, 3.0);
\draw (-1.5, 3.0) -- (-2.25, 3.0);
\draw (-2.25, 3.0) -- (-2.625, 3.0);
\draw (-2.25, 3.0) -- (-2.25, 3.375);
\draw (-2.25, 3.0) -- (-2.25, 2.625);
\draw (-1.5, 3.0) -- (-1.5, 3.75);
\draw (-1.5, 3.75) -- (-1.125, 3.75);
\draw (-1.5, 3.75) -- (-1.875, 3.75);
\draw (-1.5, 3.75) -- (-1.5, 4.125);
\draw (-1.5, 3.0) -- (-1.5, 2.25);
\draw (-1.5, 2.25) -- (-1.125, 2.25);
\draw (-1.5, 2.25) -- (-1.875, 2.25);
\draw (-1.5, 2.25) -- (-1.5, 1.875);
\draw (0, 3) -- (0.0, 4.5);
\draw (0.0, 4.5) -- (0.75, 4.5);
\draw (0.75, 4.5) -- (1.125, 4.5);
\draw (0.75, 4.5) -- (0.75, 4.875);
\draw (0.75, 4.5) -- (0.75, 4.125);
\draw (0.0, 4.5) -- (-0.75, 4.5);
\draw (-0.75, 4.5) -- (-1.125, 4.5);
\draw (-0.75, 4.5) -- (-0.75, 4.875);
\draw (-0.75, 4.5) -- (-0.75, 4.125);
\draw (0.0, 4.5) -- (0.0, 5.25);
\draw (0.0, 5.25) -- (0.375, 5.25);
\draw (0.0, 5.25) -- (-0.375, 5.25);
\draw (0.0, 5.25) -- (0.0, 5.625);
\draw (0, 0) -- (0, -3);
\draw (0, -3) -- (1.5, -3.0);
\draw (1.5, -3.0) -- (2.25, -3.0);
\draw (2.25, -3.0) -- (2.625, -3.0);
\draw (2.25, -3.0) -- (2.25, -2.625);
\draw (2.25, -3.0) -- (2.25, -3.375);
\draw (1.5, -3.0) -- (1.5, -2.25);
\draw (1.5, -2.25) -- (1.875, -2.25);
\draw (1.5, -2.25) -- (1.125, -2.25);
\draw (1.5, -2.25) -- (1.5, -1.875);
\draw (1.5, -3.0) -- (1.5, -3.75);
\draw (1.5, -3.75) -- (1.875, -3.75);
\draw (1.5, -3.75) -- (1.125, -3.75);
\draw (1.5, -3.75) -- (1.5, -4.125);
\draw (0, -3) -- (-1.5, -3.0);
\draw (-1.5, -3.0) -- (-2.25, -3.0);
\draw (-2.25, -3.0) -- (-2.625, -3.0);
\draw (-2.25, -3.0) -- (-2.25, -2.625);
\draw (-2.25, -3.0) -- (-2.25, -3.375);
\draw (-1.5, -3.0) -- (-1.5, -2.25);
\draw (-1.5, -2.25) -- (-1.125, -2.25);
\draw (-1.5, -2.25) -- (-1.875, -2.25);
\draw (-1.5, -2.25) -- (-1.5, -1.875);
\draw (-1.5, -3.0) -- (-1.5, -3.75);
\draw (-1.5, -3.75) -- (-1.125, -3.75);
\draw (-1.5, -3.75) -- (-1.875, -3.75);
\draw (-1.5, -3.75) -- (-1.5, -4.125);
\draw (0, -3) -- (0.0, -4.5);
\draw (0.0, -4.5) -- (0.75, -4.5);
\draw (0.75, -4.5) -- (1.125, -4.5);
\draw (0.75, -4.5) -- (0.75, -4.125);
\draw (0.75, -4.5) -- (0.75, -4.875);
\draw (0.0, -4.5) -- (-0.75, -4.5);
\draw (-0.75, -4.5) -- (-1.125, -4.5);
\draw (-0.75, -4.5) -- (-0.75, -4.125);
\draw (-0.75, -4.5) -- (-0.75, -4.875);
\draw (0.0, -4.5) -- (0.0, -5.25);
\draw (0.0, -5.25) -- (0.375, -5.25);
\draw (0.0, -5.25) -- (-0.375, -5.25);
\draw (0.0, -5.25) -- (0.0, -5.625);
\node[draw,circle,fill=white,inner sep=0,minimum size=4mm] () at (0,0) {\footnotesize 15};
\node[draw,circle,fill=white,inner sep=0,minimum size=4mm] () at (3,0) {\footnotesize 27};
\node[draw,circle,fill=white,inner sep=0,minimum size=4mm] () at (4.5,0.0) {\footnotesize 39};
\node[draw,circle,fill=white,inner sep=0,minimum size=4mm] () at (5.25,0.0) {\footnotesize 49};
\node[draw,circle,fill=white,inner sep=0,minimum size=4mm] () at (5.625,0.0) {\footnotesize 53};
\filldraw (5.25,0.375) circle (0.1);
\filldraw (5.25,-0.375) circle (0.1);
\filldraw (4.5,0.75) circle (0.1);
\node[draw,circle,fill=white,inner sep=0,minimum size=4mm] () at (4.875,0.75) {\footnotesize 52};
\filldraw (4.125,0.75) circle (0.1);
\filldraw (4.5,1.125) circle (0.1);
\filldraw (4.5,-0.75) circle (0.1);
\node[draw,circle,fill=white,inner sep=0,minimum size=4mm] () at (4.875,-0.75) {\footnotesize 46};
\filldraw (4.125,-0.75) circle (0.1);
\filldraw (4.5,-1.125) circle (0.1);
\node[draw,circle,fill=white,inner sep=0,minimum size=4mm] () at (3.0,1.5) {\footnotesize 30};
\node[draw,circle,fill=white,inner sep=0,minimum size=4mm] () at (3.75,1.5) {\footnotesize 42};
\node[draw,circle,fill=white,inner sep=0,minimum size=4mm] () at (4.125,1.5) {\footnotesize 51};
\filldraw (3.75,1.875) circle (0.1);
\filldraw (3.75,1.125) circle (0.1);
\filldraw (2.25,1.5) circle (0.1);
\filldraw (1.875,1.5) circle (0.1);
\filldraw (2.25,1.875) circle (0.1);
\filldraw (2.25,1.125) circle (0.1);
\filldraw (3.0,2.25) circle (0.1);
\node[draw,circle,fill=white,inner sep=0,minimum size=4mm] () at (3.375,2.25) {\footnotesize 45};
\filldraw (2.625,2.25) circle (0.1);
\filldraw (3.0,2.625) circle (0.1);
\node[draw,circle,fill=white,inner sep=0,minimum size=4mm] () at (3.0,-1.5) {\footnotesize 24};
\node[draw,circle,fill=white,inner sep=0,minimum size=4mm] () at (3.75,-1.5) {\footnotesize 36};
\node[draw,circle,fill=white,inner sep=0,minimum size=4mm] () at (4.125,-1.5) {\footnotesize 47};
\filldraw (3.75,-1.125) circle (0.1);
\filldraw (3.75,-1.875) circle (0.1);
\filldraw (2.25,-1.5) circle (0.1);
\filldraw (1.875,-1.5) circle (0.1);
\filldraw (2.25,-1.125) circle (0.1);
\filldraw (2.25,-1.875) circle (0.1);
\filldraw (3.0,-2.25) circle (0.1);
\node[draw,circle,fill=white,inner sep=0,minimum size=4mm] () at (3.375,-2.25) {\footnotesize 33};
\filldraw (2.625,-2.25) circle (0.1);
\filldraw (3.0,-2.625) circle (0.1);
\node[draw,circle,fill=white,inner sep=0,minimum size=4mm] () at (-3,0) {\footnotesize 5};
\node[draw,circle,fill=white,inner sep=0,minimum size=4mm] () at (-4.5,0.0) {\footnotesize 1};
\filldraw (-5.25,0.0) circle (0.1);
\filldraw (-5.625,0.0) circle (0.1);
\filldraw (-5.25,0.375) circle (0.1);
\filldraw (-5.25,-0.375) circle (0.1);
\filldraw (-4.5,0.75) circle (0.1);
\node[draw,circle,fill=white,inner sep=0,minimum size=4mm] () at (-4.125,0.75) {\footnotesize 6};
\filldraw (-4.875,0.75) circle (0.1);
\filldraw (-4.5,1.125) circle (0.1);
\filldraw (-4.5,-0.75) circle (0.1);
\node[draw,circle,fill=white,inner sep=0,minimum size=4mm] () at (-4.125,-0.75) {\footnotesize 4};
\filldraw (-4.875,-0.75) circle (0.1);
\filldraw (-4.5,-1.125) circle (0.1);
\node[draw,circle,fill=white,inner sep=0,minimum size=4mm] () at (-3.0,1.5) {\footnotesize 7};
\node[draw,circle,fill=white,inner sep=0,minimum size=4mm] () at (-2.25,1.5) {\footnotesize 17};
\node[draw,circle,fill=white,inner sep=0,minimum size=4mm] () at (-1.875,1.5) {\footnotesize 28};
\filldraw (-2.25,1.875) circle (0.1);
\filldraw (-2.25,1.125) circle (0.1);
\filldraw (-3.75,1.5) circle (0.1);
\filldraw (-4.125,1.5) circle (0.1);
\filldraw (-3.75,1.875) circle (0.1);
\filldraw (-3.75,1.125) circle (0.1);
\filldraw (-3.0,2.25) circle (0.1);
\node[draw,circle,fill=white,inner sep=0,minimum size=4mm] () at (-2.625,2.25) {\footnotesize 19};
\filldraw (-3.375,2.25) circle (0.1);
\filldraw (-3.0,2.625) circle (0.1);
\node[draw,circle,fill=white,inner sep=0,minimum size=4mm] () at (-3.0,-1.5) {\footnotesize 3};
\node[draw,circle,fill=white,inner sep=0,minimum size=4mm] () at (-2.25,-1.5) {\footnotesize 13};
\node[draw,circle,fill=white,inner sep=0,minimum size=4mm] () at (-1.875,-1.5) {\footnotesize 26};
\filldraw (-2.25,-1.125) circle (0.1);
\filldraw (-2.25,-1.875) circle (0.1);
\filldraw (-3.75,-1.5) circle (0.1);
\filldraw (-4.125,-1.5) circle (0.1);
\filldraw (-3.75,-1.125) circle (0.1);
\filldraw (-3.75,-1.875) circle (0.1);
\filldraw (-3.0,-2.25) circle (0.1);
\node[draw,circle,fill=white,inner sep=0,minimum size=4mm] () at (-2.625,-2.25) {\footnotesize 11};
\filldraw (-3.375,-2.25) circle (0.1);
\filldraw (-3.0,-2.625) circle (0.1);
\node[draw,circle,fill=white,inner sep=0,minimum size=4mm] () at (0,3) {\footnotesize 18};
\node[draw,circle,fill=white,inner sep=0,minimum size=4mm] () at (1.5,3.0) {\footnotesize 29};
\node[draw,circle,fill=white,inner sep=0,minimum size=4mm] () at (2.25,3.0) {\footnotesize 41};
\node[draw,circle,fill=white,inner sep=0,minimum size=4mm] () at (2.625,3.0) {\footnotesize 50};
\filldraw (2.25,3.375) circle (0.1);
\filldraw (2.25,2.625) circle (0.1);
\filldraw (1.5,3.75) circle (0.1);
\node[draw,circle,fill=white,inner sep=0,minimum size=4mm] () at (1.875,3.75) {\footnotesize 44};
\filldraw (1.125,3.75) circle (0.1);
\filldraw (1.5,4.125) circle (0.1);
\filldraw (1.5,2.25) circle (0.1);
\node[draw,circle,fill=white,inner sep=0,minimum size=4mm] () at (1.875,2.25) {\footnotesize 38};
\filldraw (1.125,2.25) circle (0.1);
\filldraw (1.5,1.875) circle (0.1);
\node[draw,circle,fill=white,inner sep=0,minimum size=4mm] () at (-1.5,3.0) {\footnotesize 8};
\filldraw (-2.25,3.0) circle (0.1);
\filldraw (-2.625,3.0) circle (0.1);
\filldraw (-2.25,3.375) circle (0.1);
\filldraw (-2.25,2.625) circle (0.1);
\filldraw (-1.5,3.75) circle (0.1);
\node[draw,circle,fill=white,inner sep=0,minimum size=4mm] () at (-1.125,3.75) {\footnotesize 20};
\filldraw (-1.875,3.75) circle (0.1);
\filldraw (-1.5,4.125) circle (0.1);
\filldraw (-1.5,2.25) circle (0.1);
\node[draw,circle,fill=white,inner sep=0,minimum size=4mm] () at (-1.125,2.25) {\footnotesize 16};
\filldraw (-1.875,2.25) circle (0.1);
\filldraw (-1.5,1.875) circle (0.1);
\node[draw,circle,fill=white,inner sep=0,minimum size=4mm] () at (0.0,4.5) {\footnotesize 21};
\node[draw,circle,fill=white,inner sep=0,minimum size=4mm] () at (0.75,4.5) {\footnotesize 31};
\node[draw,circle,fill=white,inner sep=0,minimum size=4mm] () at (1.125,4.5) {\footnotesize 43};
\filldraw (0.75,4.875) circle (0.1);
\filldraw (0.75,4.125) circle (0.1);
\filldraw (-0.75,4.5) circle (0.1);
\filldraw (-1.125,4.5) circle (0.1);
\filldraw (-0.75,4.875) circle (0.1);
\filldraw (-0.75,4.125) circle (0.1);
\filldraw (0.0,5.25) circle (0.1);
\node[draw,circle,fill=white,inner sep=0,minimum size=4mm] () at (0.375,5.25) {\footnotesize 32};
\filldraw (-0.375,5.25) circle (0.1);
\filldraw (0.0,5.625) circle (0.1);
\node[draw,circle,fill=white,inner sep=0,minimum size=4mm] () at (0,-3) {\footnotesize 12};
\node[draw,circle,fill=white,inner sep=0,minimum size=4mm] () at (1.5,-3.0) {\footnotesize 25};
\node[draw,circle,fill=white,inner sep=0,minimum size=4mm] () at (2.25,-3.0) {\footnotesize 37};
\node[draw,circle,fill=white,inner sep=0,minimum size=4mm] () at (2.625,-3.0) {\footnotesize 48};
\filldraw (2.25,-2.625) circle (0.1);
\filldraw (2.25,-3.375) circle (0.1);
\filldraw (1.5,-2.25) circle (0.1);
\node[draw,circle,fill=white,inner sep=0,minimum size=4mm] () at (1.875,-2.25) {\footnotesize 40};
\filldraw (1.125,-2.25) circle (0.1);
\filldraw (1.5,-1.875) circle (0.1);
\filldraw (1.5,-3.75) circle (0.1);
\node[draw,circle,fill=white,inner sep=0,minimum size=4mm] () at (1.875,-3.75) {\footnotesize 34};
\filldraw (1.125,-3.75) circle (0.1);
\filldraw (1.5,-4.125) circle (0.1);
\node[draw,circle,fill=white,inner sep=0,minimum size=4mm] () at (-1.5,-3.0) {\footnotesize 2};
\filldraw (-2.25,-3.0) circle (0.1);
\filldraw (-2.625,-3.0) circle (0.1);
\filldraw (-2.25,-2.625) circle (0.1);
\filldraw (-2.25,-3.375) circle (0.1);
\filldraw (-1.5,-2.25) circle (0.1);
\node[draw,circle,fill=white,inner sep=0,minimum size=4mm] () at (-1.125,-2.25) {\footnotesize 14};
\filldraw (-1.875,-2.25) circle (0.1);
\filldraw (-1.5,-1.875) circle (0.1);
\filldraw (-1.5,-3.75) circle (0.1);
\node[draw,circle,fill=white,inner sep=0,minimum size=4mm] () at (-1.125,-3.75) {\footnotesize 10};
\filldraw (-1.875,-3.75) circle (0.1);
\filldraw (-1.5,-4.125) circle (0.1);
\node[draw,circle,fill=white,inner sep=0,minimum size=4mm] () at (0.0,-4.5) {\footnotesize 9};
\node[draw,circle,fill=white,inner sep=0,minimum size=4mm] () at (0.75,-4.5) {\footnotesize 23};
\node[draw,circle,fill=white,inner sep=0,minimum size=4mm] () at (1.125,-4.5) {\footnotesize 35};
\filldraw (0.75,-4.125) circle (0.1);
\filldraw (0.75,-4.875) circle (0.1);
\filldraw (-0.75,-4.5) circle (0.1);
\filldraw (-1.125,-4.5) circle (0.1);
\filldraw (-0.75,-4.125) circle (0.1);
\filldraw (-0.75,-4.875) circle (0.1);
\filldraw (0.0,-5.25) circle (0.1);
\node[draw,circle,fill=white,inner sep=0,minimum size=4mm] () at (0.375,-5.25) {\footnotesize 22};
\filldraw (-0.375,-5.25) circle (0.1);
\filldraw (0.0,-5.625) circle (0.1);

%% file: tree_convex_ordering.tex
\draw (5.1903, 0.0) -- (5.543247, 0.0);
\draw (5.1903, 0.0) -- (5.1903, 0.35294699999999996);
\draw (5.1903, 0.0) -- (5.1903, -0.35294699999999996);
\draw (4.47, 0.0) -- (5.1903, 0.0);
\draw (4.47, 0.7202999999999999) -- (4.822947, 0.7202999999999999);
\draw (4.47, 0.7202999999999999) -- (4.117052999999999, 0.7202999999999999);
\draw (4.47, 0.7202999999999999) -- (4.47, 1.0732469999999998);
\draw (4.47, 0.0) -- (4.47, 0.7202999999999999);
\draw (4.47, -0.7202999999999999) -- (4.822947, -0.7202999999999999);
\draw (4.47, -0.7202999999999999) -- (4.117052999999999, -0.7202999999999999);
\draw (4.47, -0.7202999999999999) -- (4.47, -1.0732469999999998);
\draw (4.47, 0.0) -- (4.47, -0.7202999999999999);
\draw (3, 0) -- (4.47, 0.0);
\draw (3.7203, 1.47) -- (4.073247, 1.47);
\draw (3.7203, 1.47) -- (3.7203, 1.8229469999999999);
\draw (3.7203, 1.47) -- (3.7203, 1.117053);
\draw (3.0, 1.47) -- (3.7203, 1.47);
\draw (2.2797, 1.47) -- (1.9267530000000002, 1.47);
\draw (2.2797, 1.47) -- (2.2797, 1.8229469999999999);
\draw (2.2797, 1.47) -- (2.2797, 1.117053);
\draw (3.0, 1.47) -- (2.2797, 1.47);
\draw (3.0, 2.1902999999999997) -- (3.352947, 2.1902999999999997);
\draw (3.0, 2.1902999999999997) -- (2.647053, 2.1902999999999997);
\draw (3.0, 2.1902999999999997) -- (3.0, 2.5432469999999996);
\draw (3.0, 1.47) -- (3.0, 2.1902999999999997);
\draw (3, 0) -- (3.0, 1.47);
\draw (3.7203, -1.47) -- (4.073247, -1.47);
\draw (3.7203, -1.47) -- (3.7203, -1.117053);
\draw (3.7203, -1.47) -- (3.7203, -1.8229469999999999);
\draw (3.0, -1.47) -- (3.7203, -1.47);
\draw (2.2797, -1.47) -- (1.9267530000000002, -1.47);
\draw (2.2797, -1.47) -- (2.2797, -1.117053);
\draw (2.2797, -1.47) -- (2.2797, -1.8229469999999999);
\draw (3.0, -1.47) -- (2.2797, -1.47);
\draw (3.0, -2.1902999999999997) -- (3.352947, -2.1902999999999997);
\draw (3.0, -2.1902999999999997) -- (2.647053, -2.1902999999999997);
\draw (3.0, -2.1902999999999997) -- (3.0, -2.5432469999999996);
\draw (3.0, -1.47) -- (3.0, -2.1902999999999997);
\draw (3, 0) -- (3.0, -1.47);
\draw (0, 0) -- (3, 0);
\draw (-5.1903, 0.0) -- (-5.543247, 0.0);
\draw (-5.1903, 0.0) -- (-5.1903, 0.35294699999999996);
\draw (-5.1903, 0.0) -- (-5.1903, -0.35294699999999996);
\draw (-4.47, 0.0) -- (-5.1903, 0.0);
\draw (-4.47, 0.7202999999999999) -- (-4.117052999999999, 0.7202999999999999);
\draw (-4.47, 0.7202999999999999) -- (-4.822947, 0.7202999999999999);
\draw (-4.47, 0.7202999999999999) -- (-4.47, 1.0732469999999998);
\draw (-4.47, 0.0) -- (-4.47, 0.7202999999999999);
\draw (-4.47, -0.7202999999999999) -- (-4.117052999999999, -0.7202999999999999);
\draw (-4.47, -0.7202999999999999) -- (-4.822947, -0.7202999999999999);
\draw (-4.47, -0.7202999999999999) -- (-4.47, -1.0732469999999998);
\draw (-4.47, 0.0) -- (-4.47, -0.7202999999999999);
\draw (-3, 0) -- (-4.47, 0.0);
\draw (-2.2797, 1.47) -- (-1.9267530000000002, 1.47);
\draw (-2.2797, 1.47) -- (-2.2797, 1.8229469999999999);
\draw (-2.2797, 1.47) -- (-2.2797, 1.117053);
\draw (-3.0, 1.47) -- (-2.2797, 1.47);
\draw (-3.7203, 1.47) -- (-4.073247, 1.47);
\draw (-3.7203, 1.47) -- (-3.7203, 1.8229469999999999);
\draw (-3.7203, 1.47) -- (-3.7203, 1.117053);
\draw (-3.0, 1.47) -- (-3.7203, 1.47);
\draw (-3.0, 2.1902999999999997) -- (-2.647053, 2.1902999999999997);
\draw (-3.0, 2.1902999999999997) -- (-3.352947, 2.1902999999999997);
\draw (-3.0, 2.1902999999999997) -- (-3.0, 2.5432469999999996);
\draw (-3.0, 1.47) -- (-3.0, 2.1902999999999997);
\draw (-3, 0) -- (-3.0, 1.47);
\draw (-2.2797, -1.47) -- (-1.9267530000000002, -1.47);
\draw (-2.2797, -1.47) -- (-2.2797, -1.117053);
\draw (-2.2797, -1.47) -- (-2.2797, -1.8229469999999999);
\draw (-3.0, -1.47) -- (-2.2797, -1.47);
\draw (-3.7203, -1.47) -- (-4.073247, -1.47);
\draw (-3.7203, -1.47) -- (-3.7203, -1.117053);
\draw (-3.7203, -1.47) -- (-3.7203, -1.8229469999999999);
\draw (-3.0, -1.47) -- (-3.7203, -1.47);
\draw (-3.0, -2.1902999999999997) -- (-2.647053, -2.1902999999999997);
\draw (-3.0, -2.1902999999999997) -- (-3.352947, -2.1902999999999997);
\draw (-3.0, -2.1902999999999997) -- (-3.0, -2.5432469999999996);
\draw (-3.0, -1.47) -- (-3.0, -2.1902999999999997);
\draw (-3, 0) -- (-3.0, -1.47);
\draw (0, 0) -- (-3, 0);
\draw (2.1902999999999997, 3.0) -- (2.5432469999999996, 3.0);
\draw (2.1902999999999997, 3.0) -- (2.1902999999999997, 3.352947);
\draw (2.1902999999999997, 3.0) -- (2.1902999999999997, 2.647053);
\draw (1.47, 3.0) -- (2.1902999999999997, 3.0);
\draw (1.47, 3.7203) -- (1.8229469999999999, 3.7203);
\draw (1.47, 3.7203) -- (1.117053, 3.7203);
\draw (1.47, 3.7203) -- (1.47, 4.073247);
\draw (1.47, 3.0) -- (1.47, 3.7203);
\draw (1.47, 2.2797) -- (1.8229469999999999, 2.2797);
\draw (1.47, 2.2797) -- (1.117053, 2.2797);
\draw (1.47, 2.2797) -- (1.47, 1.9267530000000002);
\draw (1.47, 3.0) -- (1.47, 2.2797);
\draw (0, 3) -- (1.47, 3.0);
\draw (-2.1902999999999997, 3.0) -- (-2.5432469999999996, 3.0);
\draw (-2.1902999999999997, 3.0) -- (-2.1902999999999997, 3.352947);
\draw (-2.1902999999999997, 3.0) -- (-2.1902999999999997, 2.647053);
\draw (-1.47, 3.0) -- (-2.1902999999999997, 3.0);
\draw (-1.47, 3.7203) -- (-1.117053, 3.7203);
\draw (-1.47, 3.7203) -- (-1.8229469999999999, 3.7203);
\draw (-1.47, 3.7203) -- (-1.47, 4.073247);
\draw (-1.47, 3.0) -- (-1.47, 3.7203);
\draw (-1.47, 2.2797) -- (-1.117053, 2.2797);
\draw (-1.47, 2.2797) -- (-1.8229469999999999, 2.2797);
\draw (-1.47, 2.2797) -- (-1.47, 1.9267530000000002);
\draw (-1.47, 3.0) -- (-1.47, 2.2797);
\draw (0, 3) -- (-1.47, 3.0);
\draw (0.7202999999999999, 4.47) -- (1.0732469999999998, 4.47);
\draw (0.7202999999999999, 4.47) -- (0.7202999999999999, 4.822947);
\draw (0.7202999999999999, 4.47) -- (0.7202999999999999, 4.117052999999999);
\draw (0.0, 4.47) -- (0.7202999999999999, 4.47);
\draw (-0.7202999999999999, 4.47) -- (-1.0732469999999998, 4.47);
\draw (-0.7202999999999999, 4.47) -- (-0.7202999999999999, 4.822947);
\draw (-0.7202999999999999, 4.47) -- (-0.7202999999999999, 4.117052999999999);
\draw (0.0, 4.47) -- (-0.7202999999999999, 4.47);
\draw (0.0, 5.1903) -- (0.35294699999999996, 5.1903);
\draw (0.0, 5.1903) -- (-0.35294699999999996, 5.1903);
\draw (0.0, 5.1903) -- (0.0, 5.543247);
\draw (0.0, 4.47) -- (0.0, 5.1903);
\draw (0, 3) -- (0.0, 4.47);
\draw (0, 0) -- (0, 3);
\draw (2.1902999999999997, -3.0) -- (2.5432469999999996, -3.0);
\draw (2.1902999999999997, -3.0) -- (2.1902999999999997, -2.647053);
\draw (2.1902999999999997, -3.0) -- (2.1902999999999997, -3.352947);
\draw (1.47, -3.0) -- (2.1902999999999997, -3.0);
\draw (1.47, -2.2797) -- (1.8229469999999999, -2.2797);
\draw (1.47, -2.2797) -- (1.117053, -2.2797);
\draw (1.47, -2.2797) -- (1.47, -1.9267530000000002);
\draw (1.47, -3.0) -- (1.47, -2.2797);
\draw (1.47, -3.7203) -- (1.8229469999999999, -3.7203);
\draw (1.47, -3.7203) -- (1.117053, -3.7203);
\draw (1.47, -3.7203) -- (1.47, -4.073247);
\draw (1.47, -3.0) -- (1.47, -3.7203);
\draw (0, -3) -- (1.47, -3.0);
\draw (-2.1902999999999997, -3.0) -- (-2.5432469999999996, -3.0);
\draw (-2.1902999999999997, -3.0) -- (-2.1902999999999997, -2.647053);
\draw (-2.1902999999999997, -3.0) -- (-2.1902999999999997, -3.352947);
\draw (-1.47, -3.0) -- (-2.1902999999999997, -3.0);
\draw (-1.47, -2.2797) -- (-1.117053, -2.2797);
\draw (-1.47, -2.2797) -- (-1.8229469999999999, -2.2797);
\draw (-1.47, -2.2797) -- (-1.47, -1.9267530000000002);
\draw (-1.47, -3.0) -- (-1.47, -2.2797);
\draw (-1.47, -3.7203) -- (-1.117053, -3.7203);
\draw (-1.47, -3.7203) -- (-1.8229469999999999, -3.7203);
\draw (-1.47, -3.7203) -- (-1.47, -4.073247);
\draw (-1.47, -3.0) -- (-1.47, -3.7203);
\draw (0, -3) -- (-1.47, -3.0);
\draw (0.7202999999999999, -4.47) -- (1.0732469999999998, -4.47);
\draw (0.7202999999999999, -4.47) -- (0.7202999999999999, -4.117052999999999);
\draw (0.7202999999999999, -4.47) -- (0.7202999999999999, -4.822947);
\draw (0.0, -4.47) -- (0.7202999999999999, -4.47);
\draw (-0.7202999999999999, -4.47) -- (-1.0732469999999998, -4.47);
\draw (-0.7202999999999999, -4.47) -- (-0.7202999999999999, -4.117052999999999);
\draw (-0.7202999999999999, -4.47) -- (-0.7202999999999999, -4.822947);
\draw (0.0, -4.47) -- (-0.7202999999999999, -4.47);
\draw (0.0, -5.1903) -- (0.35294699999999996, -5.1903);
\draw (0.0, -5.1903) -- (-0.35294699999999996, -5.1903);
\draw (0.0, -5.1903) -- (0.0, -5.543247);
\draw (0.0, -4.47) -- (0.0, -5.1903);
\draw (0, -3) -- (0.0, -4.47);
\draw (0, 0) -- (0, -3);
\node[draw,circle,fill=white,inner sep=0,minimum size=4mm,line width=0.5] () at (0,0) {\tiny{13}};
\node[draw,circle,fill=black!38.75!white,inner sep=0,minimum size=4mm,line width=0.5] () at (5.543247,0.0) {\tiny{156}};
\node[draw,circle,fill=black!23.25!white,inner sep=0,minimum size=4mm,line width=1.5] () at (5.1903,0.35294699999999996) {\tiny{94}};
\node[draw,circle,fill=black!26.5!white,inner sep=0,minimum size=4mm,line width=1.5] () at (5.1903,-0.35294699999999996) {\tiny{107}};
\node[draw,circle,fill=black!16.25!white,inner sep=0,minimum size=4mm,line width=1.5] () at (5.1903,0.0) {\tiny{66}};
\node[draw,circle,fill=black!26.0!white,inner sep=0,minimum size=4mm,line width=1.5] () at (4.822947,0.7202999999999999) {\tiny{105}};
\node[draw,circle,fill=black!24.0!white,inner sep=0,minimum size=4mm,line width=1.5] () at (4.117052999999999,0.7202999999999999) {\tiny{97}};
\node[draw,circle,fill=black!36.0!white,inner sep=0,minimum size=4mm,line width=0.5] () at (4.47,1.0732469999999998) {\tiny{145}};
\node[draw,circle,fill=black!19.25!white,inner sep=0,minimum size=4mm,line width=0.5] () at (4.47,0.7202999999999999) {\tiny{78}};
\node[draw,circle,fill=black!34.5!white,inner sep=0,minimum size=4mm,line width=0.5] () at (4.822947,-0.7202999999999999) {\tiny{139}};
\node[draw,circle,fill=black!31.5!white,inner sep=0,minimum size=4mm,line width=1.5] () at (4.117052999999999,-0.7202999999999999) {\tiny{127}};
\node[draw,circle,fill=black!21.25!white,inner sep=0,minimum size=4mm,line width=1.5] () at (4.47,-1.0732469999999998) {\tiny{86}};
\node[draw,circle,fill=black!16.5!white,inner sep=0,minimum size=4mm,line width=1.5] () at (4.47,-0.7202999999999999) {\tiny{67}};
\node[draw,circle,fill=black!7.75!white,inner sep=0,minimum size=4mm,line width=1.5] () at (4.47,0.0) {\tiny{32}};
\node[draw,circle,fill=black!28.25!white,inner sep=0,minimum size=4mm,line width=1.5] () at (4.073247,1.47) {\tiny{114}};
\node[draw,circle,fill=black!26.25!white,inner sep=0,minimum size=4mm,line width=1.5] () at (3.7203,1.8229469999999999) {\tiny{106}};
\node[draw,circle,fill=black!38.0!white,inner sep=0,minimum size=4mm,line width=0.5] () at (3.7203,1.117053) {\tiny{153}};
\node[draw,circle,fill=black!15.75!white,inner sep=0,minimum size=4mm,line width=1.5] () at (3.7203,1.47) {\tiny{64}};
\node[draw,circle,fill=black!33.75!white,inner sep=0,minimum size=4mm,line width=1.5] () at (1.9267530000000002,1.47) {\tiny{136}};
\node[draw,circle,fill=black!39.25!white,inner sep=0,minimum size=4mm,line width=0.5] () at (2.2797,1.8229469999999999) {\tiny{158}};
\node[draw,circle,fill=black!21.0!white,inner sep=0,minimum size=4mm,line width=1.5] () at (2.2797,1.117053) {\tiny{85}};
\node[draw,circle,fill=black!14.5!white,inner sep=0,minimum size=4mm,line width=1.5] () at (2.2797,1.47) {\tiny{59}};
\node[draw,circle,fill=black!31.25!white,inner sep=0,minimum size=4mm,line width=0.5] () at (3.352947,2.1902999999999997) {\tiny{126}};
\node[draw,circle,fill=black!30.5!white,inner sep=0,minimum size=4mm,line width=1.5] () at (2.647053,2.1902999999999997) {\tiny{123}};
\node[draw,circle,fill=black!28.5!white,inner sep=0,minimum size=4mm,line width=1.5] () at (3.0,2.5432469999999996) {\tiny{115}};
\node[draw,circle,fill=black!16.75!white,inner sep=0,minimum size=4mm,line width=0.5] () at (3.0,2.1902999999999997) {\tiny{68}};
\node[draw,circle,fill=black!11.0!white,inner sep=0,minimum size=4mm,line width=0.5] () at (3.0,1.47) {\tiny{45}};
\node[draw,circle,fill=black!27.75!white,inner sep=0,minimum size=4mm,line width=1.5] () at (4.073247,-1.47) {\tiny{112}};
\node[draw,circle,fill=black!22.5!white,inner sep=0,minimum size=4mm,line width=1.5] () at (3.7203,-1.117053) {\tiny{91}};
\node[draw,circle,fill=black!33.5!white,inner sep=0,minimum size=4mm,line width=0.5] () at (3.7203,-1.8229469999999999) {\tiny{135}};
\node[draw,circle,fill=black!13.25!white,inner sep=0,minimum size=4mm,line width=1.5] () at (3.7203,-1.47) {\tiny{54}};
\node[draw,circle,fill=black!26.75!white,inner sep=0,minimum size=4mm,line width=1.5] () at (1.9267530000000002,-1.47) {\tiny{108}};
\node[draw,circle,fill=black!34.25!white,inner sep=0,minimum size=4mm,line width=0.5] () at (2.2797,-1.117053) {\tiny{138}};
\node[draw,circle,fill=black!30.25!white,inner sep=0,minimum size=4mm,line width=1.5] () at (2.2797,-1.8229469999999999) {\tiny{122}};
\node[draw,circle,fill=black!13.0!white,inner sep=0,minimum size=4mm,line width=1.5] () at (2.2797,-1.47) {\tiny{53}};
\node[draw,circle,fill=black!31.75!white,inner sep=0,minimum size=4mm,line width=1.5] () at (3.352947,-2.1902999999999997) {\tiny{128}};
\node[draw,circle,fill=black!37.0!white,inner sep=0,minimum size=4mm,line width=0.5] () at (2.647053,-2.1902999999999997) {\tiny{149}};
\node[draw,circle,fill=black!20.5!white,inner sep=0,minimum size=4mm,line width=1.5] () at (3.0,-2.5432469999999996) {\tiny{83}};
\node[draw,circle,fill=black!14.75!white,inner sep=0,minimum size=4mm,line width=0.5] () at (3.0,-2.1902999999999997) {\tiny{60}};
\node[draw,circle,fill=black!8.25!white,inner sep=0,minimum size=4mm,line width=1.5] () at (3.0,-1.47) {\tiny{34}};
\node[draw,circle,fill=black!5.0!white,inner sep=0,minimum size=4mm,line width=1.5] () at (3,0) {\tiny{21}};
\node[draw,circle,fill=black!25.25!white,inner sep=0,minimum size=4mm,line width=1.5] () at (-5.543247,0.0) {\tiny{102}};
\node[draw,circle,fill=black!36.25!white,inner sep=0,minimum size=4mm,line width=0.5] () at (-5.1903,0.35294699999999996) {\tiny{146}};
\node[draw,circle,fill=black!21.5!white,inner sep=0,minimum size=4mm,line width=1.5] () at (-5.1903,-0.35294699999999996) {\tiny{87}};
\node[draw,circle,fill=black!14.0!white,inner sep=0,minimum size=4mm,line width=1.5] () at (-5.1903,0.0) {\tiny{57}};
\node[draw,circle,fill=black!36.5!white,inner sep=0,minimum size=4mm,line width=0.5] () at (-4.117052999999999,0.7202999999999999) {\tiny{147}};
\node[draw,circle,fill=black!34.75!white,inner sep=0,minimum size=4mm,line width=1.5] () at (-4.822947,0.7202999999999999) {\tiny{140}};
\node[draw,circle,fill=black!29.75!white,inner sep=0,minimum size=4mm,line width=1.5] () at (-4.47,1.0732469999999998) {\tiny{120}};
\node[draw,circle,fill=black!17.25!white,inner sep=0,minimum size=4mm,line width=0.5] () at (-4.47,0.7202999999999999) {\tiny{70}};
\node[draw,circle,fill=black!22.75!white,inner sep=0,minimum size=4mm,line width=1.5] () at (-4.117052999999999,-0.7202999999999999) {\tiny{92}};
\node[draw,circle,fill=black!27.0!white,inner sep=0,minimum size=4mm,line width=0.5] () at (-4.822947,-0.7202999999999999) {\tiny{109}};
\node[draw,circle,fill=black!24.25!white,inner sep=0,minimum size=4mm,line width=1.5] () at (-4.47,-1.0732469999999998) {\tiny{98}};
\node[draw,circle,fill=black!13.75!white,inner sep=0,minimum size=4mm,line width=1.5] () at (-4.47,-0.7202999999999999) {\tiny{56}};
\node[draw,circle,fill=black!8.0!white,inner sep=0,minimum size=4mm,line width=1.5] () at (-4.47,0.0) {\tiny{33}};
\node[draw,circle,fill=black!20.0!white,inner sep=0,minimum size=4mm,line width=1.5] () at (-1.9267530000000002,1.47) {\tiny{81}};
\node[draw,circle,fill=black!37.25!white,inner sep=0,minimum size=4mm,line width=0.5] () at (-2.2797,1.8229469999999999) {\tiny{150}};
\node[draw,circle,fill=black!28.75!white,inner sep=0,minimum size=4mm,line width=1.5] () at (-2.2797,1.117053) {\tiny{116}};
\node[draw,circle,fill=black!15.25!white,inner sep=0,minimum size=4mm,line width=1.5] () at (-2.2797,1.47) {\tiny{62}};
\node[draw,circle,fill=black!35.25!white,inner sep=0,minimum size=4mm,line width=0.5] () at (-4.073247,1.47) {\tiny{142}};
\node[draw,circle,fill=black!20.75!white,inner sep=0,minimum size=4mm,line width=1.5] () at (-3.7203,1.8229469999999999) {\tiny{84}};
\node[draw,circle,fill=black!33.0!white,inner sep=0,minimum size=4mm,line width=1.5] () at (-3.7203,1.117053) {\tiny{133}};
\node[draw,circle,fill=black!18.5!white,inner sep=0,minimum size=4mm,line width=1.5] () at (-3.7203,1.47) {\tiny{75}};
\node[draw,circle,fill=black!32.0!white,inner sep=0,minimum size=4mm,line width=1.5] () at (-2.647053,2.1902999999999997) {\tiny{129}};
\node[draw,circle,fill=black!32.5!white,inner sep=0,minimum size=4mm,line width=1.5] () at (-3.352947,2.1902999999999997) {\tiny{131}};
\node[draw,circle,fill=black!34.0!white,inner sep=0,minimum size=4mm,line width=0.5] () at (-3.0,2.5432469999999996) {\tiny{137}};
\node[draw,circle,fill=black!19.0!white,inner sep=0,minimum size=4mm,line width=0.5] () at (-3.0,2.1902999999999997) {\tiny{77}};
\node[draw,circle,fill=black!10.25!white,inner sep=0,minimum size=4mm,line width=0.5] () at (-3.0,1.47) {\tiny{42}};
\node[draw,circle,fill=black!37.75!white,inner sep=0,minimum size=4mm,line width=0.5] () at (-1.9267530000000002,-1.47) {\tiny{152}};
\node[draw,circle,fill=black!27.25!white,inner sep=0,minimum size=4mm,line width=1.5] () at (-2.2797,-1.117053) {\tiny{110}};
\node[draw,circle,fill=black!35.0!white,inner sep=0,minimum size=4mm,line width=1.5] () at (-2.2797,-1.8229469999999999) {\tiny{141}};
\node[draw,circle,fill=black!17.75!white,inner sep=0,minimum size=4mm,line width=1.5] () at (-2.2797,-1.47) {\tiny{72}};
\node[draw,circle,fill=black!23.0!white,inner sep=0,minimum size=4mm,line width=1.5] () at (-4.073247,-1.47) {\tiny{93}};
\node[draw,circle,fill=black!29.5!white,inner sep=0,minimum size=4mm,line width=0.5] () at (-3.7203,-1.117053) {\tiny{119}};
\node[draw,circle,fill=black!27.5!white,inner sep=0,minimum size=4mm,line width=1.5] () at (-3.7203,-1.8229469999999999) {\tiny{111}};
\node[draw,circle,fill=black!18.75!white,inner sep=0,minimum size=4mm,line width=1.5] () at (-3.7203,-1.47) {\tiny{76}};
\node[draw,circle,fill=black!24.75!white,inner sep=0,minimum size=4mm,line width=1.5] () at (-2.647053,-2.1902999999999997) {\tiny{100}};
\node[draw,circle,fill=black!30.0!white,inner sep=0,minimum size=4mm,line width=0.5] () at (-3.352947,-2.1902999999999997) {\tiny{121}};
\node[draw,circle,fill=black!24.5!white,inner sep=0,minimum size=4mm,line width=1.5] () at (-3.0,-2.5432469999999996) {\tiny{99}};
\node[draw,circle,fill=black!19.75!white,inner sep=0,minimum size=4mm,line width=0.5] () at (-3.0,-2.1902999999999997) {\tiny{80}};
\node[draw,circle,fill=black!9.75!white,inner sep=0,minimum size=4mm,line width=1.5] () at (-3.0,-1.47) {\tiny{40}};
\node[draw,circle,fill=black!5.25!white,inner sep=0,minimum size=4mm,line width=0.5] () at (-3,0) {\tiny{22}};
\node[draw,circle,fill=black!38.25!white,inner sep=0,minimum size=4mm,line width=0.5] () at (2.5432469999999996,3.0) {\tiny{154}};
\node[draw,circle,fill=black!29.0!white,inner sep=0,minimum size=4mm,line width=1.5] () at (2.1902999999999997,3.352947) {\tiny{117}};
\node[draw,circle,fill=black!35.5!white,inner sep=0,minimum size=4mm,line width=1.5] () at (2.1902999999999997,2.647053) {\tiny{143}};
\node[draw,circle,fill=black!13.5!white,inner sep=0,minimum size=4mm,line width=1.5] () at (2.1902999999999997,3.0) {\tiny{55}};
\node[draw,circle,fill=black!37.5!white,inner sep=0,minimum size=4mm,line width=1.5] () at (1.8229469999999999,3.7203) {\tiny{151}};
\node[draw,circle,fill=black!40.0!white,inner sep=0,minimum size=4mm,line width=0.5] () at (1.117053,3.7203) {\tiny{161}};
\node[draw,circle,fill=black!38.5!white,inner sep=0,minimum size=4mm,line width=1.5] () at (1.47,4.073247) {\tiny{155}};
\node[draw,circle,fill=black!17.5!white,inner sep=0,minimum size=4mm,line width=0.5] () at (1.47,3.7203) {\tiny{71}};
\node[draw,circle,fill=black!31.0!white,inner sep=0,minimum size=4mm,line width=0.5] () at (1.8229469999999999,2.2797) {\tiny{125}};
\node[draw,circle,fill=black!25.5!white,inner sep=0,minimum size=4mm,line width=1.5] () at (1.117053,2.2797) {\tiny{103}};
\node[draw,circle,fill=black!22.25!white,inner sep=0,minimum size=4mm,line width=1.5] () at (1.47,1.9267530000000002) {\tiny{90}};
\node[draw,circle,fill=black!12.75!white,inner sep=0,minimum size=4mm,line width=1.5] () at (1.47,2.2797) {\tiny{52}};
\node[draw,circle,fill=black!11.25!white,inner sep=0,minimum size=4mm,line width=1.5] () at (1.47,3.0) {\tiny{46}};
\node[draw,circle,fill=black!32.25!white,inner sep=0,minimum size=4mm,line width=0.5] () at (-2.5432469999999996,3.0) {\tiny{130}};
\node[draw,circle,fill=black!25.75!white,inner sep=0,minimum size=4mm,line width=1.5] () at (-2.1902999999999997,3.352947) {\tiny{104}};
\node[draw,circle,fill=black!29.25!white,inner sep=0,minimum size=4mm,line width=1.5] () at (-2.1902999999999997,2.647053) {\tiny{118}};
\node[draw,circle,fill=black!15.0!white,inner sep=0,minimum size=4mm,line width=1.5] () at (-2.1902999999999997,3.0) {\tiny{61}};
\node[draw,circle,fill=black!33.25!white,inner sep=0,minimum size=4mm,line width=0.5] () at (-1.117053,3.7203) {\tiny{134}};
\node[draw,circle,fill=black!20.25!white,inner sep=0,minimum size=4mm,line width=1.5] () at (-1.8229469999999999,3.7203) {\tiny{82}};
\node[draw,circle,fill=black!22.0!white,inner sep=0,minimum size=4mm,line width=1.5] () at (-1.47,4.073247) {\tiny{89}};
\node[draw,circle,fill=black!18.25!white,inner sep=0,minimum size=4mm,line width=0.5] () at (-1.47,3.7203) {\tiny{74}};
\node[draw,circle,fill=black!21.75!white,inner sep=0,minimum size=4mm,line width=1.5] () at (-1.117053,2.2797) {\tiny{88}};
\node[draw,circle,fill=black!30.75!white,inner sep=0,minimum size=4mm,line width=0.5] () at (-1.8229469999999999,2.2797) {\tiny{124}};
\node[draw,circle,fill=black!23.75!white,inner sep=0,minimum size=4mm,line width=1.5] () at (-1.47,1.9267530000000002) {\tiny{96}};
\node[draw,circle,fill=black!16.0!white,inner sep=0,minimum size=4mm,line width=1.5] () at (-1.47,2.2797) {\tiny{65}};
\node[draw,circle,fill=black!12.25!white,inner sep=0,minimum size=4mm,line width=0.5] () at (-1.47,3.0) {\tiny{50}};
\node[draw,circle,fill=black!39.0!white,inner sep=0,minimum size=4mm,line width=1.5] () at (1.0732469999999998,4.47) {\tiny{157}};
\node[draw,circle,fill=black!39.75!white,inner sep=0,minimum size=4mm,line width=0.5] () at (0.7202999999999999,4.822947) {\tiny{160}};
\node[draw,circle,fill=black!32.75!white,inner sep=0,minimum size=4mm,line width=1.5] () at (0.7202999999999999,4.117052999999999) {\tiny{132}};
\node[draw,circle,fill=black!18.0!white,inner sep=0,minimum size=4mm,line width=0.5] () at (0.7202999999999999,4.47) {\tiny{73}};
\node[draw,circle,fill=black!25.0!white,inner sep=0,minimum size=4mm,line width=1.5] () at (-1.0732469999999998,4.47) {\tiny{101}};
\node[draw,circle,fill=black!23.5!white,inner sep=0,minimum size=4mm,line width=1.5] () at (-0.7202999999999999,4.822947) {\tiny{95}};
\node[draw,circle,fill=black!36.75!white,inner sep=0,minimum size=4mm,line width=0.5] () at (-0.7202999999999999,4.117052999999999) {\tiny{148}};
\node[draw,circle,fill=black!12.5!white,inner sep=0,minimum size=4mm,line width=1.5] () at (-0.7202999999999999,4.47) {\tiny{51}};
\node[draw,circle,fill=black!35.75!white,inner sep=0,minimum size=4mm,line width=1.5] () at (0.35294699999999996,5.1903) {\tiny{144}};
\node[draw,circle,fill=black!39.5!white,inner sep=0,minimum size=4mm,line width=0.5] () at (-0.35294699999999996,5.1903) {\tiny{159}};
\node[draw,circle,fill=black!28.0!white,inner sep=0,minimum size=4mm,line width=1.5] () at (0.0,5.543247) {\tiny{113}};
\node[draw,circle,fill=black!15.5!white,inner sep=0,minimum size=4mm,line width=1.5] () at (0.0,5.1903) {\tiny{63}};
\node[draw,circle,fill=black!7.5!white,inner sep=0,minimum size=4mm,line width=1.5] () at (0.0,4.47) {\tiny{31}};
\node[draw,circle,fill=black!3.5!white,inner sep=0,minimum size=4mm,line width=1.5] () at (0,3) {\tiny{15}};
\node[draw,circle,fill=black!10.0!white,inner sep=0,minimum size=4mm,line width=0.5] () at (2.5432469999999996,-3.0) {\tiny{41}};
\node[draw,circle,fill=black!7.25!white,inner sep=0,minimum size=4mm,line width=1.5] () at (2.1902999999999997,-2.647053) {\tiny{30}};
\node[draw,circle,fill=black!9.25!white,inner sep=0,minimum size=4mm,line width=1.5] () at (2.1902999999999997,-3.352947) {\tiny{38}};
\node[draw,circle,fill=black!3.75!white,inner sep=0,minimum size=4mm,line width=1.5] () at (2.1902999999999997,-3.0) {\tiny{16}};
\node[draw,circle,fill=black!8.5!white,inner sep=0,minimum size=4mm,line width=1.5] () at (1.8229469999999999,-2.2797) {\tiny{35}};
\node[draw,circle,fill=black!8.75!white,inner sep=0,minimum size=4mm,line width=1.5] () at (1.117053,-2.2797) {\tiny{36}};
\node[draw,circle,fill=black!14.25!white,inner sep=0,minimum size=4mm,line width=0.5] () at (1.47,-1.9267530000000002) {\tiny{58}};
\node[draw,circle,fill=black!6.0!white,inner sep=0,minimum size=4mm,line width=0.5] () at (1.47,-2.2797) {\tiny{25}};
\node[draw,circle,fill=black!6.75!white,inner sep=0,minimum size=4mm,line width=1.5] () at (1.8229469999999999,-3.7203) {\tiny{28}};
\node[draw,circle,fill=black!7.0!white,inner sep=0,minimum size=4mm,line width=1.5] () at (1.117053,-3.7203) {\tiny{29}};
\node[draw,circle,fill=black!9.0!white,inner sep=0,minimum size=4mm,line width=0.5] () at (1.47,-4.073247) {\tiny{37}};
\node[draw,circle,fill=black!4.0!white,inner sep=0,minimum size=4mm,line width=1.5] () at (1.47,-3.7203) {\tiny{17}};
\node[draw,circle,fill=black!2.75!white,inner sep=0,minimum size=4mm,line width=1.5] () at (1.47,-3.0) {\tiny{12}};
\node[draw,circle,fill=black!2.0!white,inner sep=0,minimum size=4mm,line width=1.5] () at (-2.5432469999999996,-3.0) {\tiny{9}};
\node[draw,circle,fill=black!5.5!white,inner sep=0,minimum size=4mm,line width=0.5] () at (-2.1902999999999997,-2.647053) {\tiny{23}};
\node[draw,circle,fill=black!3.25!white,inner sep=0,minimum size=4mm,line width=1.5] () at (-2.1902999999999997,-3.352947) {\tiny{14}};
\node[draw,circle,fill=black!0.5!white,inner sep=0,minimum size=4mm,line width=1.5] () at (-2.1902999999999997,-3.0) {\tiny{3}};
\node[draw,circle,fill=black!1.25!white,inner sep=0,minimum size=4mm,line width=1.5] () at (-1.117053,-2.2797) {\tiny{6}};
\node[draw,circle,fill=black!1.5!white,inner sep=0,minimum size=4mm,line width=0.5] () at (-1.8229469999999999,-2.2797) {\tiny{7}};
\node[draw,circle,fill=black!1.0!white,inner sep=0,minimum size=4mm,line width=1.5] () at (-1.47,-1.9267530000000002) {\tiny{5}};
\node[draw,circle,fill=black!0.0!white,inner sep=0,minimum size=4mm,line width=1.5] () at (-1.47,-2.2797) {\tiny{1}};
\node[draw,circle,fill=black!2.25!white,inner sep=0,minimum size=4mm,line width=1.5] () at (-1.117053,-3.7203) {\tiny{10}};
\node[draw,circle,fill=black!4.25!white,inner sep=0,minimum size=4mm,line width=1.5] () at (-1.8229469999999999,-3.7203) {\tiny{18}};
\node[draw,circle,fill=black!5.75!white,inner sep=0,minimum size=4mm,line width=0.5] () at (-1.47,-4.073247) {\tiny{24}};
\node[draw,circle,fill=black!1.75!white,inner sep=0,minimum size=4mm,line width=0.5] () at (-1.47,-3.7203) {\tiny{8}};
\node[draw,circle,fill=black!0.25!white,inner sep=0,minimum size=4mm,line width=1.5] () at (-1.47,-3.0) {\tiny{2}};
\node[draw,circle,fill=black!6.5!white,inner sep=0,minimum size=4mm,line width=1.5] () at (1.0732469999999998,-4.47) {\tiny{27}};
\node[draw,circle,fill=black!10.5!white,inner sep=0,minimum size=4mm,line width=0.5] () at (0.7202999999999999,-4.117052999999999) {\tiny{43}};
\node[draw,circle,fill=black!9.5!white,inner sep=0,minimum size=4mm,line width=1.5] () at (0.7202999999999999,-4.822947) {\tiny{39}};
\node[draw,circle,fill=black!4.75!white,inner sep=0,minimum size=4mm,line width=1.5] () at (0.7202999999999999,-4.47) {\tiny{20}};
\node[draw,circle,fill=black!17.0!white,inner sep=0,minimum size=4mm,line width=0.5] () at (-1.0732469999999998,-4.47) {\tiny{69}};
\node[draw,circle,fill=black!11.75!white,inner sep=0,minimum size=4mm,line width=1.5] () at (-0.7202999999999999,-4.117052999999999) {\tiny{48}};
\node[draw,circle,fill=black!10.75!white,inner sep=0,minimum size=4mm,line width=1.5] () at (-0.7202999999999999,-4.822947) {\tiny{44}};
\node[draw,circle,fill=black!6.25!white,inner sep=0,minimum size=4mm,line width=0.5] () at (-0.7202999999999999,-4.47) {\tiny{26}};
\node[draw,circle,fill=black!12.0!white,inner sep=0,minimum size=4mm,line width=1.5] () at (0.35294699999999996,-5.1903) {\tiny{49}};
\node[draw,circle,fill=black!19.5!white,inner sep=0,minimum size=4mm,line width=0.5] () at (-0.35294699999999996,-5.1903) {\tiny{79}};
\node[draw,circle,fill=black!11.5!white,inner sep=0,minimum size=4mm,line width=1.5] () at (0.0,-5.543247) {\tiny{47}};
\node[draw,circle,fill=black!4.5!white,inner sep=0,minimum size=4mm,line width=1.5] () at (0.0,-5.1903) {\tiny{19}};
\node[draw,circle,fill=black!2.5!white,inner sep=0,minimum size=4mm,line width=1.5] () at (0.0,-4.47) {\tiny{11}};
\node[draw,circle,fill=black!0.75!white,inner sep=0,minimum size=4mm,line width=1.5] () at (0,-3) {\tiny{4}};